\PassOptionsToPackage{style=alphabetic}{bibtex}
\PassOptionsToPackage{final}{showkeys} 
\PassOptionsToPackage
{left=2.5cm, 
  right=2.5cm,
  top=1cm,
  bottom=1cm,
  footskip=18pt,
  includehead,
  includefoot,
  headsep=12pt}
{geometry}
\documentclass[reqno,11pt]{article}
\usepackage{macroGS}
\usepackage{tikz}
\usetikzlibrary{arrows.meta,positioning,fit,backgrounds}
\DeclareFontFamily{U}{rsfs}{}
\DeclareFontShape{U}{rsfs}{m}{n}{<-6>rsfs5<-8.5>rsfs7<->rsfs10}{}
\newcommand{\XO}[0]{\HtN\cap L^\infty(\Omega)}
\newcommand{\VO}[0]{\rmV(\Omega)}
\newcommand{\PVO}[0]{\rmV_{+\!+}(\Omega)}

\newcommand{\ZO}[0]{\mathrm{Test}(\Omega)}
\newcommand{\PZO}[0]{\mathrm{Test}_{+\!+}(\Omega)}
\newcommand{\WpO}[0]{H^2_{\!\surd}(\Omega)}%

\newcommand{\PXO}{\HtN\cap L^\infty_{+\!+}(\Omega)}

\newcommand{\duality}[2]{\la #1,#2\ra}
\newcommand{\DHtN}[0]{H^{-2}_{\sfN}(\Omega)}%
\newcommand{\DHtnrm}[1]{\|#1\|_{-2,\sfN}}%
\newcommand{\Htnrm}[1]{\|#1\|_{2,\sfN}}
\newcommand{\MDHtN}{\cM^{-2}_{\sfN,+}}
\newcommand{\MDHtnrm}[1]{\|#1\|_{\MDHtN(\overline\Omega)}}
\newcommand{\Cpt}[0]{\mathsf{cap}_{2,\sfN}}

\newcommand{\scalpt}[2]{(#1,#2)_{2}}
\newcommand{\tnrm}[1]{\|#1\|_{2}}
\newcommand{\pnrm}[1]{\|#1\|_{p}}
\newcommand{\Lap}[0]{\Delta}
\newcommand{\BLap}[0]{\Lap^2}

\newcommand{\HtN}[0]{H^2_{\sfN}(\Omega)}
\newcommand{\frAo}{\frA^r}

\title{Maximal monotonicity and contraction semigroup for the
  quantum drift-diffusion (Derrida--Lebowitz--Speer--Spohn) equation}
\begin{document}

\author{
Daniel Matthes
\thanks{Zentrum Mathematik, Technische Universit\"at M\"unchen,
  Garching, Germany; email:
  \textsf{matthes@ma.tum.de}.},
 Giuseppe Savar\'e\
 \thanks{Department of Decision Sciences and BIDSA, Bocconi University,
 Milano, Italy; email:
 \textsf{giuseppe.savare@unibocconi.it}.},
 Andr\'e Schlichting
 \thanks{Institut f\"ur Angewandte Analysis, Universit\"at Ulm,
  Ulm, Germany; email:
  \textsf{andre.schlichting@uni-ulm.de}.}
}

\date{\today}

\maketitle

\begin{abstract}
We study the quantum drift-diffusion, or
Derrida--Lebowitz--Speer--Spohn (DLSS), equation for a nonnegative
density $\varrho$ on a bounded convex domain with Neumann boundary
conditions, in the square-root variable $u=\sqrt\varrho$. We show that
the DLSS operator, defined and monotone on smooth strictly positive
functions, admits a \emph{unique} maximal monotone extension in
$L^2(\Omega)$, explicitly given by the minimal (defect-free) operator
plus the normal cone of the positivity constraint. The generated
semigroup, which contracts the Hellinger distance between the
densities, thus yields a canonical solution --- existing, unique, and
stable for every nonnegative $L^2$ initial datum and in every space
dimension --- independent of any approximation scheme: it is in fact
the unique contraction semigroup extending the classical evolutions
that emanate from smooth, uniformly positive data. The implicit
Euler scheme converges to it, and $\sqrt u\in L^2_{\rm loc}(H^2)$
along the flow. When the datum belongs to the domain of the operator,
the solution is strong and satisfies the equation pointwise, with no
reaction term created on the vacuum $\{u=0\}$. We characterize the
trajectories in several equivalent ways --- as B\'enilan integral
solutions and through one-sided weak formulations --- prove the
maximality of the operator also in the $H^2$--$H^{-2}$ duality and,
in dimension $d\le3$, identify the flow with the weak solutions in
the uniqueness class of Fischer. A second-order estimate of
independent interest underlies the construction: on a convex domain
with Neumann conditions the dissipation $\int_\Omega(\Lap
u)^2/u\,\d x$ is finite exactly when $\sqrt u\in H^2(\Omega)$, and it
then controls the full Hessian of $\sqrt u$, in every dimension.
\end{abstract}

{\small\noindent

}

{\small\tableofcontents}

\section{Introduction}

The quantum drift-diffusion, or Derrida--Lebowitz--Speer--Spohn (DLSS),
equation governs the evolution of a nonnegative density
$\varrho=\varrho(x,t)$ on a bounded convex domain $\Omega\subset\R^d$.
In divergence form it reads
\begin{equation}
	\label{eq:intro-rho-div}
	\partial_t\varrho=
	-2\,\nabla\!\cdot\!\Big(\varrho\,\nabla\frac{\Lap\sqrt\varrho}{\sqrt\varrho}\Big),
\end{equation}
and, expanding the fourth-order operator, equivalently
\begin{equation}
	\label{eq:intro-rho}
	\partial_t\varrho+\Lap^2\varrho
	-\sum_{i,j=1}^d\partial^2_{ij}\Big(\frac{\partial_i\varrho\,\partial_j\varrho}\varrho\Big)=0 ,
\end{equation}
always complemented with no-flux (Neumann) boundary conditions on the
convex domain $\Omega$ and with a nonnegative initial datum
$\varrho|_{t=0}=\varrho_0$. It was introduced in the analysis of
interface fluctuations \cite{DLSS91,DLSS91b,Bleher-Lebowitz-Speer1994}
and it is the simplest fourth-order quantum diffusion arising in
semiconductor modelling
\cite{Jungel16,Juengel-Pinnau2000,Bukal-Juengel-Matthes2013}.
The two names have distinct origins, and we use the abbreviation DLSS
throughout. In semiconductor modelling the quantum drift-diffusion system
carries in addition a drift term generated by the electric potential $V$ and
an isothermal pressure term,
\begin{equation}
	\label{eq:isothermal-intro}
	\partial_t\varrho=\nabla\!\cdot\!\Big(\varrho\,\nabla\Big(
	\theta \ln\varrho-\eps^2 \frac{\Lap\sqrt\varrho}{\sqrt\varrho}+V\Big)\Big),
\end{equation}
so that \eqref{eq:intro-rho-div} is the field-free, purely quantum case
$\theta=0,\, V\equiv0,\,\eps^2=2$ in which the pressure term is neglected; all the results below
concern that case. 
Its theory
is delicate: the equation is of fourth order, it obeys no maximum
principle, and the nonlinearity is singular  on the \emph{vacuum}
$\{\varrho=0\}$, which a strictly positive smooth datum may in general
reach in finite time.

\subsection{Lyapunov structure, gradient flow and a priori estimates}
\label{subsec:intro-rho}

Equation~\eqref{eq:intro-rho} carries a rich Lyapunov structure. The
mass is conserved,
\begin{equation}
	\label{eq:intro-mass}
	\int_\Omega\varrho_t\,\d x=\int_\Omega\varrho_0\,\d x ,
\end{equation}
while the Boltzmann entropy and the Fisher information
\begin{equation}
	\label{eq:intro-funct}
	\mathcal H(\varrho):=\int_\Omega\varrho\ln\varrho\,\d x ,
	\qquad
	\mathcal F(\varrho):=\int_\Omega|\nabla\sqrt\varrho|^2\,\d x ,
\end{equation}
are nonincreasing along the flow. The entropy dissipation is the
explicit identity
\begin{equation}
	\label{eq:intro-entropy-diss}
	\mathcal H(\varrho_T)+
	\int_S^T\!\!\int_\Omega\varrho_t\,\big|\rmD^2\ln\varrho_t\big|^2\,\d x\,\d t
	=\mathcal H(\varrho_S),\qquad 0\le S\le T ,
\end{equation}
at the origin of the algorithmic entropy method and of the sharp decay
estimates
\cite{Juengel-Matthes2006,Juengel-Violet2007,Dolbeault-Gentil-Juengel2006,%
Juengel-Toscani2003,Caceres-Carrillo-Toscani2005}. A deeper mechanism
underlies the decay of $\mathcal F$: 
\cite{Gianazza-Savare-Toscani09} showed that \eqref{eq:intro-rho} is the
gradient flow of the Fisher information $\mathcal F$ with respect to the
$L^2$-Wasserstein distance, in the sense of Otto's formal Riemannian
calculus \cite{Otto2001,Otto-Villani2000} and of the metric theory of
gradient flows \cite{JKO98,AGS08}.

Likewise, \eqref{eq:intro-rho} is also the gradient flow of the \emph{entropy}
$\mathcal H$ with respect to the \emph{diffusive transport} distance~\cite{MRSS23,MRS25}, a second-order analogue of the
Benamou--Brenier formula in which the continuity equation is driven by a
double divergence; the dissipation
$\int_\Omega\varrho\,|\rmD^2\ln\varrho|^2\,\d x$ in
\eqref{eq:intro-entropy-diss} corresponds to the metric
derivative, and already displays that second-order structure. 
Recently in~\cite{MSS25}, this structure is derived and generalized 
to nonlinear mobility, as an effective limit of chemical reaction systems, 
which adds another microscopic system as the possible origin of the equation.

Weak solutions and their long-time behaviour were constructed, under
various boundary conditions and by different approximation schemes, in
\cite{Juengel-Pinnau2000,Gualdani-Juengel-Toscani2006,Jungel-Matthes08};
each scheme distinguishes its own weak formulation. The one relevant to
us tests \eqref{eq:intro-rho} against a smooth function $\varphi$,
\begin{equation}
	\label{eq:intro-weak-rho}
	\int_\Omega\partial_t\varrho_t\,\varphi\,\d x
	+\int_\Omega\varrho_t\,\Lap^2\varphi\,\d x
	-\int_\Omega\sum_{i,j=1}^d
	\frac{\partial_i\varrho_t\,\partial_j\varrho_t}{\varrho_t}\,
	\partial^2_{ij}\varphi\,\d x=0 ,
\end{equation}
and corresponds, in the variable $u=\sqrt\varrho$, to the $b$-form
identity \eqref{w-equivalent2} which has also been exploited in the present paper.
Fischer \cite{Fischer13} proved uniqueness, for $d\le3$, in the
regularity class $\varrho^{1/2},\varrho^{1/4}\in L^2_{\rm loc}(H^2)$,
whose solutions are known to exist under periodic boundary conditions; a
structure-preserving discretization was proposed in \cite{MRSS23}.

Finally, the flow is a \emph{contraction} in a natural distance: if
$\varrho,\sigma$ are two solutions, then
\begin{equation}
	\label{eq:intro-hellinger}
	\int_\Omega\big|\sqrt{\varrho_t}-\sqrt{\sigma_t}\big|^2\,\d x
	\le
	\int_\Omega\big|\sqrt{\varrho_s}-\sqrt{\sigma_s}\big|^2\,\d x ,
	\qquad 0\le s\le t ,
\end{equation}
that is, the DLSS flow contracts the \emph{Hellinger} distance between
densities. As we shall see, \eqref{eq:intro-hellinger} is precisely the
$L^2$-contraction of the semigroup in the variable $u=\sqrt\varrho$.

\subsection{The square-root variable and the maximal monotone approach}
\label{subsec:intro-u}

Estimate \eqref{eq:intro-hellinger} already singles out $u:=\sqrt\varrho$
as the natural unknown: the Wasserstein picture makes $\varrho$ natural,
but the formal computations pointing to uniqueness
\cite{Jungel-Matthes08,Jungel16} are carried out at the level of the
square root. In this variable \eqref{eq:intro-rho} becomes
\begin{equation}
	\label{eq:intro-u}
	\partial_t u+\Lap^2u-\frac{(\Lap u)^2}u=0,
\end{equation}
the pointwise identity satisfied by our solutions
(Theorem~\ref{thm:generation}, equation~\eqref{eq:right}).

Two scalar quantities drive the analysis. The first is the mass, which
in the variable $u$ is the \emph{square} $L^2$-norm,
$\tnrm{u_t}^2=\int_\Omega\varrho_t\,\d x$, constant in time. The second
is the dissipation
\begin{equation}
	\label{eq:intro-Du}
	\sfD(u):=\int_\Omega\frac{(\Lap u)^2}u\,\d x ,
\end{equation}
which is controlled by the \emph{linear} functional $\int_\Omega u$:
along regular solutions of the flow
\begin{equation}
	\label{eq:intro-massD}
	\int_\Omega u_T\,\d x=\int_\Omega u_S\,\d x+\int_S^T\sfD(u_t)\,\d t,
\end{equation}
so that $t\mapsto\int_\Omega u_t$ is nondecreasing and
$\int_0^{+\infty}\sfD(u_t)\,\d t$ is finite (Proposition~\ref{prop:flow-regularity}). Since a finite dissipation
$\sfD(u)$ controls $\sqrt u$ in $H^2$ in every dimension
(Theorem~\ref{thm:Hessian}, \eqref{eq:twosided}; cf.\ the optimal
regularity of square roots \cite{Lions-Villani1995}), every trajectory
satisfies $\varrho^{1/4}=\sqrt u\in L^2_{\rm loc}(H^2)$ for arbitrary
$L^2_+$ data.

A finer and, in some respects, \emph{crucial} estimate concerns $u$
itself. When the entropy of the datum is finite, the entropy dissipation
\eqref{eq:intro-entropy-diss}, rewritten in the variable $u$, takes the
dimension-free form
\begin{equation}
	\label{eq:intro-H2u}
	\sfE(u_t)+4c_d\int_0^t\!\!\int_\Omega|\rmD^2u_r|^2\,\d x\,\d r\le\sfE(u_0),
	\qquad
	\sfE(u):=\int_\Omega u^2\ln(u^2)\,\d x ,
\end{equation}
with a constant $c_d>0$ in \emph{every} dimension, and yields
$u\in L^2_{\rm loc}(H^2)$
(Proposition~\ref{prop:entropy-regularization}). 

The contraction estimate \eqref{eq:intro-hellinger}  suggests that the map $u\mapsto\Lap^2u-(\Lap u)^2/u$ is
\emph{monotone} in $L^2(\Omega)$, and it is (Theorem~\ref{thm:DLSS-monotone}). One is thus
led to read \eqref{eq:intro-u} as the inclusion
\begin{equation}
	\label{eq:intro-inclusion}
	\partial_tu+\sfA u\ni0
\end{equation}
for a maximal monotone operator $\sfA$ in $L^2(\Omega)$, and to solve it
by the implicit Euler scheme. Three obstacles must be
overcome. First, the operator is defined and monotone a priori only on
smooth strictly positive functions (Section~\ref{subsec:positive}), and
might admit several maximal monotone extensions, hence several distinct
semigroups. Second, it is unclear how to extend it to the whole positive cone $L^2_+(\Omega)$ of
$L^2(\Omega)$. Third, in contrast with the Wasserstein formulation for
$\varrho$, at the level of $u$ the operator carries no variational
(minimization) structure, so that not even one step of the implicit
scheme is a minimization problem. 
The aim of this paper is to address and solve these main questions.

\subsection{Main results}
\label{subsec:intro-main}

We denote by $\sfA^\circ$ the minimal, defect-free DLSS operator
\begin{equation}
	\label{eq:intro-Amin}
	\sfA^\circ u:=\Lap^2u-\frac{(\Lap u)^2}u ,\quad
	\dom(\sfA^\circ)=
	\left\{u:
		\begin{aligned}
		&u\in H^2_+(\Omega)\text{ with } \partial_\nn u=0\text{ on }\partial\Omega, \\
	&
	\Lap^2u,\ \frac{(\Lap u)^2}u \in L^1(\Omega),\quad 
	\sfA^\circ u\in L^2(\Omega)
		\end{aligned}
	\right\}
\end{equation}
where the singular term $(\Lap u)^2/u$ just acts on $\{u>0\}$.
The definition of the domain of $\sfA^\circ$ is the minimal one needed to make sense of all the terms.

Our first result identifies
the maximal monotone extension of $\sfA^\circ$, denoted by $\sfA$, completely.
\begin{itemize}
	\item\emph{The unique maximal monotone extension.} 
	$\sfA^\circ$ has a unique maximal
	monotone extension in the positive cone of $L^2(\Omega)$ (Theorem~\ref{thm:main-maximal})
	which keeps the same domain as $\sfA^\circ$ and has
	the normal-cone structure
	\begin{equation}
		\label{eq:intro-structure}
		\sfA=\sfA^\circ+\partial I_{L^2_+} ,
		\quad
		\partial I_{L^2_+}(u):=\Big\{g\in L^2(\Omega):
		g\le 0,\ gu=0\Big\}
	\end{equation}
	the minimal operator plus the normal cone of the positivity
	constraint. The structure of $\sfA$ on states touching the vacuum, on which this
	description rests, is the content of Theorem~\ref{thm:main-equivalent} and 
	its Corollary \ref{prop:minimal-section}, which shows, in particular, that $\sfA^\circ u$ is the
	element of minimal norm in $\sfA u$.
	\item
	\emph{Canonicity.} $\sfA$ is the \emph{unique} maximal monotone
	operator in $L^2(\Omega)$ that extends the smooth core and whose domain
	consists of nonnegative functions (Corollary~\ref{cor:test-L2}): no
	choice is involved in passing from the classical operator on smooth
	positive functions to $\sfA$, and the generated semigroup is therefore
	unambiguous. Equivalently, at the level of the dynamics: $\sfS_t$
	is the unique continuous semigroup of contractions which coincides
	with the classical evolutions (for which only a local existence
	result is available), see Corollary~\ref{cor:unique-semigroup}.
	This is the result that
	makes the construction canonical
	rather than one construction among several.
	\item\emph{Strong solutions and the vacuum.} For every
	$u_0\in\dom(\sfA^\circ)$ (in particular for every smooth, strictly
	positive datum) the semigroup produces a Lipschitz in time strong solution
	with $u_t\in\dom(\sfA^\circ)$ and $\tnrm{\sfA^\circ u_t}$ nonincreasing,
	satisfying the pointwise equation \eqref{eq:intro-u}
	(Theorem~\ref{thm:generation}). \emph{No} reaction term is created
	when $u$ touches zero, so the solution is a genuine DLSS
	solution.
	\item\emph{The semigroup and its characterizations.} The contraction
	semigroup $\sfS_t$ extends to arbitrary data
	$u_0\in L^2_+(\Omega)=\overline{\dom(\sfA)}$
	(Lemma~\ref{le:domain-closure}). Its trajectories are exactly the
	curves with constant $L^2$-norm satisfying either the
	\emph{supersolution} inequality
	\begin{equation}
		\label{eq:intro-super}
		\partial_t\int_\Omega u_t\,v\,\d x+\fra(u_t,v)\ge0
		,\qquad
		\fra(u,v)=\int_\Omega\Big(\Lap u\,\Lap v-\frac{(\Lap u)^2}u\,v\Big)\d x ,
	\end{equation}
	tested against a suitable class of 
	smooth and uniformly positive test functions $v\in \PZO$
	(Theorem~\ref{def:weak-solution}) or, equivalently, the
	\emph{B\'enilan} integral inequalities
	\begin{equation}
		\label{eq:intro-benilan}
		\partial_t\int_\Omega u_t\,v\,\d x\ge\int_\Omega \sfA^\circ v\,u_t\,\d x
		\quad\text{for every }v\in \PZO
	\end{equation}
	(Definition~\ref{def:benilan-solution}). In every case
	$\sqrt u\in L^2_{\rm loc}(H^2)$
	(Proposition~\ref{prop:flow-regularity}) and, under finite entropy or
	for $d\le3$, $u\in L^2_{\rm loc}(H^2)$
	(Proposition~\ref{prop:entropy-regularization}).
	\item\emph{Second order calculus for square roots.} Independently of
	the evolution, the dissipation $\sfD$ of \eqref{eq:intro-Du} is finite
	\emph{exactly} when $\sqrt u$ lies in $\HtN$, and it then
	controls the full Hessian of $\sqrt u$, with a constant degenerating no
	faster than $d^{-2}$ (Theorem~\ref{thm:Hessian}). On a convex domain
	with Neumann conditions the Laplacian alone thus governs the whole
	second order structure of $\sqrt u$, in every dimension. The estimate
	is of independent interest and is used throughout the paper.
	\item\emph{Contraction and stability.} The semigroup is an
	$L^2$-contraction,
	$\tnrm{\sfS_tu_0-\sfS_t\hat u_0}\le\tnrm{u_0-\hat u_0}$; since
	$\tnrm{u-\hat u}^2=\int_\Omega|\sqrt\varrho-\sqrt{\hat\varrho}|^2\,\d x$,
	this is exactly the Hellinger contraction \eqref{eq:intro-hellinger},
	and $\sfS_t$ depends continuously on the datum.
\end{itemize}
It is worth noticing that the presence of 
the normal cone $\partial I_{L^2_+}$ in 
\eqref{eq:intro-structure} is a natural by-product of the maximality of the operator $\sfA$ and of the fact that the domain of $\sfA$ is contained in $L^2_+(\Omega)$.
In fact, every maximal monotone operator $\mathsf M$ in a Hilbert space $\mathbb H$ whose domain is contained in a closed convex set $\mathbb K\subset \mathbb H$ 
satisfies 
\begin{equation}
	\label{eq:intro-maximal}
	f\in \mathsf M u\quad	\Longrightarrow\quad
	f+\partial I_{\mathbb K}(u)\subset \mathsf M u.
\end{equation}
 In our case, the normal cone $\partial I_{L^2_+}$ is explicitly characterized by \eqref{eq:intro-structure}
 and the maximality of $\sfA$ also shows that no further term has to be added to $\sfA^\circ$.
\paragraph{What is new.}
Existence of suitable classes of weak solutions to \eqref{eq:intro-rho}, and their long-time
behaviour, are well established
\cite{Juengel-Pinnau2000,Gualdani-Juengel-Toscani2006,Jungel-Matthes08,%
Gianazza-Savare-Toscani09}, and uniqueness is known 
for periodic boundary conditions and $d\le3$ in the
regularity class of \cite{Fischer13}.
Singling out a good notion of weak solution is in fact a challenge in
itself: as recalled above, each approximation scheme distinguishes its own
formulation, and existence is granted only under an additional assumption
on the datum, typically the finiteness of the entropy $\mathcal
H(\varrho_0)$. What the present paper adds is of a
different nature: it identifies a \emph{single} operator, canonically
attached to the classical expression \eqref{eq:intro-u}, which governs
the equation for \emph{arbitrary} nonnegative $L^2$ data and in
\emph{every} dimension. Concretely, the new points are the monotonicity of
the DLSS operator in the $H^2$--$H^{-2}$ duality
(Theorem~\ref{thm:DLSS-monotone}) together with its maximality
(Theorems~\ref{thm:main-maximal} and~\ref{thm:frA-maximal}); the fact that
the maximal monotone extension is unique, so that the semigroup is not the
by-product of a particular approximation scheme
(Corollary~\ref{cor:test-L2}); the description of the operator on the
vacuum as a normal cone, with no reaction term created when $u$ touches
zero (Theorem~\ref{thm:vacuum-trace}); and the second order estimate of
Theorem~\ref{thm:Hessian}, which is of independent interest.

To these we add a new notion of weak solution, given by conditions
{\rm(W1)--(W3)} of Theorem~\ref{def:weak-solution}: it makes sense for
\emph{arbitrary} nonnegative data $u_0\in L^2_+(\Omega)$ and in every
dimension, it involves no assumption on the time derivative of the curve,
and for it we prove both existence and uniqueness
(Theorem~\ref{thm:generation}). The notion is robust: since the solution
map is a contraction, it depends $1$-Lipschitz continuously on the datum
and the class of solutions is closed under locally uniform $L^2$
convergence. Finally, the maximality of $\sfA$ turns the heuristics of
\eqref{eq:intro-inclusion} into a theorem: the implicit Euler scheme
converges through the exponential formula \eqref{eq:exponential}, thus
providing a new approximation scheme for \eqref{eq:intro-rho}, whose
single step is solved in Section~\ref{sec:maximality} by a variational
inequality (Theorem~\ref{thm:existence}) and not, as in the Wasserstein
setting, by a minimization problem.  Each step involves the minimal
operator $\sfA^\circ$ only, preserves nonnegativity with no recourse to a
maximum principle, and cannot enlarge the vacuum region
(Proposition~\ref{prop:Euler-vacuum}).  Along with it, the whole toolbox of
contraction semigroups becomes available: the regularity of strong
solutions, the right differentiability \eqref{eq:right}, the decay of
$\tnrm{\sfA^\circ u_t}$, and the inhomogeneous problem of
Remark~\ref{rem:inhomogeneous}.

A limitation should be stated here. Whether the semigroup $\sfS_t$
coincides with the weak solutions produced by the approximation schemes of
\cite{Juengel-Pinnau2000,Gualdani-Juengel-Toscani2006,Jungel-Matthes08} is
settled only for $d\le3$, through Theorem~\ref{thm:bform-fischer} and the
uniqueness theorem of \cite{Fischer13} (and, for the Neumann conditions
considered here, conditionally on Remark~\ref{rem:fischer-BC}). For
$d\ge4$ no uniqueness theory for weak solutions is available, and the
question remains open; what our results do give in every dimension is that
\emph{one} distinguished, canonically determined solution exists, depends
continuously on the datum, and is characterized by the equivalent
formulations listed above.

Around this core we develop several complements: 
commutation estimates for the Heat semigroup
(Corollary~\ref{cor:main-ineq}); several equivalent descriptions of the
graph of $\sfA$ (Theorem~\ref{thm:main-equivalent}); the fine structure
of the defect measure on the vacuum (Theorem~\ref{thm:vacuum-trace},
Proposition~\ref{prop:defect-variational}); and, for $d\le3$, the
identification of the semigroup with the solutions of the $b$-form
equation in Fischer's regularity class (Theorem~\ref{thm:bform-fischer}),
which links our results to the uniqueness theorem of \cite{Fischer13}.

\paragraph{Structure of the paper.}
Section~\ref{sec:prel} fixes notation and collects the preliminaries on
the Neumann Laplacian, the Heat semigroup, capacity, the second-order
calculus for nonnegative functions and the abstract theory of maximal
monotone operators. Section~\ref{sec:DLSS} constructs the DLSS operator
$\frA$ between $\HtN$ and $\DHtN$, analyses the defect measure, proves
monotonicity and introduces the $L^2$ realization $\sfA$.
Section~\ref{sec:evolution} develops the evolution equation: generation
of the semigroup, strong and B\'enilan solutions, the implicit Euler
scheme, space regularity, and
the distributional and weak characterizations (including the $d\le3$
theory). Section~\ref{sec:maximality} proves the maximality of $\sfA$ in
$L^2$ and of $\frA$ in the $H^2$--$H^{-2}$ duality. The appendices
collect the $L^1$ Neumann Laplacian, the proofs of the second-order
calculus, the link with Fischer's approach, a general existence result
for variational inequalities, and the auxiliary regularization
estimates.
Figure~\ref{fig:dependencies} summarizes how the results depend on one
another, from the tools of Section~\ref{sec:prel} up to the main theorems.

\begin{figure}[htbp]
\centering
\resizebox{\textwidth}{!}{%
\begin{tikzpicture}[
	every node/.style={font=\scriptsize},
	res/.style={draw, rounded corners=2pt, align=center, text width=2.35cm,
		minimum height=1.0cm, inner sep=2.5pt, fill=white},
	key/.style={res, very thick},
	dep/.style={-{Latex[length=1.5mm]}, rounded corners=4pt,
		shorten >=1pt, shorten <=1pt, gray!55!black},
]
\def\xa{1.6}\def\xb{4.7}\def\xc{7.8}\def\xd{10.9}\def\xe{14.0}
\def\ya{0}\def\yb{-2.9}\def\yc{-5.8}\def\yd{-8.7}\def\ye{-11.6}\def\yf{-14.5}
\def\hAB{-1.45}\def\hBC{-4.35}\def\hCD{-7.25}\def\hDE{-10.15}\def\hEF{-13.05}
\def\ga{3.15}\def\gb{6.25}\def\gc{9.35}\def\gd{12.45}\def\gR{16.1}\def\gRR{16.75}

\node[res] (A1) at (\xa,\ya) {Lem.~\ref{le:useful-approximation}\\ test algebra\\ $\ZO$};
\node[res] (A2) at (\xb,\ya) {Lem.~\ref{le:convex-estimate}\\ Jensen for $\sfH_t$};
\node[res] (A3) at (\xc,\ya) {Lem.~\ref{le:sqrt-calculus}\\ calculus for $\sqrt u$};
\node[res] (A4) at (\xd,\ya) {Thm.~\ref{thm:representation}\\ positive functionals};
\node[res] (A5) at (\xe,\ya) {Thm.~\ref{thm:Minty}, \ref{thm:BC-VI}\\ Minty; variational ineq.};
\node[res] (B1) at (\xb,\yb) {Prop.~\ref{prop:quartic}\\ quartic gradient bound};
\node[key] (B2) at (\xc,\yb) {Thm.~\ref{thm:Hessian}\\ $\sfD(u)<\infty\Leftrightarrow\sqrt u\in\HtN$};
\node[res] (B3) at (\xd,\yb) {Cor.~\ref{cor:main-ineq}\\ commutation with $\sfH_t$};
\node[res] (C1) at (\xa,\yc) {Lem.~\ref{le:obvious}\\ (A.1)--(A.3$''$)};
\node[res] (C2) at (\xb,\yc) {Prop.~\ref{prop:defect-variational}\\ defect measure};
\node[key] (C3) at (\xc,\yc) {Thm.~\ref{thm:DLSS-monotone}\\ $\frA$ is monotone};
\node[res] (C4) at (\xd,\yc) {Thm.~\ref{thm:vacuum-trace}\\ structure on the vacuum};
\node[res] (C5) at (\xe,\yc) {Prop.~\ref{le:DLSS-equivalence}\\ (A.2$'$)+(A.3$''$) characterize $\frA$};
\node[res] (D1) at (\xa,\yd) {Thm.~\ref{thm:existence}\\ regularized ineq.\ $\eps\Phi$};
\node[res] (D2) at (\xb,\yd) {Thm.~\ref{thm:eta}\\ limit $\eps\downarrow0$};
\node[key] (D3) at (\xc,\yd) {Thm.~\ref{thm:main-maximal}\\ $\sfA$ maximal monotone};
\node[res] (D4) at (\xd,\yd) {Thm.~\ref{thm:frA-maximal}\\ $\frA$ maximal monotone};
\node[key] (D5) at (\xe,\yd) {Cor.~\ref{cor:test-L2}\\ \emph{unique} extension of the core};
\node[key] (E1) at (\xb,\ye) {Thm.~\ref{thm:generation}\\ semigroup $\sfS_t$, strong solutions};
\node[res] (E2) at (\xc,\ye) {Prop.~\ref{prop:flow-regularity}\\ $\sqrt{u}\in L^2_{\rm loc}(\HtN)$};
\node[res] (E3) at (\xd,\ye) {Prop.~\ref{prop:entropy-regularization}\\ $u\in L^2_{\rm loc}(\HtN)$};
\node[key] (E4) at (\xe+0.35,\ye) {Thm.~\ref{thm:bform-flow}\\ $b$-form along the flow, $d\le3$};
\node[res] (F0) at (\xa,\yf) {Cor.~\ref{prop:minimal-section}\\ $\sfA=\sfA^\circ+\partial I_{L^2_+}$};
\node[res] (F1) at (\xb,\yf) {Thm.~\ref{thm:main-equivalent}\\ equivalent forms of $\sfA$};
\node[key] (F2) at (\xc,\yf) {Thm.~\ref{def:weak-solution}\\ weak formulation};
\node[key] (F3) at (\xd,\yf) {Thm.~\ref{thm:bform-fischer}\\ $b$-form, $d\le3$};

\draw[dep] (A1) -- (C1);
\draw[dep] (A2) -- (B2);
\draw[dep] (A2) -- (B3);
\draw[dep] (A3) -- (B1);
\draw[dep] (A3) -- (B2);
\draw[dep] (A3) -- (C4);
\draw[dep] (A4) -- (C5);
\draw[dep] (A5.south) -- (\xe,\hAB) -- (\gd-0.17,\hAB) -- (\gd-0.17,\hCD) -- (\xd,\hCD) -- (D4.north);
\draw[dep] (A5.east) -- (\gR,\ya) -- (\gR,\hCD-0.45) -- (\xa,\hCD-0.45) -- (D1.north);
\draw[dep] (B1) -- (B2);
\draw[dep] (B2) -- (C2);
\draw[dep] (B2.east) -- (\gc,\yb) -- (\gc,\hDE+0.55) -- (\xc,\hDE+0.55) -- (E2.north);
\draw[dep] (B3) -- (D5);
\draw[dep] (C1) -- (C2);
\draw[dep] (C2) -- (C3);
\draw[dep] (C3) -- (D3);
\draw[dep] (C3.south east) to[out=-45,in=150] (D4.north west);
\draw[dep] (C1.north east) -- (\ga-0.28,\hBC) -- (\xd,\hBC) -- (C4.north);
\draw[dep] (C2.north) -- (\xb,\hBC-0.4) -- (\xd-0.35,\hBC-0.4) -- (\xd-0.35,\yc+0.5);
\draw[dep] (C2.south west) -- (\ga,\yc-0.62) -- (\ga,\yf) -- (F1.west);
\draw[dep] (C2.south east) -- (\gb,\yc-0.62) -- (\gb,\hDE) -- (E3.west);
\draw[dep] (C4.east) -- (\gd+0.17,\yc) -- (\gd+0.17,\hEF+0.3) -- (\xb+0.35,\hEF+0.3) -- (\xb+0.35,\yf+0.5);
\draw[dep] (C5.west) -- (D2.east);
\draw[dep] (C5.east) -- (\gRR,\yc) -- (\gRR,\hEF-0.3) -- (\xb+1.6,\hEF-0.3) -- (F1.east);
\draw[dep] (D1) -- (D2);
\draw[dep] (D1.south) -- (\xa,\hDE+0.3) -- (\xd,\hDE+0.3) -- (D4.south);
\draw[dep] (D2.north east) to[out=35,in=145] (D4.north west);
\draw[dep] (D2) -- (D3);
\draw[dep] (D3.north east) to[out=35,in=145] (D5.north west);
\draw[dep] (D3) -- (E1);
\draw[dep] (E1) -- (E2);
\draw[dep] (E1.north east) to[out=35,in=145] (E3.north west);
\draw[dep] (E2) -- (E3);
\draw[dep] (E1) -- (F2);
\draw[dep] (E2) -- (F2);
\draw[dep] (E2) -- (F3);
\draw[dep] (E3) -- (F3);
\draw[dep] (E3) -- (E4);
\draw[dep] (E2.south east) to[out=-25,in=200] (E4.south west);
\draw[dep] (E4) -- (F3.east);
\draw[dep] (F1.south west) to[out=215,in=325] (F0.south east);
\end{tikzpicture}}
\caption{Principal dependencies among the results. Arrows point from a
statement to the results whose proof uses it; the thicker boxes are the
principal results. Tools from
Section~\ref{sec:prel} and the appendices are in the top row, the construction and the
monotonicity of $\frA$ in the middle, the maximality and the generated
semigroup below, and the equivalent characterizations in the last row.
Only the principal edges are drawn.}
\label{fig:dependencies}
\end{figure}
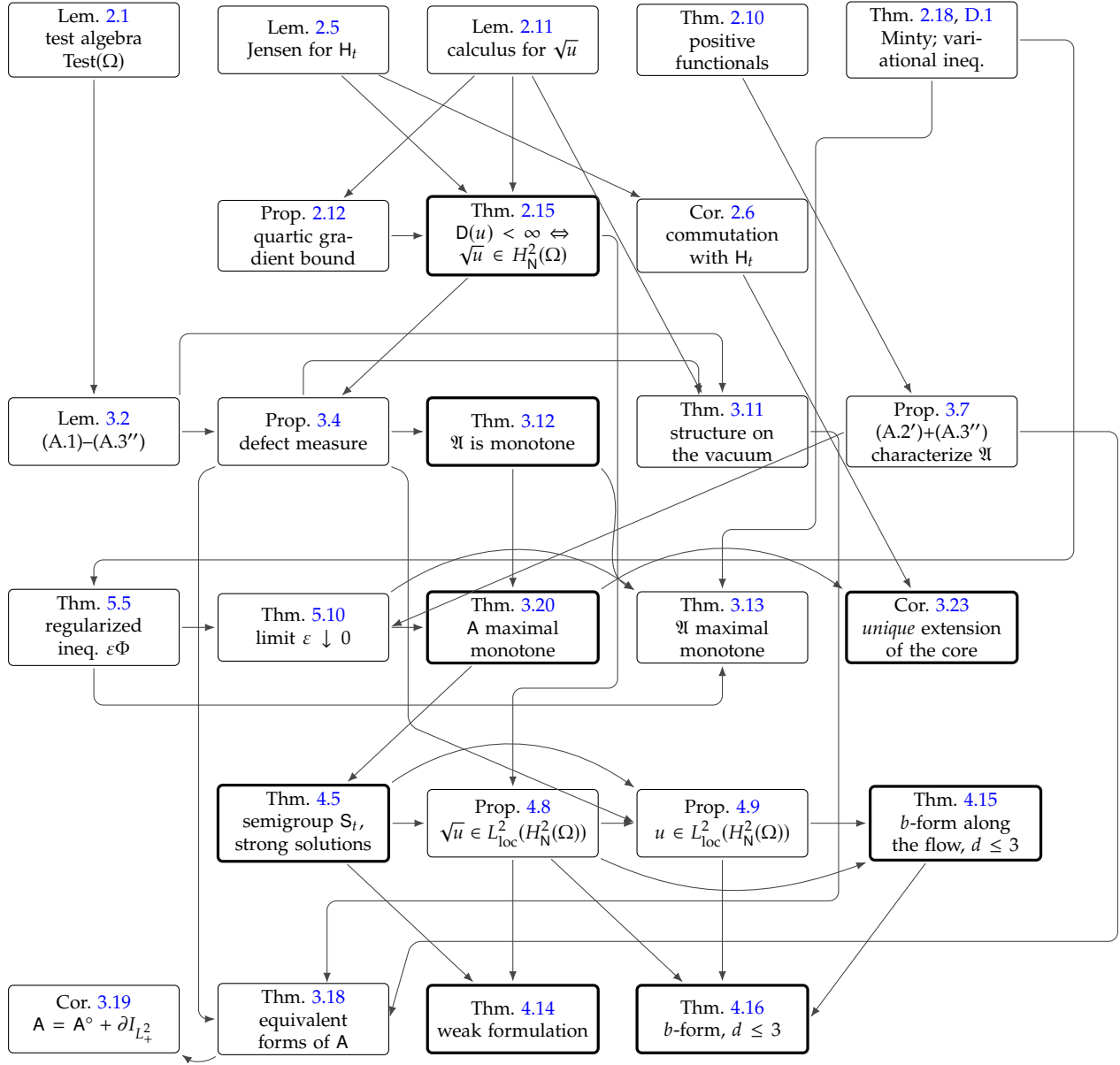
\clearpage
\bigskip

\centerline{\large\bfseries Main notation}

\smallskip
\halign{$#$\hfil\quad&#\hfil\cr
\Omega&open, bounded and convex subset of $\R^d$\cr
\Leb d&Lebesgue measure in $\R^d$\cr
\noalign{\smallskip}
\scalpt\cdot\cdot,\ \tnrm\cdot&scalar product and norm of $L^2(\Omega)$\cr
L^2_+(\Omega),\ L^\infty_+(\Omega)&cones of nonnegative functions in $L^2$ or $L^\infty$\cr
L^\infty_{++}(\Omega)&uniformly positive functions, with
positive (essential) infimum\cr
\HtN,\ \Htnrm{\cdot}&$H^2(\Omega)$ functions with the Neumann condition $\partial_\nn u=0$\cr
\WpO&nonnegative $u\in L^2(\Omega)$ with $\sqrt u\in\HtN$\cr
\DHtN&the dual space $(\HtN)'=H^{-2}_{\sfN}(\Omega)$\cr
\la\cdot,\cdot\ra&duality pairing between $\HtN$ and $\DHtN$, extending $\scalpt\cdot\cdot$\cr
\MDHtN(\overline\Omega)&positive finite measures belonging to $\DHtN$\cr
\ZO,\ \PZO&test algebra dense in $\HtN$, and its uniformly positive elements\cr
\noalign{\smallskip}
\Lap&Laplace operator with homogeneous Neumann conditions\cr
\nabla u,\ \rmD^2u&gradient and Hessian of $u$\cr
\partial_\nn&outer normal derivative on $\partial\Omega$\cr
\sfH_t&Heat (Markov) semigroup with Neumann conditions\cr
\sfR&resolvent $(\sfI-\Lap)^{-1}$ of the Neumann Laplacian\cr
\Cpt,\ \tilde u&$(2,2)$-capacity, and $\Cpt$-quasi-continuous representative of $u$\cr
\noalign{\smallskip}
\varrho,\ u=\sqrt\varrho,\ w=\sqrt u=\varrho^{1/4}&density and its roots\cr
\sfh(a,b)=b^2/a&perspective function of the DLSS operator\cr
\sfD(u)=\int_\Omega\sfh(u,\Lap u)\,\d x&dissipation (Fisher-type) functional\cr
\sfE&(relative) entropy functional\cr
\fra,\ \frb&the bilinear form and the $b$-form\cr
\frA&DLSS operator from $\HtN$ to $\DHtN$\cr
\sfA,\ \sfA^\circ&$L^2$ realization of $\frA$, and its minimal section\cr
\sfS_t&DLSS contraction semigroup in $L^2_+(\Omega)$\cr
}

\section{Preliminaries and second-order calculus for nonnegative functions}
\label{sec:prel}

Let $\Omega\subset \R^d$ be an open, bounded and convex domain.
We denote by $\scalpt \cdot\cdot$ the scalar product in $L^2(\Omega)$ with norm $\tnrm\cdot$.

\subsection{The Laplace operator with Neumann boundary conditions}
\label{subsec:Lap}
The Dirichlet bilinear  form in $H^1(\Omega)$ 
\begin{equation}
	\frd(u,v):=\int_\Omega \nabla u\cdot\nabla v\,\d x,\quad
	\dom(\frd):=H^1(\Omega),
\end{equation}
is associated with 
the Laplace operator
with homogeneous Neumann boundary conditions
$\Lap:\HtN\to L^2(\Omega)$ 
with domain
\begin{equation}
	\label{eq:domL}
	\HtN:=\Big\{u\in H^2(\Omega):
	\partial_\nn u=0\text{ on }\partial\Omega\Big\}
\end{equation} 
defined by
\begin{equation}
	\label{eq:Lap}
	u\in \HtN,\ -\Lap u=f\in L^2(\Omega)
	\quad\Leftrightarrow\quad
	u\in H^1(\Omega),\quad
	\frd(u,v)=\int_\Omega fv\,\d x
	\quad\text{for every }v\in H^1(\Omega).
\end{equation} 
Since $\Omega$ is convex, we have \cite[Section~3.2]{Grisvard85}
\begin{equation}
	\label{eq:H2bound}
	\int_\Omega (\Lap u)^2\,\d x\ge 
	\int_\Omega |\rmD^2 u|^2\,\d x\quad\text{for every }u\in \HtN.
\end{equation}
We can define an equivalent norm in 
the Hilbert space $\HtN$ 
by considering the operator
$$(\sfI-\Lap): \HtN\to L^2(\Omega),\quad 
u\mapsto u-\Lap u,$$
and setting
\begin{equation}
	\label{eq:H2norm-equivalent}
	\Htnrm u^2:=\tnrm{u-\Lap u}^2.
\end{equation}
As usual, we identify
$L^2(\Omega)$ with its dual via 
the duality induced by its scalar product.
Since $\HtN$ is dense in $L^2(\Omega)$,
we can extend the $L^2$ scalar product to the duality pairing
$\la \cdot,\cdot\ra$
between $\HtN$ and its dual $(\HtN)'$ which we will denote by $\DHtN$.
We have $L^2(\Omega)\subset \DHtN$ 
with dense embedding and a Hilbert triple
\begin{equation}
  \label{eq:2}
  \HtN\subset^{\rm ds}L^2(\Omega)\equiv \big(L^2(\Omega)\big)'
  \subset^{\rm ds}(\HtN)'=\DHtN.
\end{equation}
      Consider now the resolvent operator
$\sfR:L^2(\Omega)\to \HtN$ defined by 
$\sfR:=(\sfI-\Lap)^{-1}$: for every $f\in L^2(\Omega)$ 
$u=\sfR f$ solves the elliptic equation
\begin{equation}
	\label{eq:resolvent}
	u-\Lap u=f\quad\text{in }\Omega,\quad
	\partial_\nn u=0\quad\text{on }\partial\Omega,
\end{equation}
which corresponds to the variational formulation
\begin{equation}
	\label{eq:var-resolvent}
	\scalpt uv+\frd(u,v)=\scalpt fv\quad\text{for every }v\in H^1(\Omega).
\end{equation}
$\sfR$ can be extended to a 
positivity preserving contraction in 
every $L^p(\Omega)$, $p\in [1,+\infty)$, and it is easy to check that 
\begin{equation}
	\label{eq:dual-norm}
	\DHtnrm{f}:=
        \sup_{h\in \HtN,\ \Htnrm{h}\le 1} 
        \la{f},{h}\ra=
	\tnrm{\sfR f},\quad f\in L^2(\Omega).
\end{equation}
The above identity shows that $\sfR$ can be uniquely extended to a linear and continuous operator from $\DHtN$ to $L^2(\Omega)$
and \eqref{eq:dual-norm} can be extended to arbitrary elements $f\in \DHtN$
and provides an equivalent 
characterization of 
the dual norm of $\HtN$.
We also have
\begin{equation}
	\label{eq:duality1}
	\la f,w\ra=\int_\Omega \sfR f\,
	(w-\Lap w)\,\d x.
\end{equation}
Notice that the continuous and symmetric bilinear form 
\begin{equation}
	\label{eq:Dir2}
	\frd_2(u,v):=\int_\Omega \Lap u\,\Lap v\,\d x,\quad 
	u,v\in \HtN
\end{equation}
induces a bounded operator 
$\BLap:\HtN\to\DHtN$ satisfying
\begin{equation}
	\label{eq:BLap}
	\la \BLap u,v\ra=
	\int_\Omega \Lap u\,\Lap v\,\d x
	\quad \text{for every }
	u,v\in \HtN.
\end{equation}

\subsection{Heat flow and Markov semigroup}
\label{subsec:Heat}
$-\Lap$ generates a Markov analytic semigroup of contractions 
$(\sfH_t)_{t\ge 0}$ in $L^2(\Omega)$, corresponding to the solution
$u(t,\cdot):=\sfH_t u_0$ of the Heat equation
\begin{equation}
	\label{eq:Heat}
	\partial_t u-\Lap u=0\ \text{in }\Omega\times (0,+\infty),\quad
	\partial_\nn u=0\text{ on }\partial\Omega\times (0,+\infty),\quad 
	u(0,\cdot)=u_0.
\end{equation}
$(\sfH_t)_{t\ge 0}$ can be uniquely extended to
an analytic semigroup of contractions (still denoted by $\sfH_t$)  
in  every $L^p(\Omega)$, $p\in [1,+\infty)$. $\sfH_t$ 
is also ultracontractive and has the strong Feller property. In particular, 
for every $f\in L^1(\Omega)$ and $t>0$ 
$\sfH_t f$ has a (unique) Lipschitz continuous representative. 

Let us recall the most important estimates:
\begin{subequations}\label{eq:HeatEstimates}
\begin{align}
\label{eq:1infty}
	\|\sfH_t u\|_{L^p(\Omega)}&\le 
	\|u\|_{L^p(\Omega)},\quad &
	\|\sfH_t u\|_{L^\infty(\Omega)}
	&\le \frac {C_{\infty,1}}{t^{\frac{d}{2}}} \|u\|_{L^1(\Omega)}\\
	\lip(\sfH_t u)&\le 
	\lip(u)&
	\lip(\sfH_t u)&\le \frac{C_{\rm lip,\infty}}{\sqrt{t}}
	\|u\|_{L^\infty(\Omega)}\\
	\sfH_t(\Lap u)&=\Lap \sfH_t(u),\quad&
	\|\Lap \sfH_t u\|_{L^2(\Omega)}&\le 
	\frac Ct\|u\|_{L^2(\Omega)}\\
	\Htnrm{\sfH_t u}&\le 
	\Htnrm{u}&&
\end{align}
\end{subequations}
Combining the above estimates and using the semigroup property, we easily get
that
for
every $u\in L^1(\Omega)$ and every $t>0$, $\sfH_t u\in \HtN$ is a Lipschitz
function
with (essentially) bounded Laplacian;
if $u\ge 0$, $u\not\equiv 0$, then
$\min_{\overline\Omega}\sfH_t u>0$, by the strict positivity of the
Neumann heat kernel on the connected set $\Omega$ (see
e.g.~\cite{Davies89}).

We can also easily get the following approximation result
\begin{lemma}
  \label{le:useful-approximation}
  The space 
  \begin{equation}
    \label{eq:defZ}
    \ZO:=\Big\{u\in \HtN\cap \Lip(\overline\Omega):
    \Lap u\in L^\infty(\Omega)\Big\}
  \end{equation}
  is an algebra which is dense in $\HtN$ and separates the points of $\overline\Omega$.
  Moreover we have
  \begin{equation}
  	\label{eq:Z-acts}
  	u\in \HtN,\ 
  	w\in \ZO\quad\Rightarrow\quad
  	uw\in \HtN.
  \end{equation}
  In addition, $\ZO$ is also dense in $\rmC(\overline \Omega)$.
\end{lemma}
The proof is postponed to Appendix~\ref{app:sec2-technical}.

We conclude this section by introducing the natural realization in
$L^1(\Omega)$ of the Neumann Laplacian $\Lap$ generating the semigroup
$\sfH_t$ studied above.
\begin{definition}[Transposition Laplacian]
  \label{def:L1-Laplacian}
  For $u,g\in L^1(\Omega)$ we say that 
  \begin{equation}
    \label{eq:L1-Laplacian}
    \text{$u\in D_{L1}(\Lap)$ with
  $\Lap u=g$ if}
  \qquad \int_\Omega u\,\Lap\zeta\, \d x=\int_\Omega g\,\zeta\, \d x
    \qquad\forall\,\zeta\in\ZO.
  \end{equation}
  Since $\ZO$ is dense in $\rmC(\overline\Omega)$
  (Lemma \ref{le:useful-approximation}), $g$ is uniquely determined by $u$.
\end{definition}
\begin{proposition}[Graph closure and integration by parts]
  \label{prop:L1-closure}
  The graph of the $L^1$ realization of $\Lap$ defined by
  \eqref{eq:L1-Laplacian}
  coincides with the closure of the graph of $\Lap$ on
  $\ZO\times L^\infty(\Omega)$ (equivalently, on $\HtN\times L^2(\Omega)$) with respect to the product norm
  on
  $L^1(\Omega)\times L^1(\Omega)$.
  In particular for every $u\in D_{L1}(\Lap)$, with $\Lap u=g$,
  \begin{equation}
    \label{eq:L1-closure}
    \sfH_\tau u\to u\quad\text{and}\quad
    \Lap(\sfH_\tau u)=\sfH_\tau(\Lap u)\to \Lap u
    \quad\text{in }L^1(\Omega)\text{ as }\tau\downarrow0.
  \end{equation}
  If in addition $u\in W^{1,1}(\Omega)$, then
  \begin{equation}
    \label{eq:L1-IBP}
    \int_\Omega g\,\zeta\,\d x=-\int_\Omega \nabla u\cdot\nabla\zeta\,\d x
    \qquad\forall\,\zeta\in\ZO,
  \end{equation}
  and, by density, \eqref{eq:L1-IBP} extends to every
  $\zeta\in\Lip(\overline\Omega)$.

  Finally, if $u\in D_{L1}(\Lap)$ 
  \begin{equation}
	\label{eq:bilap-parts}
	\int_\Omega \Lap u\Lap v\,\d x=
	\int_\Omega u\Lap^2 v\,\d x
	\quad 
	\text{for every $v\in \ZO$ with $\Lap v\in \ZO$}.
  \end{equation}
\end{proposition}
The proofs, standard semigroup-approximation arguments, are postponed to
  Appendix~\ref{app:sec2-technical}.
\begin{remark}[Consistency]\label{rem:consistency}
  It is easy to check by regularization 
  that if $u,g$ belong to $L^2(\Omega)$ in 
  \eqref{eq:L1-Laplacian} then $u\in \HtN$ and $g=\Lap u$ 
  according to \eqref{eq:Lap}, the
  usual Neumann Laplacian of $u$. Thus
  Definition~\ref{def:L1-Laplacian} extends the classical Neumann
  Laplacian on $\HtN$, and already encodes $\partial_\nn u=0$ at the
  $L^1$ level, even though no boundary trace of $u$ is available in
  general.
\end{remark}
We finally record two consequences of the convexity properties of the
Heat semigroup $\sfH_t$ that will be used repeatedly in the sequel;
their proofs are collected in Appendix~\ref{app:sparse-proofs}.
\begin{lemma}
	\label{le:convex-estimate}
	Let $\Theta:\R^M\to [0,+\infty]$ be
	a convex l.s.c.~function
	and let $\uu=(u^1,\cdots,u^M)\in 
	L^1(\Omega;\R^M)$ such that 
	$\Theta\circ \uu\in L^1(\Omega)$.
	Setting $\uu_t=\sfH_t\uu=
	(\sfH_t u^1,\cdots,\sfH_t u^M)$ 
	we have 
	\begin{equation}
		\label{eq:pointwise-convex-estimate}
		\sfH_t(\Theta\circ\uu)\ge 
		\Theta(\sfH_t \uu)\quad \text{a.e.~in }\Omega. 
	\end{equation}
\end{lemma}
The DLSS operator involves the convex and lower semicontinuous
integrand $\sfh:[0,+\infty)\times\R\to[0,+\infty]$,
\begin{equation}
	\label{eq:defh}
	\sfh(a,b):=\begin{cases}
		\dfrac{b^2}a,&a>0,\\
		0,&a=0,\ b=0,\\
		+\infty,&a=0,\ b\neq0,
	\end{cases}
\end{equation}
i.e.~the lower semicontinuous \emph{perspective function} of
$s\mapsto s^2$, $\sfh(a,b)=a\,(b/a)^2$ \cite[Section IV.2.2]{HUL93}.
\begin{corollary}
\label{cor:main-ineq}
	For every $u\in D_{L1}(\Lap)$ nonnegative
	and every $w
	\in L^\infty_+(\Omega)$ 
	we have 
	\begin{equation}
		\label{eq:sfh-crucial}
		\int_{\Omega}\sfh(u,\Lap u)
		\,\sfH_t w\,\d x
		\ge \int_\Omega 
		\sfh(\sfH_t u,\Lap \sfH_t u)
		 \, w\,\d x.
	\end{equation}
\end{corollary}
\begin{corollary}
\label{cor:main-ineq2}
	For every nonnegative $u\in \HtN$ we have
	\begin{equation}
	\label{eq:fra-vs-S}
		\int_\Omega \Lap u\,\Lap(u-\sfH_t u)\,\d x-
		\int_{\{ u>0 \}}\frac{(\Lap u)^2}u
		(u-\sfH_t u)\,\d x\ge 0.
	\end{equation}
\end{corollary}

\subsection{\texorpdfstring{$(2,2)$}{(2,2)}-capacity}

We refer to \cite{Kazumi-Shigekawa92}
(see also \cite{Fukushima-Kaneko85,Kaneko06,Fukushima93,Hoh-Jacob04}).
The space $\HtN$ 
coincides with $\FF_{2,2}$ 
of \cite{Kazumi-Shigekawa92} and 
satisfies the regularity assumptions (A.1,A.2,A.3) of 
\cite{Kazumi-Shigekawa92}: in fact
Lemma \ref{le:useful-approximation} shows that 
$\ZO\subset \HtN\cap \rmC(\overline\Omega)$ is dense
in $\HtN$, 
the algebra $\ZO$
separates the points of $\overline\Omega$ 
and $\overline\Omega$ is compact.

The $\Cpt$-capacity of a set $A\subset \overline\Omega$ is then defined as
\begin{equation}
	\Cpt(A):=
	\inf \Big\{\Htnrm{u}^2:
	u\in \HtN,\ u\ge 1\text{ a.e.~on some
	(relatively) open set containing $A$}\Big\}.
\end{equation}
$\Cpt$ is a Choquet capacity
in $\overline\Omega$, in particular
it is outer and inner regular and stable 
under increasing limits of sets and decreasing limits of closed (thus compact) sets
in $\overline\Omega$.

We say that 
a statement holds $\Cpt\text{-q.e.}$ ($\Cpt$-quasi everywhere) if it holds except on a set of null $\Cpt$-capacity.
A function $u$ is called $\Cpt$-quasi-continuous if for every $\eps>0$ there exists a (relatively) open 
	set $G_\eps\subset \overline\Omega$ with $\Cpt(G_\eps)<\eps$ 
	such that $u\restr{\overline\Omega\setminus G_\eps}$ is continuous. 
	$\Cpt$-quasi-uniform convergence can be defined in a similar way.
\begin{remark}
	\label{rem:continuous-representative}
	In dimension $d\le 3$, 
	every function in $\HtN$ 
	has a continuous representative in $\rmC(\overline\Omega)$ 
	so that $\Cpt(A)=0$ 
	implies $A=\emptyset$
	and the theory 
	we are recalling below is much simpler.
\end{remark}
If $u\in \HtN$ then $u$ has a 
Lebesgue representative 
$\tilde u$ which is $\Cpt$-quasi continuous. 
$\tilde u$ admits an interesting representation in terms of the semigroup $\sfH_t$, since 
it is possible to prove that
\begin{equation}
	\label{eq:Cp-representation}
	\tilde u(x)=\lim_{t\down0}u_t(x)
	\text{ for $\Cpt$-q.e.~$x\in \overline\Omega$},\quad
	u_t:=\sfH_t u\in \Lip(\overline\Omega).
\end{equation}
Consider now a finite positive measure
$\mu\in \cM_+(\overline\Omega)$. 
We say that $\mu$ is in $\MDHtN(\overline\Omega)$ if 
\begin{equation}
	\label{eq:dual-measure}
	\MDHtnrm{\mu}:=
	\sup \Big\{\int_{\overline\Omega} w\,\d\mu:
	w\in \ZO,\
	\Htnrm{w}\le 1\Big\}<+\infty.
\end{equation}
Since $\ZO$
is dense in $\HtN$, 
every $\mu\in \MDHtN(\overline\Omega)$ 
induces a unique element $\ell_\mu \in 
\DHtN$ which satisfies 
the positivity property
\begin{equation}
	\label{eq:positivity}
	\la\ell_\mu,w\ra\ge0\quad\text{for every }w\in \HtN,\ w\ge 0\text{ a.e.}
\end{equation}
\begin{remark}
	\label{rem:L1-H-2}
This construction applies in particular to 
a function $f\in L^1_+(\Omega)$, inducing a finite measure $\mu=f\mathscr L^d$. 
According to \eqref{eq:dual-measure}, we simply say that $f\in \DHtN$ 
if 
there exists a constant $C\ge0$ such that 
\begin{equation}
	\label{eq:dual-measureL1}
	\int_\Omega fw\,\d x\le C\Htnrm{w}\quad \text{for every }
	w\in \ZO.
\end{equation}
In this case $f\mathscr L^d\in \MDHtN(\overline\Omega)$.
\end{remark}
Conversely, we have the following representation result
for positive functionals $\ell
\in \DHtN$, thus satisfying
\begin{equation}
  \label{eq:3}
  \la \ell,w\ra\ge 0\quad\text{for every }w\in \HtN,\ w\ge 0\text{
    a.e. in }\Omega.
\end{equation}
\begin{theorem}[Representation of positive functionals
  in $\DHtN$]
  \label{thm:representation}
  Every positive element $\ell
  \in \DHtN$ satisfying \eqref{eq:3}
  can be represented by a unique measure
$\mu\in \MDHtN(\overline\Omega)$
so that 
\begin{equation}
	\label{eq:representation}
	\ell=\ell_\mu,\qquad
        \la \ell,w\ra=
	\int_{\overline\Omega} w\,\d\mu\quad\text{for every }
	w\in \HtN\cap\rmC(\overline\Omega).
\end{equation}
Moreover, for every 
$\mu\in \MDHtN(\overline\Omega)$ we have
\begin{equation}
	\label{eq:measure-capacity}
	\mu(B)\le \MDHtnrm{\mu}\Cpt(B)^{1/2}\quad\text{for every Borel set }B\subset \overline\Omega,
\end{equation}
so that 
\begin{equation}
	\label{eq:Cp-null-negligible}
	\mu(B)=0\quad\text{for every Borel set with }
	\Cpt(B)=0.
\end{equation}
Furthermore, if 
$\tilde w$ is a $\Cpt$-quasi-continuous representative of $w\in \HtN$ we have
\begin{equation}
	\label{eq:wow}
	\tilde w\in L^1(\overline\Omega;\mu),\quad
	\la \ell_\mu,w\ra=
	\int_{\overline\Omega} \tilde w\,\d\mu,\quad
	\|\tilde w\|_{L^1(\overline\Omega;\mu)}
	\le 
	\Htnrm{w}\MDHtnrm{\mu}.
\end{equation}
\end{theorem}
In the following we will identify $\ell_\mu$ with $\mu$ and we write
$\la \mu,w\ra$ instead of $\la \ell_\mu,w\ra$.

\subsection{Second order quantities associated with nonnegative functions}
\label{subsec:second-order}

Both the definition and the study of the DLSS operator involve
the convex and lower semicontinuous integrand $\sfh$ of
\eqref{eq:defh} and the associated
functional
\begin{equation}
	\label{eq:defD}
	\sfD(u):={\int_\Omega \sfh(u,\Lap u)\,\d x} ,
\end{equation}
which is well defined (possibly $+\infty$) for every nonnegative $u\in
D_{L1}(\Lap)$, the domain of the $L^1$ realization of the Neumann
Laplacian of Definition~\ref{def:L1-Laplacian}.
We set $\sfD(u)=+\infty$ if $u\notin D_{L1}(\Lap)$.
{Equivalently,
\begin{equation}
	\label{eq:defD-explicit}
	\sfD(u)=\int_{\{u>0\}}\frac{(\Lap u)^2}u\,\d x
	\quad\text{if }\Lap u=0\ \text{$\Leb d$-a.e.~on }\{u=0\},
	\qquad
	\sfD(u)=+\infty\quad\text{otherwise,}
\end{equation}
so that the finiteness of $\sfD(u)$ also encodes the vanishing of
$\Lap u$ on the vacuum region $\{u=0\}$.
This constraint is quite natural: we will see in Lemma~\ref{le:sqrt-calculus} below that
$\Lap u$ vanishes a.e.~on $\{u=0\}$ whenever $\sqrt u\in\HtN$.}
Recall that nonnegative
functions of $\HtN$ belong to $D_{L1}(\Lap)$ with the same Laplacian
(see the Consistency Remark~\ref{rem:consistency} in Section~\ref{subsec:Heat}).

\medskip
In this subsection we collect the basic relations between $\sfD(u)$ and
the second order quantities associated with $w:=\sqrt u$, under the
minimal regularity condition $\sqrt u\in \HtN$; we will also introduce a
natural lower semicontinuous extension of $\sfD$ to nonnegative
functions in $L^2(\Omega)$ which need not belong to $H^2(\Omega)$.
All the proofs are collected in Appendix~\ref{app:second-order}.

We introduce the class
\begin{equation}
	\label{eq:defW}
	\WpO:=\Big\{u\in L^2(\Omega):\ u\ge0\ \text{a.e.~in }\Omega,\
	\sqrt u\in \HtN\Big\}.
\end{equation}
Notice that we do \emph{not} require $u\in H^2(\Omega)$: by the next
lemma, however, every $u\in \WpO$ possesses a Neumann Laplacian in the
$L^1$ sense of Definition~\ref{def:L1-Laplacian}.

\begin{lemma}[Calculus for square roots]
	\label{le:sqrt-calculus}
	Let $u\in \WpO$ and $w:=\sqrt u\in\HtN$. Then:
	\begin{enumerate}[\rm (a)]
	\item $\nabla w=0$ and $\rmD^2 w=0$ $\Leb d$-a.e.~on $\{w=0\}$;
	\item $u\in W^{1,1}(\Omega)$ with $\nabla u=2w\nabla w$;
	\item $u\in D_{L1}(\Lap)$ with
		\begin{equation}
			\label{eq:chain-Lap}
			\Lap u=2w\Lap w+2|\nabla w|^2\in L^1(\Omega);
		\end{equation}
		in particular $\Lap u=0$ $\Leb d$-a.e.~on $\{u=0\}$;
	\item $\Leb d$-a.e.~on $\{u>0\}$ we have
		\begin{equation}
			\label{eq:pointwise-sqrt}
			\frac{(\Lap u)^2}{u}=
			4\Big(\Lap w+\frac{|\nabla w|^2}{w}\Big)^{\!2} .
		\end{equation}
	\end{enumerate}
\end{lemma}
By Lemma~\ref{le:sqrt-calculus}(c) we have $\WpO\subset D_{L1}(\Lap)$
{and every $u\in\WpO$ complies with the vacuum constraint in
\eqref{eq:defD-explicit}}, so that $\sfD$ is well defined on $\WpO$;
moreover, combining (a), (c) and (d),
\begin{equation}
	\label{eq:defD-w}
	\sfD(u)=4\int_{\{w>0\}}
	\Big(\Lap w+\frac{|\nabla w|^2}{w}\Big)^{\!2}\,\d x
	\quad\text{for every }u\in\WpO,\ w=\sqrt u.
\end{equation}

The next result shows that the quartic gradient term hidden in
\eqref{eq:defD-w} is always controlled by the Hessian of $w$; in
particular $\sfD$ is \emph{finite} on the whole class $\WpO$.
\begin{proposition}[Quartic gradient bound and upper bound for $\sfD$]
	\label{prop:quartic}
	For every $u\in\WpO$, $w=\sqrt u$, we have
	\begin{equation}
		\label{eq:quartic}
		\int_{\{w>0\}}\frac{|\nabla w|^4}{w^2}\,\d x
		\le\Big(\tnrm{\Lap w}+2\tnrm{\rmD^2 w}\Big)^{\!2}
		\le 9\int_\Omega(\Lap w)^2\,\d x
	\end{equation}
	and therefore
	\begin{equation}
		\label{eq:D-upper}
		\sfD(u)\le
		8\int_\Omega(\Lap w)^2\,\d x
		+8\int_{\{w>0\}}\frac{|\nabla w|^4}{w^2}\,\d x
		\le
		80\int_\Omega(\Lap w)^2\,\d x<+\infty.
	\end{equation}
\end{proposition}
The first inequality in \eqref{eq:quartic} holds on every bounded
Lipschitz domain; only the second one uses the convexity of $\Omega$,
through \eqref{eq:H2bound} in the form
$\tnrm{\rmD^2w}\le\tnrm{\Lap w}$. In this way all the upper bounds are
expressed by the sole quantity $\int_\Omega(\Lap w)^2\,\d x$,
symmetrically to \eqref{eq:defD}, and the constants are not optimized.
Let us also remark that the proof of \eqref{eq:quartic} given in
Appendix~\ref{app:second-order} only uses the fact that $w$ is a
nonnegative element of $\HtN$, not that it is the square root of $u$:
the quartic gradient bound
\begin{equation}
	\label{eq:quartic-general}
	\int_{\{w>0\}}\frac{|\nabla w|^4}{w^2}\,\d x
	\le\Big(\tnrm{\Lap w}+2\tnrm{\rmD^2 w}\Big)^{\!2}
	\qquad\text{for every }w\in\HtN,\ w\ge0,
\end{equation}
will also be applied with $w:=u$ itself.

The converse lower bound of $\sfD(u)$ in terms of the Hessian of
$\sqrt u$ is deeper, since only the \emph{Laplacian} of $u$ appears in
\eqref{eq:defD}: it relies on the convexity of $\Omega$ through
\eqref{eq:H2bound} and on a dimensional algebraic (sum of squares)
inequality. It is best stated for the following relaxation of $\sfD$,
which is well defined for every nonnegative $u\in L^2(\Omega)$ and will
also play an important role in the a priori estimates of the next
sections.
\begin{definition}[Relaxed functional]
	\label{def:relaxedD}
	For every nonnegative $u\in L^2(\Omega)$ we set
	\begin{equation}
		\label{eq:defDstar}
		\sfG_\tau(u):=\int_\Omega
		\frac{(\Lap \sfH_\tau u)^2}{\sfH_\tau u}\,\d x\quad
		(\tau>0),
		\qquad
		\sfD^*(u):= \sup_{\tau>0}\sfG_\tau(u) = \lim_{\tau\to 0^+} \sfG_\tau(u) ,
	\end{equation}
	with the convention $\sfG_\tau(0):=0$.
\end{definition}
Definition~\ref{def:relaxedD} is meaningful: if $u\ge0$,
$u\not\equiv0$, then for every $\tau>0$ the function $\sfH_\tau u$ is
Lipschitz with (essentially) bounded Laplacian (Section
\ref{subsec:Heat}) and admits a continuous representative with
$\min_{\overline\Omega}\sfH_\tau u>0$, by the strict positivity of the
Neumann heat kernel on the connected set $\Omega$ (see
e.g.~\cite{Davies89}); hence $\sfG_\tau(u)<+\infty$.
Also, by writing the integrand as $\sfh(\sfH_\tau u, \Lap \sfH_\tau u)$ with $\sfh$ defined in~\eqref{eq:defh}, the joint convexity and one-homogeneity of $\sfh$ imply that $\tau \mapsto \sfG_\tau(u)$ is monotone decreasing (by Jensen's inequality for the Markov kernel of $\sfH_{\tau-s}$, cf.~Lemma~\ref{le:convex-estimate}), so that $\sup_{\tau>0}\sfG_\tau(u) = \lim_{\tau\to 0^+} \sfG_\tau(u)$.
\begin{lemma}[Lower semicontinuity properties of $\sfD^*$]
	\label{le:relaxedD}
	Each functional $\sfG_\tau$, and therefore $\sfD^*$, is
		convex and sequentially lower semicontinuous on $\big\{u\in
		L^2(\Omega):u\ge0\big\}$ with respect to the weak
		convergence of~$L^2(\Omega)$.
\end{lemma}
\begin{theorem}[Hessian and quartic bounds]
	\label{thm:Hessian}
	Let $u\in L^2(\Omega)$, $u\ge0$.
	\begin{enumerate}
		\item  
		$\sfD^*(u)=\sfD(u)$
		\item $\sfD(u)<+\infty$
		if and only if $w:=\sqrt u\in\HtN$, that is $\dom(\sfD) = \WpO$.
		In this case
		\begin{equation}
			\label{eq:Hessian}
			\int_\Omega|\rmD^2 w|^2\,\d x\le\frac{d^2}4\,\sfD^*(u),
			\qquad
			\int_{\{w>0\}}\frac{|\nabla w|^4}{w^2}\,\d x\le
			\Big(\frac{4d^2}9+\frac16\Big)\,\sfD^*(u).
		\end{equation}
		\item In particular
	\begin{equation}
		\label{eq:twosided}
		\frac4{d^2}\int_\Omega|\rmD^2\sqrt u|^2\,\d x
		\le \sfD(u)\le
		80\int_\Omega(\Lap\sqrt u)^2\,\d x .
	\end{equation}
	\end{enumerate}
\end{theorem}
For $u\in\WpO$, when there is no risk of ambiguity, we will simply write
\begin{equation}
	\label{eq:evoc-notation}
	\frac{(\Lap u)^2}{u}(x):=
	\sfh(u(x),\Lap u(x))=
	\begin{cases}
		\frac{(\Lap u(x))^2}{u(x)}&\text{if }u(x)>0,\\
		0&\text{if }u(x)=0.
	\end{cases}
\end{equation}

We postpone the proof of Lemma \ref{le:relaxedD} and Theorem~\ref{thm:Hessian} to	
Appendix~\ref{app:second-order}, where we also provide a more detailed
discussion of the related algebraic inequalities.
\begin{remark}[Dimension one]
	\label{rem:d1}
	For $d=1$ the constant $4/d^2=4$ in \eqref{eq:Hessian} and
	\eqref{eq:twosided} is obtained by a two-line argument, and the
	inequality refines to the identity
	$\sfD(u)=4\int_\Omega (w'')^2\,\d x+\frac{20}3
	\int_{\{w>0\}}\frac{(w')^4}{w^2}\,\d x$
	for positive $u$ (see the end of the proof of
	Proposition~\ref{prop:positive-case} in
	Appendix~\ref{app:second-order}).
\end{remark}
\begin{remark}[Bibliographical remarks]
	\label{rem:secondorder-refs}
	On general Lipschitz domains, identities and equivalences between
	the second order functionals
	$\int\big(2|\rmD^2u|^2+(\Lap u)^2\big)/u\,\d x$ and the
	corresponding quantities of $\sqrt u$ (also with respect to
	weighted measures) have been established in
	\cite[Section 3]{Gianazza-Savare-Toscani09}, see in particular
	Theorem 3.2 and Corollaries 3.1, 3.2 there: in those results the
	\emph{full Hessian} of $u$ enters the functional.
	The content of Theorem~\ref{thm:Hessian} is that on a
	\emph{convex} domain with Neumann boundary condition the sole
	Laplacian --- i.e.~the functional \eqref{eq:defD} which rules the
	DLSS operator --- controls the full Hessian of $\sqrt u$ in every
	dimension, with a constant degenerating not faster than $d^{-2}$.
	Entropy--dissipation inequalities of this type on the torus are at
	the core of the existence theory of \cite{Jungel-Matthes08} (see
	Lemma 2.2 there, together with the formulation (1.4) of the DLSS
	equation in terms of $\sqrt u$), and are obtained by the
	algorithmic entropy construction method; see also~\cite{MRSS23}.
	Finally, the class $\WpO$ essentially coincides with the
	regularity class in which Fischer~\cite{Fischer13} proved his
	uniqueness results for the DLSS equation: this connection will
	play a crucial role in the characterization of the DLSS
	semigroup.
\end{remark}
\subsection{Maximal monotone operators between a Hilbert space
and its dual}
\label{subsec:monotone-VV}

Let $\W$ be a real Hilbert space with scalar product
$(\cdot,\cdot)_\W$ and let $\W'$ be its dual space, with duality
pairing $\la\cdot,\cdot\ra$. We do not necessarily identify $\W$ with
$\W'$; the Riesz isomorphism (or duality map) $\rmJ_\W:\W\to\W'$ is
defined by
\begin{equation}
	\label{eq:duality-map}
	\la \rmJ_\W v,w\ra:=(v,w)_\W
	\quad\text{for every }v,w\in\W .
\end{equation}
A multivalued operator $\frM:\W\rightrightarrows\W'$ can be
identified with its graph $\graph(\frM)\subset\W\times\W'$, writing
$f\in\frM v$ if and only if $(v,f)\in\graph(\frM)$.
$\frM$ is
\begin{itemize}
	\item \emph{proper} if its domain
	$\dom(\frM):=\{v\in\W:\frM v\neq\emptyset\}$ is not empty,
	\item
	\emph{monotone} if $\la f-g,v-w\ra\ge0$ for every
	$(v,f),(w,g)\in\graph(\frM)$, 
	\item \emph{maximal monotone} if,
	moreover, its graph is maximal in the class of monotone subsets of
	$\W\times\W'$ ordered by inclusion.
\end{itemize}
Maximal monotone operators are proper, and maximality can be
equivalently characterized in terms of the surjectivity of the
perturbed operators $\lambda\rmJ_\W+\frM$, $\lambda>0$.
\begin{theorem}[Minty's characterization]
	\label{thm:Minty}
	Let $\frM:\W\rightrightarrows\W'$ be a monotone operator. The
	following properties are equivalent:
	\begin{enumerate}[\rm (a)]
	\item $\frM$ is maximal monotone;
	\item $\rmJ_\W+\frM$ is surjective: for every $f\in\W'$
		there exists $v\in\dom(\frM)$ such that
		$\rmJ_\W v+\frM v\ni f$;
	\item $\lambda\rmJ_\W+\frM$ is surjective for every
		$\lambda>0$.
	\end{enumerate}
	In this case, for every $\lambda>0$ the solution $v_\lambda\in
	\dom(\frM)$ of $\lambda \rmJ_\W v_\lambda+\frM v_\lambda\ni f$
	is unique and, if $v_\lambda,\hat v_\lambda$ correspond to data
	$f,\hat f\in\W'$, then
	\begin{equation}
		\label{eq:resolvent-estimate-W}
		\|v_\lambda-\hat v_\lambda\|_\W\le
		\frac1\lambda\,\|f-\hat f\|_{\W'} .
	\end{equation}
\end{theorem}
When $\W$ is identified with $\W'$ (so that $\rmJ_\W=\rmI$),
Theorem~\ref{thm:Minty} is the classical result of Minty~\cite{Minty62} (see also \cite[Proposition 2.2]{Brezis73}, and
\cite[Section 2.1]{Barbu10} for reflexive Banach spaces); the general
case follows by applying the classical one to the monotone operator
$\rmJ_\W^{-1}\circ\frM:\W\rightrightarrows\W$, and
\eqref{eq:resolvent-estimate-W} by pairing the difference of two
resolvent relations with $v_\lambda-\hat v_\lambda$.

\medskip
In this paper we will consider the dualities generated by the
specific Hilbert triple \eqref{eq:2}: from now on we set
$\H:=L^2(\Omega)$, identified with its dual $\H'$, and we endow
$\HtN$ with the equivalent Hilbertian norm
\eqref{eq:H2norm-equivalent}, so that \eqref{eq:2} reads
$\HtN\subset^{\rm ds}\H\equiv\H'\subset^{\rm ds}\DHtN$.
Only the intermediate space is identified with its dual,
so that $\rmJ_\H=\rmI$, whereas the duality map of
$\HtN$, which we will simply denote by $\rmJ$, takes the explicit
form
\begin{equation}
	\label{eq:J-explicit}
	\rmJ u=u-2\Lap u+\BLap u,
	\qquad
	\la \rmJ u,v\ra=\int_\Omega(u-\Lap u)(v-\Lap v)\,\d x=
	\scalpt uv+2\frd(u,v)+\frd_2(u,v),
\end{equation}
with $\BLap$ as in \eqref{eq:BLap} (expand the product and integrate
by parts as in \eqref{eq:Lap}): $\rmJ$ is thus the realization, as a
bounded operator from $\HtN$ to $\DHtN$, of the formal fourth order
operator $(\sfI-\Lap)^2=\sfI-2\Lap+\Lap^2$.
Under the identification, the duality pairing between $\H\subset\DHtN$
and $\HtN\subset\H$ reduces to the scalar product of $\H$:
\begin{equation}
	\label{eq:pairing-bridge}
	\la f,v\ra=(f,v)_\H
	\qquad\text{for every }f\in\H,\ v\in\HtN .
\end{equation}
The definitions and Theorem~\ref{thm:Minty} then apply to two
distinct instances inside the same triple:
\begin{itemize}
\item $\W=\HtN$, $\W'=\DHtN$: multivalued operators
	$\frA:\HtN\rightrightarrows\DHtN$, whose maximality is
	characterized by the solvability of
	$\lambda\rmJ v+\frA v\ni f$ for every $f\in\DHtN$, with the
	resolvent estimate \eqref{eq:resolvent-estimate-W} in the norms
	\eqref{eq:H2norm-equivalent} and \eqref{eq:dual-norm};
\item $\W=\H\equiv\H'=\W'$: multivalued operators
	$\sfA:\H\rightrightarrows\H$, for which we say that $\sfA$ is
	maximal \emph{in $\H$}. Since $\rmJ_\H$ is the identity,
	maximality in $\H$ is characterized by the solvability, for
	every $f\in\H$ (and every, or some, $\lambda>0$), of the
	classical resolvent equation
	\begin{equation}
		\label{eq:resolvent-abstract}
		\lambda u_\lambda+\sfA u_\lambda\ni f.
	\end{equation}
	This instance underlies the
	classical theory of contraction semigroups recalled in the next
	subsection.
\end{itemize}
The two instances are naturally related by a \emph{restriction}
construction:
$\frA$ induces a multivalued operator $\sfA:\H\rightrightarrows \H$ by restricting the graph of $\frA$ to $\H\times \H$:
\begin{equation}
	\label{eq:restriction}
	\graph(\sfA):=\graph(\frA)\cap (\H\times \H)=
	\{(v,f)\in \H\times \H:(v,f)\in \graph(\frA)\}.
\end{equation}
Still assuming that $\sfA$ is proper, we can define the domain of $\sfA$ as
\begin{equation}
	\label{eq:domain-restriction}
	\dom(\sfA):=\{v\in \H:(v,f)\in \graph(\sfA)\text{ for some }f\in \H\}.
\end{equation}
It is immediate to check that $\sfA$ is monotone in $\H$ if $\frA$ is monotone from $\HtN$ to $\DHtN$.

The interplay between the two instances of maximality is a crucial
point of our strategy. The next result shows that the maximality of
the restriction $\sfA$ is equivalent to the solvability in $\HtN$ of
the equation perturbed by the \emph{identity}, i.e.~by the
inclusion $\HtN\subset\H\subset\DHtN$ instead of the duality map
$\rmJ$, for data in the smaller space $\H$; no maximality of $\frA$
is assumed.
\begin{proposition}[Maximality of the restriction]
	\label{prop:restriction-maximality}
	Let $\frA:\HtN\rightrightarrows\DHtN$ be a monotone operator and let
	$\sfA$ be its restriction \eqref{eq:restriction}. The following
	properties are equivalent:
	\begin{enumerate}[\rm (a)]
	\item $\sfA$ is maximal monotone in $\H$;
	\item for every $f\in\H$ there exists $v\in\dom(\frA)$ such
		that
		\begin{equation}
			\label{eq:identity-resolvent-VV}
			v+\frA v\ni f\quad\text{in }\DHtN.
		\end{equation}
	\end{enumerate}
	In this case, for every $f\in\H$ the solution $v$ of
	\eqref{eq:identity-resolvent-VV} is unique, it belongs to
	$\dom(\sfA)$, and it coincides with the solution of
	\eqref{eq:resolvent-abstract} for $\lambda=1$.
\end{proposition}
\begin{proof}
	(b)$\Rightarrow$(a): let $f\in\H$ and let $v,\xi$ satisfy
	$\xi\in\frA v$, $v+\xi=f$ in $\DHtN$. Since $v\in\HtN\subset\H$ we
	get $\xi=f-v\in\H$, so that
	$(v,\xi)\in\graph(\frA)\cap(\H\times\H)=\graph(\sfA)$ and $u:=v$
	solves \eqref{eq:resolvent-abstract} with $\lambda=1$;
	Theorem~\ref{thm:Minty} (instance $\W=\H$) then shows that
	$\sfA$ is maximal monotone in $\H$. 

	The converse implication is trivial 
	since $\graph(\sfA)\subset\graph(\frA)$.

	The final conclusion is again a consequence of Theorem~\ref{thm:Minty}.
\end{proof}
Proposition~\ref{prop:restriction-maximality} describes precisely the
path we will follow for the DLSS operator. 	
By
solving $u+\frA u\ni f$ for every $f\in L^2(\Omega)\subset \DHtN$, we show the maximality of $\sfA$ in $\H$. 

Notice that the Cauchy problem \eqref{eq:Cauchy-A} of the next
subsection is driven by the $\H$-valued derivative of $u$, 
so that the
generation of the contraction semigroup requires the maximality of
$\sfA$ in $\H$: for the DLSS operator we will prove \emph{both}
properties in Section~\ref{sec:maximality}.

On the other hand, neither of the two maximality properties implies
the other one: picking $f_0\in\DHtN\setminus\H$, the constant
operator $\frA v:\equiv\{f_0\}$ is maximal
(Theorem~\ref{thm:Minty}) with an empty restricted graph, whereas
removing from $\graph(\rmJ)$ a single pair $(v_0,\rmJ v_0)$ with
$\rmJ v_0\notin\H$ destroys its maximality without affecting its
restriction, a nonnegative selfadjoint --- thus maximal monotone ---
operator in $\H$.

\subsection{Semigroups generated by maximal monotone operators}
\label{subsec:maximal-general}
For the general theory of maximal monotone operators in Hilbert spaces
and of their semigroups we refer to \cite{Brezis73}; here we recall
the statements that will be used in the sequel.

If $\sfA$ is maximal monotone
and $u_0\in \dom(\sfA)$,
the Cauchy problem 
\begin{equation}
	\label{eq:Cauchy-A}
	u'(t)+\sfA u(t)\ni 0,\quad 
	\text{a.e. in } (0,+\infty),\quad
	u(0)=u_0 
\end{equation}
has a unique solution 
$u\in \Lip([0,+\infty);\H)$ 
which can be expressed by the exponential formula
\begin{equation}
	\label{eq:exponential}
	u(t)=\lim_{n\to\infty}
	\Big(\rmI+\frac t n
	\sfA\Big)^{-n}u_0,\quad
	t>0,
\end{equation}
and it is right-differentiable:
\begin{equation}
	\label{eq:right-derivative}
	\lim_{h\down0}
	\frac{u(t+h)-u(t)}h=
	-\sfA^\circ u(t)\quad
	\text{for every }t\ge0 
\end{equation}
where $\sfA^\circ u$ 
denotes the (unique) element of minimal $\H$-norm in $\sfA u$.

The limit in \eqref{eq:exponential}
also holds for every 
$u_0\in \overline{\dom(\sfA)}$ 
and the corresponding map $\sfS_t: 
u_0\to u(t)$ is a semi-group of contractions in $\overline{\dom(\sfA)}$ satisfying
\begin{equation}
	\label{eq:semigroup}
	\lim_{t\down 0}\sfS_t u=u,\quad 
	\|\sfS_t u-\sfS_t v\|_\H 
	\le \|u-v\|_\H
	\quad
	\text{for every }u,v\in 
	\overline{\dom(\sfA)}.
\end{equation}
For every $u_0\in \overline{\dom(\sfA)}$ 
the continuous curve
$u\in \rmC([0,+\infty);\overline{\dom(\sfA)})$ defined by 
$u(t):=\sfS_t(u_0)$, $t\ge 0$, 
is the unique 
integral solution of 
\eqref{eq:Cauchy-A},
according to the B\'enilan formulation
\cite{Benilan72}, \cite[Ch.~III, Prop.~3.6]{Brezis73}:

\begin{equation}
	\label{eq:Benilan}
	\frac{\d}{\d t}
	\frac 12 \|u(t)-v\|_\H^2
	\le ( g,v-u(t))_\H
	\quad \text{in }\mathscr D'(0,+\infty)
	\quad\text{for every }
	(v,g)\in \graph(\sfA);\quad
	u(0)=u_0.
\end{equation}
\eqref{eq:Benilan}
can also be written as
\begin{equation}
	\label{eq:Benilan-integral}
	\frac 12 \|u(t)-v\|_\H^2
	-
	\frac 12 \|u(s)-v\|_\H^2
	\le 
	\int_s^t ( g,v-u(r))_\H\,\d r
	\quad\text{for every }
	0\le s<t,\ 
	(v,g)\in \graph(\sfA).
\end{equation}
We will refer to this class of curves as B\'enilan integral solutions.
The solutions to~\eqref{eq:Benilan-integral} are a weak form of~\eqref{eq:Cauchy-A}, since for $-u'(t)\in \sfA u(t)$ and $(v,g)\in \graph(\sfA)$, then
\begin{align*}
	\frac{\d}{\d t} \frac 12 \|u(t)-v\|_\H^2 = ( u'(t),u(t) -v)_\H  =
	 ( u'(t) + g , u(t) - v)_\H - (g , u(t) -v )_\H 
\end{align*}
and the monotonicity of $\sfA$ gives
$( u'(t) + g , u(t) - v)_\H
=-( -u'(t) - g , u(t) - v)_\H\le0$,
so that the inequality~\eqref{eq:Benilan} follows.
The solutions~\eqref{eq:Benilan-integral} are just the integrated version of~\eqref{eq:Benilan}.
\begin{remark}[The inhomogeneous equation]
	\label{rem:inhomogeneous}
	The maximal monotonicity of $\sfA$ in $\H$ yields the well-posedness of the \emph{inhomogeneous} Cauchy
	problem
	\begin{equation}
		\label{eq:Cauchy-f}
		u'(t)+\sfA u(t)\ni f(t)
		\quad\text{a.e.~in }(0,T),\qquad u(0)=u_0
	\end{equation}
	{\em \cite[Chap.~III]{Brezis73}}. If $f\in W^{1,1}(0,T;\H)$ (or
	$f\in\mathrm{BV}([0,T];\H)$) and $u_0\in\dom(\sfA)$, then
	\eqref{eq:Cauchy-f} admits a unique strong solution
	$u\in\Lip([0,T];\H)$; for arbitrary $f\in L^1(0,T;\H)$ and
	$u_0\in\overline{\dom(\sfA)}$ there is a unique \emph{integral
	solution} $u\in\rmC([0,T];\H)$ in the sense of B\'enilan
	\cite{Benilan72} (obtained as the uniform limit of the strong
	solutions corresponding to approximating data), characterized by
	the family of inequalities
	\begin{equation}
		\label{eq:Benilan-f}
		\frac12\|u(t)-v\|_\H^2-\frac12\|u(s)-v\|_\H^2\le
		\int_s^t\big(f(r)-g,u(r)-v\big)_\H\,\d r
		\quad
		\begin{gathered}
			\text{for every }0\le s<t\le T\\
			\text{and }(v,g)\in\graph(\sfA),
		\end{gathered}
	\end{equation}
	together with $u(0)=u_0$.
	Moreover, two integral solutions $u_1,u_2$ corresponding to the
	data $(u_{0,1},f_1)$, $(u_{0,2},f_2)$ satisfy the stability
	estimate
	\begin{equation}
		\label{eq:inhom-contraction}
		\|u_1(t)-u_2(t)\|_\H\le
		\|u_{0,1}-u_{0,2}\|_\H+
		\int_0^t\|f_1(r)-f_2(r)\|_\H\,\d r,
		\qquad t\in[0,T].
	\end{equation}
	This is one of the concrete gains of the semigroup approach: once the
	maximality of the (restricted) DLSS operator $\sfA$ in
	$L^2(\Omega)$ is established, existence, uniqueness and stability
	for the DLSS equation driven by a source term $f(t)$ follow at
	once from the abstract theory, whereas methods tailored to the
	autonomous equation (e.g.~those based on the decay of entropy
	functionals along the flow) 
	need more effort to be extended to 
	\eqref{eq:Cauchy-f}.
\end{remark}
\section{The DLSS operator}
\label{sec:DLSS}
The aim of this section is to introduce the 4-th order DLSS differential operator
\begin{equation}
	\label{eq:DLSS-smooth}
	u\mapsto \Lap^2 u-\frac{(\Lap u)^2}u,\quad u\ge 0
\end{equation}
starting first from a suitable class of sufficiently smooth functions
and then constructing its natural monotone extension $\frA$
in the duality $\DHtN$-$\HtN$ and
its restriction $\sfA$ to $L^2(\Omega)$ as we explained in Section \ref{subsec:monotone-VV}.
Our main result is to show that $\sfA$ is maximal monotone in $L^2(\Omega)$
and it is the \emph{unique} maximal monotone extension of 
\eqref{eq:DLSS-smooth}.

A third, \emph{weak}, level of interpretation of the DLSS
operator, defined on the whole class $\WpO$ of~\eqref{eq:defW},
with values in the dual of the test space $\ZO$, will be
introduced in Section~\ref{subsec:very-weak}: 
it plays no role in the
generation of the semigroup, but it will be crucial to characterize
the integral B\'enilan solutions starting from arbitrary initial data in $L^2(\Omega)$.

\subsection{The DLSS operator
for uniformly positive functions}
\label{subsec:positive}
We introduce the space
\begin{equation}
	\label{eq:Linftyplus}
	L^\infty_{+\!+}(\Omega):=\Big\{u\in L^\infty(\Omega):
	\essinf\limits_{\Omega} u>0\Big\},
\end{equation}
and we define the 
function $\fra:\big(\PXO\big)\times \big(\XO\big) \to \R$ 
\begin{equation}
	\label{eq:form}
	\begin{aligned}
		\fra(u,v):={}&\int_\Omega 
		\Big(\Lap u\Lap v-\frac{(\Lap u)^2}uv\Big)\,\d x
	        \\={}&\frd_2(u,v)-
	        \int_\Omega 
	        \frac{(\Lap u)^2}uv\,\d x,
			\qquad
			u\in \PXO,\ v\in \XO.
	\end{aligned}
\end{equation}
If $u\in \PXO$ then 
\begin{equation}
	\label{eq:operator}
	\frAo u:=\Lap^2 u-\frac{(\Lap u)^2}u
	\in \DHtN
	+L^1(\Omega)
\end{equation}
and it is easy to see that 
\begin{equation}
	\label{eq:operator-form}
	\la \frAo u , v\ra =
	\fra(u,v)\quad\text{for every }
	u\in \PXO,\ v\in \XO.
\end{equation}
Since $\fra(u,u)=0$ we also have
\begin{equation}
	\label{eq:formb}
	\fra(u,u-v)=-\fra(u,v)
	=\int_\Omega
	\Big({\frac{(\Lap u)^2}u}v-
	\Lap u\Lap v\Big)\,\d x.
\end{equation}
It is also easy to check that 
for every $u,v\in \PXO$ we have
\begin{align}
	\notag
	\fra(u,u-v)+\fra(v,v-u)&=
	-\fra(u,v)-\fra(v,u)\\&=
	\notag
	\int_\Omega
	\Big({\frac{(\Lap u)^2}u}v+
	{\frac{(\Lap v)^2}v}u
	-2\Lap u\Lap v\Big)\,\d x
	\\&=
		\label{eq:monotone}
		\int_\Omega 
	\Big|\Lap u\sqrt{\frac vu}-
	\Lap v\sqrt{\frac{u}{v}}\Big|^2
	\ge0.
\end{align}
\paragraph{Useful estimates:}
\begin{enumerate}[E1]
	\item $v=0$:
\begin{equation}
	\label{eq:est0}
	\fra(u,u)=0\quad \text{for every }u\in \PXO,
\end{equation}
which corresponds to the conservation of the $L^2(\Omega)$ norm.
\item $v=1$:
\begin{equation}
	\label{eq:est1}
	\fra(u,u-1)=-\fra(u,1)=
	\int_\Omega \frac{(\Lap u)^2}u\,\d x
	\quad\text{for every }u\in \PXO.
\end{equation}
Notice that the right-hand side is precisely the functional
$\sfD(u)$ of \eqref{eq:defD}: since \eqref{eq:est1} will provide a
uniform control of $\sfD$ on the domain of any reasonable extension
of the DLSS operator
(see Proposition~\ref{prop:defect-variational}),
Theorem~\ref{thm:Hessian} suggests that such a
domain should be contained in the class $\WpO$ of \eqref{eq:defW}
and, symmetrically, that on the whole of $\WpO$ the DLSS operator
should retain a (very weak) meaning: see
Section~\ref{subsec:very-weak}.
\item 
By Lemma 
\ref{le:useful-approximation}
\begin{equation}
	\label{eq:good-multiplication}
	u\in \XO,\ w\in \ZO
	\quad 
	\Rightarrow\quad
	uw\in \XO;
\end{equation}
choosing $u\in \PXO$, 
$w\in \ZO$ and $v:=uw$ 
we obtain 
\begin{equation}
	\label{eq:est0phi}
		\frb(u,w):=	\fra(u,u\,w)
=	
	\int_\Omega 
	\Big(u\Lap u\Lap w+2\Lap u
	\nabla u\cdot\nabla w\Big)\,\d x 
\end{equation}
It is interesting that the form 
$\frb$ 
is also well defined in $\HtN\times \ZO$.
\item
Suppose that $u\in \PXO$; then
the function $v=u\ln u$ belongs to $\XO$
with $\Lap v=\Lap u(1+\ln u)+|\nabla u|^2/u$.
Indeed, $|\nabla u|^2/u\in L^2(\Omega)$ 
by \eqref{eq:quartic-general} applied to $w:=u$. 

Since $\ln u\in L^\infty(\Omega)$, testing the
identity \eqref{eq:Lap} for $u$ against $(1+\ln u)\varphi$,
$\varphi\in H^1(\Omega)\cap L^\infty(\Omega)$, gives
$$\frd(v,\varphi)=-\int_\Omega\big(\Lap u(1+\ln
u)+|\nabla u|^2/u\big)\varphi\,\d x,$$ so that $v\in\HtN$ by
\eqref{eq:Lap} with the stated Laplacian.
Moreover
\begin{align}
\label{eq:crucial-estimate}
	\fra(u,u\ln u)&=
	\int_\Omega \bigg[(\Lap u)^2+
	\Lap u\,\frac{|\nabla u|^2}u\bigg]
	\ge 
	\int_\Omega 
	\bigg|
	\rmD^2 u-\frac{\nabla u\otimes \nabla u}u\bigg|^2\,\d x
	\ge 
	c_d \int_\Omega |\rmD^2 u|^2\,\d x
\end{align}
for some $c_d>0$, where 
the last inequality follows by
\eqref{eq:H2bound} and the 
calculations of 
\cite[Theorem~5.1]{Gianazza-Savare-Toscani09}.
\item 
The following estimate requires smoothness,
it will be justified by a different argument 
using Corollary \ref{cor:main-ineq2}.
\begin{equation}
	\label{eq:est2}
	\fra(u,-\Lap u)=
	\int_\Omega |\nabla\Lap u|^2\,\d x+
	\int_\Omega \frac{(\Lap u)^3}u\,\d x	
	=
	\int_\Omega \Bigl|\nabla\Lap u-
	\frac {\Lap u}{u}\nabla u
	\Bigr|^2\,\d x\ge 0
\end{equation}
since
\begin{align*}
	\int_\Omega \frac{(\Lap u)^3}u\,\d x&=
	\int_\Omega \Lap u\frac{(\Lap u)^2}u\,\d x
	\\&= 
	-\int_\Omega \nabla u
	\cdot \Big(2\frac{\Lap u}u\nabla \Lap u-
	\Big(\frac {\Lap u}{u}\Big)^2\nabla u\Big)\,\d x
	\\&=
	\int_\Omega \biggl( \Bigl(\frac {\Lap u}{u}\Bigr)^2
	|\nabla u|^2-
	2\frac{\Lap u}u\nabla \Lap u\cdot \nabla u
	\biggr)\,\d x.
\end{align*}
\end{enumerate}

\subsection{The realization 
of the DLSS operator between \texorpdfstring{$\HtN$}{H2N} and \texorpdfstring{$\DHtN$}{H-2N}.}
\label{subsec:frA}
We want to extend the previous
discussion, and in particular the definition
\eqref{eq:operator} of the
DLSS operator
$\frAo$
\begin{displaymath}
	\frAo(u)=\Lap^2 u-\frac{(\Lap u)^2}u
\end{displaymath}
 to a more general class 
 of function in $\HtN$.

Since $\Lap^2 u\in \DHtN$ for every $u\in \HtN$, 
 a first natural condition is to ask that 
 the singular term $\frac{(\Lap u)^2}u$ 
 defines a $L^1(\Omega)$ function
 (i.e.~$u\in \WpO$) which is also an element of $\DHtN$
 (see Remark \ref{rem:L1-H-2}).
 In order to gain maximality, we will also see that 
 $(\Lap u)^2/u$ should be relaxed,
  by adding arbitrary (nonnegative finite) measures 
  $\nu\in \MDHtN(\overline\Omega)$ 
  concentrated in the vacuum set $\{\tilde u=0\}$ 
  so that 
  $\mu=(\Lap u)^2/u\,\mathscr L^d+\nu\in \MDHtN(\overline\Omega)$ 
 satisfies $\tilde u\mu=(\Lap u)^2\,\Leb d$. 
 The use of the quasi-continuous representative $\tilde u$ 
 of $u\in \HtN$
 is needed to define the product $\tilde u\mu$ 
 if $d>3$, since 
 we do not know if $u$ is continuous.
 \smallskip
\begin{definition}[The operator $\frA$]
	\label{def:frA-strong}
	A nonnegative $u\in \HtN$ 
	belongs to 
	$\dom(\frA)$ if and only if
	$\mathsf D(u)<\infty$ 
	(i.e.~$u\in \WpO$) and  
	$\displaystyle
		\frac{(\Lap u)^2}{u}\,\Leb d\restr{\{u>0\}}
		\in \DHtN$.
	In this case we set 
	\begin{equation}
		\label{eq:framin}
		\frA^\circ(u):=
		\Lap^2 u-\frac{(\Lap u)^2}{ u}\,\Leb d\restr{\{ u>0\}}
		\in \DHtN,
	\end{equation}
	\begin{equation}
		\label{eq:varineq1}
		\frA(u):=
		\Big\{\frA^\circ(u)-\nu:
		\nu\in \MDHtN(\overline\Omega),
		\ 
		 \text{concentrated on }\{\tilde u=0\}\Big\}.	
	\end{equation}
	We call \emph{defect measure} 
	associated with $f\in \frA u$ the measure 
	$\nu\in \MDHtN(\overline\Omega)$
	satisfying $f=\frA^\circ(u)-\nu$. 
	In particular, if $\Cpt(\{\tilde u=0\})=0$ then
	the defect measure is necessarily null,
	$\frA(u)$ is a singleton and reduces to 
	$\frA^\circ(u)=\Lap^2 u-\frac{(\Lap u)^2}{u}\,\Leb d\restr{\{u>0\}}$.
\end{definition}
We wrote 
$\frac{(\Lap u)^2}{ u}\,\Leb d\restr{\{ u>0\}}$ to emphasize
the dual nature of the distribution, interpreted 
as a measure in $\MDHtN(\overline\Omega)$.
Notice that 
$$\sfh(u,\Lap u)\mathscr L^d=
\frac{(\Lap u)^2}{ u}\,\Leb d\restr{\{ u>0\}}=
\frac{(\Lap u)^2}{\tilde u}\,\Leb d\restr{\{\tilde u>0\}}
\quad \text{since $u=\tilde u$ $\mathscr L^d$-a.e.}$$
As usual, we will often simply write $\frac{(\Lap u)^2}{ u}=
\sfh(u,\Lap u)$. 

Let us highlight a few relevant properties related to the definition of $\frA$, which will be useful to obtain
equivalent characterizations.
\begin{lemma}
	\label{le:obvious}
	Suppose that 
	$(u,f)\in \graph(\frA)\subset \HtN
	\times \DHtN $.
	Then 
	\begin{enumerate}[\rm ({A}.1)]
		\item $u\ge 0$.
		\item 
		$\mu:=\Lap^2 u-f$ 
		is a positive measure in $\MDHtN(\overline\Omega)$, thus characterized by 
		\begin{equation}
			\label{eq:DLSS1-pre}
			\int_{\overline\Omega} v\,\d\mu=
			\int_\Omega \Lap u\Lap v\,\d x-\la f,v\ra 
			\quad\text{for every }v\in \ZO.
		\end{equation}
		\item $\tilde u\mu =(\Lap u)^2\,\Leb d$; equivalently  
		\begin{equation}
			\label{eq:DLSS2-pre}
			\la \mu,uw\ra=\int_\Omega (\Lap u)^2w\,\d
                        x\quad\text{for every }w\in \ZO.
		\end{equation}
		\item[\rm (A.3')] We have
		\begin{gather}
		\label{eq:defect-lower-pre}
		\mu\ge\frac{(\Lap u)^2}{u}\,\Leb d\restr{\{\tilde u>0\}}
		\quad\text{as measures on $\overline\Omega$,}
		\\
		\label{eq:defect-global-pre}
		\int_{\overline\Omega}\tilde u\,\d\mu\le \int_\Omega(\Lap u)^2\,\d x,
		\quad \text{ and }\quad
		\la f,u\ra\ge 0
	\end{gather}
	\item[\rm (A.3'')]
	Recalling the form $\frb:\HtN\times\ZO\to \R$
introduced by 
\eqref{eq:est0phi}
\begin{equation}
	\label{eq:frb}
		\frb(u,w):=
	\int_\Omega 
	\Big(u\Lap u\Lap w+2\Lap u
	\nabla u\cdot\nabla w\Big)\,\d x,
\end{equation}
	it holds
	\begin{equation}
		\label{eq:DLSS2b-pre}
		\frb(u,w)=
		\la f,uw\ra 
		\quad\text{for every }w\in
		\ZO.
	\end{equation}	
	\end{enumerate}
\end{lemma}
\begin{remark}
	\label{rem:check}
	The equivalences stated in 
	{\rm (A.2) and (A.3)} follow
	by the 
	the density of
	$\ZO$ in 
	$\HtN$ and in $\rmC(\overline\Omega)$ (see Lemma~\ref{le:useful-approximation}), respectively.
	In fact, 
	the density of
	$\ZO$ in 
	$\HtN$
	shows that 
	\eqref{eq:DLSS1-pre} yields
	\begin{displaymath}
		\la \BLap u-\mu,v\ra=
		\la f,v\ra
		\quad\text{for every }v\in \HtN,
	\end{displaymath}
	so that \eqref{eq:DLSS1-pre} yields
	$\BLap u-\mu=f$ (here we identify $\mu$ with $\ell_\mu$).
	
	Similarly, \eqref{eq:DLSS2-pre}
	and \eqref{eq:wow}
	yield
	\begin{displaymath}
		\int_{\overline\Omega} w\tilde u \,\d\mu=
		\int_\Omega (\Lap u)^2 w\,\d x
		\quad\text{for every }w\in \rmC(\overline\Omega),
	\end{displaymath}
	which shows that $\tilde u\mu=(\Lap u)^2\,\Leb d$.
\end{remark}
\begin{proof}
	(A.1) is obvious.
	
	\noindent
	(A.2): we can write $f=\Lap^2 u-\frac{(\Lap u)^2}u\mathscr L^d\restr{u>0}-\nu$ with
	$\nu\in \MDHtN(\overline\Omega)$ concentrated on 
	$\{\tilde u=0\}$, so that
	\begin{equation}
		\label{eq:mu}
		\mu=\frac{(\Lap u)^2}u\mathscr L^d\restr{u>0}+\nu
		\in \MDHtN(\overline\Omega.)
	\end{equation}
	(A.3) Just multiply \eqref{eq:mu} by $\tilde u$
	(which coincides with $u$ $\mathscr L^d$-a.e.) and use the fact that $\tilde u\nu=0$.

	\noindent
	(A.3') still follows by \eqref{eq:mu}

	\noindent
	(A.3''):
	Recall that, by Lemma \ref{le:useful-approximation} and 
\eqref{eq:Leibniz2}, for every
$w\in \ZO $ 
it holds $uw\in \HtN$ 
and 
\begin{align}
	\notag 
	\int_\Omega \Lap u\,\Lap
	(uw)\,\d x
	&=
	\int_\Omega \Lap u
	\Big(
	\Lap u\,w+2\nabla u\cdot\nabla w +
		u\Lap w\Big)\,\d x
	\\&=
	\int_\Omega (\Lap u)^2
	w\,\d x+ 
	\frb(u,w)
	\label{eq:later}
\end{align}
Since \eqref{eq:DLSS2-pre} yields
\begin{displaymath}
	\la \mu,uw\ra=\int_\Omega (\Lap u)^2
	w\,\d x
\end{displaymath}
we deduce by \eqref{eq:DLSS1-pre}
	\begin{equation}
	\label{eq:equivalence-step1}
	\la f,uw\ra 
	=	
	\int_\Omega \Lap u\,\Lap
	(uw)\,\d x-
		\la \mu,uw\ra
		=\frb(u,w)
	\end{equation}
	so that \eqref{eq:DLSS2b-pre}
	holds.
\end{proof}
We will show now that properties (A.1)--(A.3) (or (A.3'))
are in fact sufficient to characterize a pair in the graph of $\frA$.
\begin{proposition}[Variational properties]
	\label{prop:defect-variational}
	Let $(u,f)\in\HtN\times\DHtN$ 
	be satisfying {\em (A.1) and (A.2)} 
	of Lemma \ref{le:obvious}. 
	Then condition {\rm(A.3)} is equivalent to condition
	${\rm (A.3')}$. 
	
	If they are satisfied then 
	\eqref{eq:defect-global-pre} in fact holds with equality,
		\begin{equation}
			\label{eq:defect-global}
			\int_{\overline\Omega}\tilde u\,\d\mu=\int_\Omega(\Lap u)^2\,\d x,\qquad 
		\la f,u\ra=0.
		\end{equation}
	and
	\begin{enumerate}[\rm 1.]
		\item The following estimates hold:
		\begin{align}
		\label{eq:mass-identity}
		\mu(\overline\Omega)&=-\la f,1\ra\le
		|\Omega|^{1/2}\,\DHtnrm f,
		\\
		\label{eq:mu-ac-part}
		\mu\restr{\{\tilde u>0\}}&=
		\frac{(\Lap u)^2}{\tilde u}\,
		\Leb d\restr{\{\tilde u>0\}},
		\\
		\label{eq:Dfinite}
		\sfD(u)&=\mu(\{\tilde u>0\})\le-\la f,1\ra<\infty.
	\end{align}
		\item 
		$\sqrt u\in\HtN$, i.e.~$u\in\WpO$, and
	\begin{equation}
		\label{eq:static-bounds}
		\frac4{d^2}\int_\Omega|\rmD^2\sqrt u|^2\,\d x\le
		\sfD(u)\le|\Omega|^{1/2}\,\DHtnrm f .
	\end{equation}
		\item $\displaystyle
		\frac{(\Lap u)^2}{\tilde u}\,\Leb d\restr{\{\tilde u>0\}}
		\in \DHtN\cap L^1(\Omega)$,
		\item 
		the excess
		$\nu:=\mu-(\Lap u)^2\tilde u^{-1}\Leb d\restr{\{\tilde u>0\}}\ge0$
		is concentrated on the vacuum set $\{\tilde u=0\}$.	
	\end{enumerate}
\end{proposition}
\begin{corollary}
	\label{cor:DLSS-operator1}
	A pair $(u,f)\in \HtN
	\times \DHtN $ belongs to the graph
	of the DLSS operator $\frA$ 
	(or, equivalently, $f\in \frA u$) iff 
	{\em (A.1), (A.2)} and {\em (A.3) or (A.3')}
	hold.
\end{corollary}
\begin{proof}[Proof of Proposition \ref{prop:defect-variational}]
	Let us first observe that 
	\eqref{eq:mass-identity}
	only depends on (A.1)-(A.2) and follows immediately by
	testing \eqref{eq:DLSS1-pre} with $v\equiv1\in\ZO$.

	\smallskip \noindent 	
	In order to the implication (A.3)$\Rightarrow$(A.3'), it is sufficient to 
	show that (A.3) implies \eqref{eq:defect-global}
	(which clearly implies 
	\eqref{eq:defect-global-pre})
	and \eqref{eq:mu-ac-part}.

	\eqref{eq:defect-global} follows immediately 
	testing \eqref{eq:DLSS1-pre} with $v=u$ (recall Remark~\ref{rem:check}) and
	\eqref{eq:DLSS2-pre} with $w\equiv1$ we get
	\begin{equation}
		\label{eq:identity}
		\la f,u\ra=\int_\Omega(\Lap u)^2\,\d x-\int_{\overline\Omega}\tilde
		u\,\d\mu=0
	\end{equation}
	Concerning \eqref{eq:mu-ac-part}: by Remark~\ref{rem:check} the
	finite Borel measures $\tilde u\,\mu$ and $(\Lap u)^2\Leb d$
	act in the same way on $\rmC(\overline\Omega)$, hence they
	coincide; integrating the bounded Borel functions $\tilde
	u^{-1}\mathbf1_{\{\tilde u>1/n\}}w$, $w\in\rmC(\overline\Omega)$,
	and letting $n\uparrow\infty$ by monotone convergence, we obtain
	of \eqref{eq:mu-ac-part}.

	Conversely, assume 
	that (A.3') (i.e.~\eqref{eq:defect-lower-pre}--\eqref{eq:defect-global-pre}) holds
	and set $\nu:=\mu-(\Lap u)^2\tilde u^{-1}\Leb d\restr{\{\tilde
	u>0\}}\ge0$. Then 
	$$\int_{\overline\Omega}\tilde u\,\d\mu
	=\int_\Omega(\Lap u)^2\,\d x+\int_{\overline\Omega}\tilde u\,\d\nu,$$ so
	\eqref{eq:defect-global-pre} forces $\int_{\overline\Omega}\tilde u\,\d\nu\le0$;
	as $\tilde u\ge0$ and $\nu\ge0$ we get
	$\int_{\overline\Omega}\tilde u\,\d\nu=0$, i.e.~$\nu$ is concentrated on
	$\{\tilde u=0\}$. Hence
	$\tilde u\mu=\tilde u\,(\Lap u)^2\tilde u^{-1}\Leb
	d\restr{\{\tilde u>0\}}+\tilde u\,\nu=(\Lap u)^2\,\Leb d$ (again $\Lap u=0$
	on $\{u=0\}$), which is {\rm(A.3)}. Finally, testing $\Lap^2
	u-\mu=f$ with $u\in\HtN$ gives
	$\la f,u\ra=\int_\Omega(\Lap u)^2\,\d x-\int_{\overline\Omega}\tilde u\,\d\mu$,
	so the equality in \eqref{eq:defect-global} is the same as
	$\la f,u\ra=0$.

	\medskip \noindent	
	Let us now assume (A.3)-(A.3') and prove the other properties. 
	
	Concerning 1., we have already checked \eqref{eq:mass-identity} and \eqref{eq:mu-ac-part}.
	Since $\tilde u=u$ $\Leb d$-a.e.~and
	$\Lap u=0$ $\Leb d$-a.e.~on $\{u=0\}$ (by the level set property
	of Sobolev functions, as in the proof of
	Theorem~\ref{thm:DLSS-monotone}), the total mass of
	$\mu\restr{\{\tilde u>0\}}$ coincides with $\int_\Omega \sfh(u,\Lap
	u)\,\d x=\sfD(u)$, and it is bounded by
	$\mu(\overline\Omega)=-\la f,1\ra$.
	Since $\sfD(u)<+\infty$, Theorem~\ref{thm:Hessian} gives \eqref{eq:Dfinite} and $\sqrt
	u\in\HtN$. 
	The regularity property of Claim 2. and \eqref{eq:static-bounds} immediately follow.

	Claims 3. and 4. have already been discussed in the 
	proof of the implication (A.3')$\Rightarrow$(A.3).
\end{proof}
\begin{remark}
	\label{rem:mu-removed}
	If $u\in \HtN$ and $f\in \DHtN$,
	condition {\rm (A.2)} implies
		\begin{equation}
			\label{eq:DLSS1bis}
			\int_\Omega \Lap u\Lap v\,\d x\ge 
			\la f,v\ra
                        \quad\text{for every }v\in \HtN,
                        \ v\ge 0.
		\end{equation}
		i.e.~$\BLap u\ge f$ in $\DHtN$.
\end{remark}
The role of condition (A.3'') is clarified by the following result.
\begin{proposition}
	\label{le:DLSS-equivalence}
	Let $u\in \HtN,\ f\in \DHtN$.
	$(u,f)$ belongs to the graph of $\frA$
	if and only if 
	\begin{enumerate}
		\item[\rm ({A}.1)] $u\ge 0$. 
	\item[\rm ({A}.2')] 
	$\BLap u\ge f$ in $\DHtN$, i.e.
	\begin{equation}
			\label{eq:DLSS1bisbis}
			\int_\Omega \Lap u\Lap v\,\d x\ge 
			\la f,v\ra
            \quad\text{for every }v\in \PZO
		\end{equation}
	\item[{\rm (A.3'')}]	
	\begin{equation}
		\label{eq:DLSS2bis}
		\frb(u,w)=
		\la f,uw\ra 
		\quad\text{for every }w\in
		\ZO.
	\end{equation}	
	\end{enumerate}
\end{proposition}
\begin{proof}
If 
$(u,f)\in \frA$
according to Definition \ref{def:frA-strong},
then (A.2') is trivial and we have already checked
(A.3''). 

	Conversely, let us
	assume that 
	(A.1), (A.2'), and (A.3'') hold.
	
	By \eqref{eq:DLSS1bisbis}
	the linear functional
	$\ell\in \DHtN$ defined by 
	\begin{equation*}
		\la \ell,v\ra
		:=
		\int_\Omega \Lap u\,\Lap v
		\,\d x-
		\la f,v\ra\quad
		v\in \HtN,
	\end{equation*}
	is nonnegative;
	By Theorem
	\ref{thm:representation}
	there is a unique measure
	$\mu\in \MDHtN(\overline\Omega)$ 
	such that $\ell=\ell_\mu$
	according to \eqref{eq:representation}:
	in particular,
	for this precise choice
	of $\mu$ \eqref{eq:DLSS1-pre} holds.
	
	On the other hand
	the identity 
	\eqref{eq:later} and 
	\eqref{eq:DLSS2bis}
	yield
	\begin{align*}
		\la \mu,uw\ra
		&=\int_\Omega \Lap u\,\Lap
	(uw)\,\d x-
		\la f,uw
		\ra 
		\\&=
		\int_\Omega \Lap u\,\Lap
	(uw)\,\d x-\frb(u,w)
		\\&=
		\int_\Omega (\Lap u)^2w\,\d x
		\quad
		\text{for every }w\in \ZO
	\end{align*}
	which in turn implies 
	\eqref{eq:DLSS2-pre} and (A.3). 
	We can eventually invoke 
	Corollary \ref{cor:DLSS-operator1}.
\end{proof}
We derive a further characterization.
\begin{corollary}
	\label{cor:easy2}
	A pair $(u,f)\in \HtN\times \DHtN$
	belongs to the graph of $\frA$ if and only if
	$u\ge 0$ and 
	\begin{equation}
		\label{eq:equivalent-better}
		\Lap^2 u\ge f+
		\sfh(u,\Lap u)
		\quad \text{in }\DHtN,
        \quad \text{and}\quad 
		\langle f,u\rangle \ge 0.
	\end{equation}
	If this is the case, then $\langle f,u\rangle=0$.
\end{corollary}
\begin{proof}
	One implication is an immediate consequence of 
	the definition of $\frA$.

	Conversely, if \eqref{eq:equivalent-better} holds
	then $\int_\Omega \sfh(u,\Lap u)\,\d x<\infty$ 
	so that $\mathsf D(u)<\infty$ and 
	$\displaystyle
		\frac{(\Lap u)^2}{\tilde u}\,\Leb d\restr{\{\tilde u>0\}}
		\in \DHtN$.
	Defining $\mu:=\Lap^2 u-f$ we immediately have
	\eqref{eq:defect-lower-pre} and
	\eqref{eq:defect-global-pre}. 
	We can then apply Corollary~\ref{cor:DLSS-operator1}.
\end{proof}
\begin{remark}
	\label{rem:defect-variational}
	Condition \eqref{eq:defect-global} contains the mass identity
	$\la f,u\ra=0$, the static counterpart of the conservation of
	$\|u_t\|_2^2$ along the flow. 
	The
	measure $\mu$ representing $\Lap^2u-f$ is allowed to \emph{exceed}
	the absolutely continuous square $(\Lap u)^2\tilde u^{-1}\Leb d$, the excess being the reaction on the vacuum, the normal-cone
	component $\partial I_{L^2_+}(u)$ of
	Proposition~\ref{prop:minimal-section}, while the single global
	constraint pins that excess to $\{\tilde u=0\}$, recovering the
	vacuum-trace structure of Theorem~\ref{thm:vacuum-trace}.
\end{remark}%

\subsection{The finer structure of the defect measure}
We now investigate the structure of 
the defect measure $\nu$ 
of \eqref{eq:varineq1} of $\frA$ more closely
assuming that $f\in \MDHtN(\overline\Omega)$ (or, equivalently, that 
$\Lap^2 u\in \MDHtN(\overline\Omega)$.) 
This extra
information, combined with the regularity
$\sqrt u\in\HtN$ of Proposition~\ref{prop:defect-variational}, allows a
complete identification of the measure $\nu$, which coincides
with 
the restriction of $-f$ to $\{\tilde u=0\}$.

The crucial tool is provided by the following concentration  result.
\begin{lemma}
	\label{le:Lap2structure}
	Let $u\in \HtN$ with $w=\sqrt u\in \HtN$ and 
	let $w_\eps:=\varphi(w/\eps)$ where 
	$\eps>0$ and $\varphi\in\rmC^\infty_c([0,+\infty))$ with $0\le\varphi\le1$,
	$\varphi\equiv1$ on $[0,\tfrac12]$ and $\varphi\equiv0$ on
	$[1,+\infty)$. Then for every $\chi\in\ZO$ 
	\begin{equation}
		\label{eq:technical}
		\lim_{\eps\downarrow0}\int_\Omega \Lap u\, \Lap (\chi w_\eps)\,\d x=0
	\end{equation}
\end{lemma}
\begin{proof}
	\emph{Step 0.}
	Since $w\in \HtN$ the quantities
	\begin{equation*}
		\sfT:=\int_\Omega(\Lap w)^2\,\d x,\qquad
		\sfG:=\int_{\{w>0\}}\frac{|\nabla w|^4}{w^2}\,\d x,\qquad
		\sfN:=\int_\Omega|\nabla w|^2\,\d x
	\end{equation*}
	are finite (for $\sfG$, recall \eqref{eq:quartic}); moreover
	$\nabla w=0$, $\Lap u=0$ a.e.~on $\{w=0\}=\{u=0\}$ and
	$\Lap u=2w\Lap w+2|\nabla w|^2$ a.e.~in $\Omega$
	(Lemma~\ref{le:sqrt-calculus}).

	\emph{Step 1 (test functions).}
	By the chain rule in $H^1(\Omega)$,
	$\nabla w_\eps=\tfrac1\eps\varphi'(w/\eps)\nabla w$, and all the
	derivatives are supported in the annulus
	$A_\eps:=\{\eps/2\le w\le\eps\}$.
	Testing \eqref{eq:Lap} for $w$ against
	$\tfrac1\eps\varphi'(w/\eps)\zeta$, $\zeta\in H^1(\Omega)\cap
	L^\infty(\Omega)$, we obtain
	\begin{equation*}
		\int_\Omega\nabla w_\eps\cdot\nabla\zeta\,\d x=
		-\int_\Omega\Big(
		\tfrac1\eps\varphi'(w/\eps)\Lap w+
		\tfrac1{\eps^2}\varphi''(w/\eps)|\nabla w|^2
		\Big)\zeta\,\d x ,
	\end{equation*}
	where the right-hand side defines an $L^2$-function: indeed on
	$A_\eps$ we have $w\ge\eps/2$, so that
	$$\int_{A_\eps}|\nabla w|^4\le\eps^2\sfG_\eps
	\quad\text{with}\quad 
	\sfG_\eps:=\int_{A_\eps}|\nabla w|^4/w^2.$$
	By \eqref{eq:Lap} we conclude that $w_\eps\in\HtN$ with
	\begin{equation}
		\label{eq:Lap-veps}
		\Lap w_\eps=
		\tfrac1\eps\varphi'(w/\eps)\Lap w+
		\tfrac1{\eps^2}\varphi''(w/\eps)|\nabla w|^2 .
	\end{equation}
	Setting also
	$\sfT_\eps:=\int_{A_\eps}(\Lap w)^2$ and
	$\sfN_\eps:=\int_{A_\eps}|\nabla w|^2$, the tails
	$\sfT_\eps,\sfG_\eps,\sfN_\eps$ all vanish as $\eps\down0$, by
	dominated convergence, since
	$\mathbf1_{A_\eps}\to0$ pointwise in $\Omega$.

	\emph{Step 2 (passage to the limit).}
	Let $\chi\in\ZO$, so that by \eqref{eq:Z-acts},
	$$\chi w_\eps\in\HtN,
	 \quad
	\Lap(\chi w_\eps)=\chi\Lap w_\eps+2\nabla\chi\cdot\nabla
	w_\eps+w_\eps\Lap\chi,\quad 
	\int_\Omega \Lap u\,\Lap (\chi w_\eps)\,\d x=
	I_\eps+II_\eps+III_\eps$$
	Using $\Lap u=2w\Lap w+2|\nabla w|^2$,
	$w\le \eps$ on $A_\eps$, and the Cauchy--Schwarz inequality
	we have
	\begin{align*}
		|I_\eps|&=
		2\left|\int_\Omega \Big(w\Lap w+|\nabla w|^2\Big)
		\Big(\tfrac1\eps\varphi'(w/\eps)\Lap w+
		\tfrac1{\eps^2}\varphi''(w/\eps)|\nabla w|^2\Big)
		\chi\,\d x\right|
		\\
		&\le C 
		\int_{A_\eps}
		\bigg[\frac w\eps (\Lap w)^2+
		\Big(\frac w\eps+ \frac{w^2}{\eps^2}\Big)
		\, |\Lap w|\,  \frac{|\nabla w|^2}w
		+
		\frac{w^2}{\eps^2} \frac{|\nabla w|^4}{w^2}\bigg]
		\chi\,\d x
		\\
		&\le 
		C\,\|\chi\|_\infty
		\Big[\sfT_\eps+2(\sfT_\eps\sfG_\eps)^{1/2}
		+\sfG_\eps\Big]\\
		|II_\eps|&=\frac 2\eps\left|\int_\Omega
		\Big(w\Lap w+|\nabla w|^2\Big)\,
		\varphi'(w/\eps)\,
		\nabla\chi\cdot
		\nabla w\,\d x\right|\\
		&\le C\int_{A_\eps}
		\Big(\frac w\eps |\Lap w|\,|\nabla w|
		 + \frac1\eps |\nabla w|^3\Big)\,\,
		|\nabla\chi|\,\d x
		\le C\,\|\nabla\chi\|_\infty
		\Big[(\sfT_\eps\sfN_\eps)^{1/2}
		+\eps^{1/2}\sfG_\eps^{3/4}|\Omega|^{1/4}\Big],\\
		III_\eps&=\int_\Omega\Lap u\,w_\eps\,\Lap\chi\,\d x
		\;\longrightarrow\;
		\int_{\{u=0\}}\Lap u\,\Lap\chi\,\d x=0,
	\end{align*}	
	where the last limit 
	follows by dominated convergence, since
	$w_\eps\to\mathbf1_{\{w=0\}}$ pointwise and $\Lap u=0$
	a.e.~on $\{u=0\}=\{w=0\}$. Hence
	$\int_\Omega\Lap u\,\Lap(\chi w_\eps)\,\d x\to0$ as
	$\eps\down0$.
\end{proof}
As an application, we obtain the following description of the 
defect measure.
\begin{theorem}
	\label{thm:vacuum-trace}
	Let $(u,f)\in\graph(\frA)$ with 
	defect measure $\nu$ and associated measure $\mu=
	\Lap^2u-f$,
	and let
	$\mu_0:=\mu\restr{\{\tilde u=0\}}$. 
	If $f\in \MDHtN(\overline\Omega)$ then
	\begin{equation}
		\label{eq:mu-structure-2}
		\mu=\frac{(\Lap u)^2}{\tilde u}\,
		\Leb d\restr{\{\tilde u>0\}}
		\;+\;(-f)\,\restr{\{\tilde u=0\}} .
	\end{equation}
	In particular:
	\begin{enumerate}[\rm (i)]
	\item $\nu=
	\mu_0=(-f)\restr{\{\tilde u=0\}}$;
	\item the restriction of 
	$f$ on the vacuum set $\{\tilde u=0\}$ is nonpositive;
	\item $\mu_0=0$ if and only if $f=0$ on
		$\{\tilde u=0\}$.
	\end{enumerate}
\end{theorem}
\begin{proof}
	Let us set $w=\sqrt u\in \HtN$, $\varphi, w_\eps,\chi$
	as in the previous Lemma and consider the quasi-continuous representatives
	$\tilde w=\sqrt{\tilde u}$, $\tilde w_\eps:=
	\varphi(\tilde w/\eps)$.

	$\chi\,\tilde w_\eps$ is $\Cpt$-quasi continuous and coincides
	$\Leb d$-a.e.~with $\chi w_\eps$: it is therefore the
	quasi-continuous representative of $\chi w_\eps$. As
	$\eps\down0$ it converges pointwise to
	$\chi\mathbf1_{\{\tilde u=0\}}$, so that, by dominated
	convergence with respect to the finite measure $\mu$ and by
	\eqref{eq:wow},
	\begin{equation*}
		\la\mu,\chi w_\eps\ra\to\int_{\{\tilde
		u=0\}}\chi\,\d\mu=
		\int_{\overline\Omega}\chi\,\d\mu_0,
		\qquad
		\la f,\chi w_\eps\ra=\int_\Omega \chi \tilde w_\eps\,\d f\to
		\int_{\{\tilde u=0\}}\chi\,\d f .
	\end{equation*}

	\emph{Conclusion.}
	Testing $f+\mu$ against $\chi w_\eps$ and 
	using (A.2) of Lemma~\ref{le:obvious}
	we get
	\begin{displaymath}
		\la f+\mu,\chi w_\eps\ra=
		\int \chi \tilde w_\eps \,\d(f+\mu)=
		\int_\Omega \Lap u\,\Lap(\chi w_\eps);
	\end{displaymath}
	letting
	$\eps\down0$ we obtain
	$-\int\chi\,\d\mu_0=\int_{\{\tilde u=0\}}\chi\,\d f$ for every
	$\chi\in\ZO$; by the density of $\ZO$ in
	$\rmC(\overline\Omega)$, the finite Borel measures $\mu_0$ and
	$(-f)\restr{\{\tilde u=0\}}$ coincide. Together with
	\eqref{eq:mu-ac-part} this proves \eqref{eq:mu-structure-2};
	(i) and (ii) follow from $\mu_0\ge0$, and (iii) is immediate.
\end{proof}
\subsection{Monotonicity of \texorpdfstring{$\frA$}{A} in the duality pairing
\texorpdfstring{$\HtN$--$\DHtN$}{H2N-H-2N}}
\begin{theorem}
	\label{thm:DLSS-monotone}
	The DLSS operator $\frA$
	is monotone from $\HtN$ to $\DHtN$.
\end{theorem}
\begin{proof}
	We want to reproduce the
	simpler argument of
	\eqref{eq:monotone} in this
	more general setting.
	
	Let $(u_i,f_i)$, $i=1,2$, be elements
	in the graph of $\frA$
	in $\HtN\times \DHtN$,
	let $\tilde u_i$ be $\Cpt$-quasi-continuous representatives of $u_i$, and
	let $\mu_i\in \MDHtN(\overline\Omega)$ be the corresponding 
	measures as in Definition~\ref{def:frA-strong}, so that
	$f_i=\BLap u_i-\mu_i$.
	We want to prove that 
	\begin{equation}
		\label{eq:DLSS-monotone}
	\la f_1-f_2,u_1-u_2\ra \ge 0.
	\end{equation}
	We have
	\begin{align*}
		\duality{f_i}{u_i}=
		\la \BLap u_i-\mu_i,u_i\ra=
		\int_\Omega (\Lap u_i)^2\,\d x-
		\int_{\overline\Omega} \tilde u_i\,\d\mu_i=0
	\end{align*}
	so that \eqref{eq:DLSS-monotone}
	is equivalent to show
	that
	\begin{equation}
		\label{eq:DLSS-monotone2}
		\duality 
		{ f_1}{u_2}+
		\duality{f_2}{u_1}\le  0.
	\end{equation}
	We have 
	\begin{align*}
		\duality{f_1}{u_2}&=
		\la \BLap u_1-\mu_1,u_2\ra=
		\int_\Omega \Lap u_1\,\Lap u_2\,\d x-
		\int_{\overline\Omega} \tilde u_2\,\d\mu_1,\\
		\duality {f_2}{u_1}&=
		\la \BLap u_2-\mu_2,u_1\ra=
		\int_\Omega \Lap u_2\,\Lap u_1\,\d x-
		\int_{\overline\Omega} \tilde u_1\,\d\mu_2
	\end{align*}
	and therefore 
	\eqref{eq:DLSS-monotone2} amounts to prove
	\begin{equation}
		\label{eq:DLSS-monotone3}
		\int_{\overline\Omega} \tilde u_2\,\d\mu_1+
		\int_{\overline\Omega} \tilde u_1\,\d\mu_2\ge 2\int_\Omega \Lap \tilde u_1\,\Lap \tilde u_2\,\d x.
	\end{equation}
	Let $\Omega_i:=\{x\in \Omega:
	\tilde u_i>0\}$. 
	Since $\tilde u_i\in \HtN$ we know that 
	$\Lap \tilde u_i=0$ $\Leb d$-a.e.~in $\Omega\setminus \Omega_i$ so that 
	\begin{displaymath}
		2\int_\Omega \Lap 
		\tilde u_1\,\Lap \tilde u_2\,\d x
		=
		\int_{\Omega_1\cap \Omega_2} 2\Lap 
		\tilde u_1\,\Lap \tilde u_2\,\d x
	\end{displaymath}
	On the other hand,
	since $\tilde u_i$ and $\mu_i$ are nonnegative, we have
	\begin{align*}
		\int_{\overline\Omega} \tilde u_2\,\d\mu_1
		&\ge 
		\int_{\Omega_1} \tilde u_2\,\d\mu_1=
		\int_{\Omega_1} 
		\frac{\tilde u_2}{\tilde u_1}\,\tilde u_1\,\d\mu_1
		=
		\int_{\Omega_1} 
		\frac{\tilde u_2}{\tilde u_1}\,
		(\Lap u_1)^2\,\d x
		\ge 
		\int_{\Omega_1\cap \Omega_2} 
		\frac{\tilde u_2}{\tilde u_1}\,
		(\Lap u_1)^2\,\d x\\
		\int_{\overline\Omega} \tilde u_1\,\d\mu_2
		&\ge 
		\int_{\Omega_2} \tilde u_1\,\d\mu_2=
		\int_{\Omega_2} 
		\frac{\tilde u_1}{\tilde u_2}\,\tilde u_2\,\d\mu_2
		=
		\int_{\Omega_2} 
		\frac{\tilde u_1}{\tilde u_2}\,
		(\Lap u_2)^2\,\d x
		\ge 
		\int_{\Omega_1\cap \Omega_2} 
		\frac{\tilde u_1}{\tilde u_2}\,
		(\Lap u_2)^2\,\d x
	\end{align*}
	so that 
	\begin{align*}
		\int_{\overline\Omega} \tilde u_2\,\d\mu_1
		+\int_{\overline\Omega} \tilde u_1\,\d\mu_2
		&\ge 
		\int_{\Omega_1\cap\Omega_2}
		\bigg(
		\frac{\tilde u_2}{\tilde u_1}\,
		(\Lap u_1)^2
		+\frac{\tilde u_1}{\tilde u_2}\,
		(\Lap u_2)^2
		\bigg)\,\d x
		\\&=
		\int_{\Omega_1\cap \Omega_2} 2\Lap
		\tilde u_1\,\Lap \tilde u_2\,\d x+
		\int_{\Omega_1\cap\Omega_2}
		\bigg(
		\sqrt{\frac{\tilde u_2}{\tilde u_1}}\,
		\Lap u_1
		{-}\sqrt{\frac{\tilde u_1}{\tilde u_2}}\,
		\Lap u_2
		\bigg)^2\,\d x
		\\&\ge
		\int_{\Omega_1\cap \Omega_2} 2\Lap
		\tilde u_1\,\Lap \tilde u_2\,\d x. 
	\end{align*}
\end{proof}
This subsection culminates in the following result, whose proof
requires the machinery of {Section~\ref{sec:maximality}} and will be completed there.
\begin{theorem}[$\frA$ is maximal monotone]
	\label{thm:frA-maximal}
	The DLSS operator $\frA$ is maximal monotone from $\HtN$ to
	$\DHtN$: for every $\lambda>0$ and every $f\in\DHtN$ there
	exists a (unique) solution $v\in\dom(\frA)$ of
	\begin{equation}
		\label{eq:frA-resolvent}
		\lambda\rmJ v+\frA v\ni f,
	\end{equation}
	where $\rmJ$ is the duality map \eqref{eq:J-explicit}
	(the equivalence of the two formulations being
	Theorem~\ref{thm:Minty}).
\end{theorem}
Starting from Theorem~\ref{thm:frA-maximal}, we can give other
equivalent characterizations of $\frA$.
We first state a useful property.
\begin{lemma}[Commutation inequality]
	\label{le:commutationfrA}
	If $u\in \dom(\frA)$ then for every $\tau>0$
	$\sfH_\tau u\in \dom(\frA)$ and
	\begin{equation}
		\label{eq:commutation}
		\la\frA^\circ(u),\sfH_\tau v\ra 
		\le \la \frA^\circ(\sfH_\tau u),v\ra  
		\quad
		\text{for every }v\in \HtN,\ v\ge0.
	\end{equation}
\end{lemma}
\begin{proof}
	Since $\sfH_\tau u\in \PZO$ with $\Lap \sfH_\tau u\in \ZO$, it is immediate to check that 
	$\sfH_\tau u\in \dom(\frA)$.
	Using \eqref{eq:sfh-crucial} we compute
	\begin{align*}
		\la \frA^\circ(\sfH_\tau u),v\ra
		&=
		\int_\Omega \Lap \sfH_\tau u\,
		\Lap v\,\d x-
		\int_\Omega 
		\sfh(\sfH_\tau u, \Lap \sfH_\tau u)\,v\d x
		\\&=
		\int_\Omega \Lap u\,
		\Lap \sfH_\tau v\,\d x-
		\int_\Omega 
		\sfh(\sfH_\tau u,\sfH_\tau \Lap u)\,v\d x
		\\&\ge 
		\int_\Omega \Lap u\,
		\Lap \sfH_\tau v\,\d x-
		\int_\Omega 
		\sfh(u, \Lap u)\,\sfH_\tau v\d x
		= \la\frA^\circ(u),\sfH_\tau v\ra .
	\end{align*}
\end{proof}
\begin{proposition}
    \label{prop:ineq}
    A pair $(u,f)\in \HtN\times \DHtN$
    with $u\ge 0$
    belongs to the graph of $\frA$ if and only if 
    \begin{equation}
        \label{eq:cond3}
        \int_\Omega 
        \Lap u\,\Lap v\,\d x
        \le 
        \int_\Omega \frac{(\Lap v)^2}v\,u\,\d x+
        \langle f,u-v\rangle
        \quad 
        \text{for every $v\in \PZO$}
    \end{equation}
    or
    \begin{align}
    \label{eq:cond1}
        \langle f,u\rangle&=0,\\
        \label{eq:cond2}
        \int_\Omega 
        \Lap u\,\Lap v\,\d x
        &\le 
        \int_\Omega \frac{(\Lap v)^2}v\,u\,\d x-
        \langle f,v\rangle
        \quad 
        \text{for every $v\in \PZO$}.
    \end{align}
\end{proposition}
Notice that 
\eqref{eq:cond3} is equivalent to
\begin{equation}
	\label{eq:benilan-H2}
		\la \frA^\circ v,u\ra \le \la f,u-v\ra 
		\quad
		\text{for every }v\in \PZO.
\end{equation}
\begin{proof}
    It is clear that if $f\in \frA u$
    then $\langle f,u\rangle=0$ (i.e.~\eqref{eq:cond1}
    holds) and 
    for every $v\in \PZO$
    \begin{align*}
        0&\le 
        \langle f-\frA v,u-v\rangle
        =
        -\langle \frA v,u\rangle -\langle f,v\rangle\\
        &=
        -\int_\Omega \Lap u\,\Lap v\,\d x+
        \int_\Omega \frac{(\Lap v)^2}v\,u\,\d x-
        \langle f,v\rangle
    \end{align*}
    which shows \eqref{eq:cond2}.

    It is also immediate to check that 
    \eqref{eq:cond1} and \eqref{eq:cond2} 
    imply \eqref{eq:cond3}.

    Let us now suppose that \eqref{eq:cond3} holds, i.e.
    We fix $(v,g)\in \graph(\frA)$ 
	and we write $g= \frA^\circ v-\nu$, with
	$\nu$ positive measure concentrated on $\{\tilde v=0\}$.

    Since for every $\tau>0$ the function $v_{\tau}:=\sfH_\tau v$ belongs to
    $\PZO$,  \eqref{eq:cond3} and \eqref{eq:commutation}
	yield
    \begin{align*}
        \la \frA^\circ(v),\sfH_\tau u\ra
		\le
		\la \frA^\circ(v_\tau), u\ra
		\le
        \langle f,u-v_\tau\rangle.
    \end{align*}
    We can now pass to the limit as $\tau\downarrow0$
    observing that $\sfH_\tau u\to u$
	and $\sfH_\tau v\to v$ 
    in $\HtN$: we obtain
    \begin{align*}
        \la \frA^\circ(v),u\ra\le
		\langle f,u-v\rangle.
    \end{align*}
	Since $g= \frA^\circ v-\nu\le \frA^\circ v$, 
we obtain
\begin{displaymath}
    \langle g,u\rangle \le \langle f,u-v\rangle.
\end{displaymath}
Since $\langle g,v\rangle =0$, we conclude that 
\begin{displaymath}
    \langle f-g,u-v\rangle\ge0\quad\text{for every 
    $(v,g)\in \frA$}
\end{displaymath}
and therefore $f\in \frA u$, since $\frA$ is
maximal monotone {(Theorem~\ref{thm:frA-maximal})}.
\end{proof}
We can easily deduce that
$\frA$ is the unique maximal monotone extension
(in the duality between $\HtN$ and $\DHtN$)
of the core
\begin{equation}
    \frA^\infty:=
    \Big\{(v,g):
    v\in \mathrm C^\infty(\Omega)\cap \ZO,\ 
    \Lap v\in \ZO,\ \inf_\Omega v>0,\ 
    g=\Lap^2 v-\frac{(\Lap v)^2}v\Big\}.
\end{equation}
\begin{theorem}
    \label{thm:test}
    If the pair $(u,f)\in \HtN\times \DHtN$
    satisfies
    \begin{equation}
    \label{eq:test}
    \langle f-g,u-v\rangle\ge0\quad \text{for every }
    (v,g)\in \frA^\infty,
    \end{equation}
    then $f\in \frA u$.
    In particular,
    $\frA$ is the unique maximal monotone extension
    of $\frA^\infty$.
\end{theorem}
\begin{proof}
    It is clear that 
    the graph of $\frA^\infty$ is contained
    in the graph of $\frA$
    so that every element $(u,f)\in \frA$
    satisfies \eqref{eq:test}.
    In order to prove the converse property, stated in the Theorem, 
    let us suppose that $(u,f)$ satisfy
    \eqref{eq:test}.
    Since $\langle v,g\rangle=0 $
    and $u\in \HtN$
    \begin{displaymath}
        -\langle g,u\rangle =
        -\int_\Omega \Lap u\Lap v\,\d x+
        \int_\Omega u\frac{(\Lap v)^2}v\,\d x
    \end{displaymath}
    so that \eqref{eq:test} yields
    \begin{equation}
        \label{eq:cond4}
        \int_\Omega \Lap u\Lap v\,\d x
        \le 
        \int_\Omega u\frac{(\Lap v)^2}v\,\d x+
        \langle f,u-v\rangle
        \quad\text{for every $(v,g)\in \frA^\infty$.}
    \end{equation}
    A regularization via the Heat flow
    immediately yields that
    $(u,f)$ satisfies
    \eqref{eq:cond3} so that $f\in \frA u.$
\end{proof}

\subsection{The realization \texorpdfstring{$\sfA$}{A} of the DLSS operator in
\texorpdfstring{$L^2(\Omega)$}{L2}}
\label{subsec:sfA}
According to \eqref{eq:restriction} we introduce the following definition.
\begin{definition}
	\label{def:sfA}
	The operator
$\sfA:\H\rightrightarrows\H$, $\H=L^2(\Omega)$, is the restriction of
$\frA$ to $\H\times\H$: its graph is given by all pairs $(u,f)\in
L^2(\Omega)\times L^2(\Omega)$ such that $u\in\HtN$ and $f\in\frA u$.
In particular $\sfA$ is monotone in $\H$ and
$\dom(\sfA)\subset\dom(\frA)\subset\WpO$.
\end{definition} 

We collect in the next result the main equivalent characterizations of $\sfA$.
\begin{theorem}
	\label{thm:main-equivalent}
	Let $u\in \HtN,f\in L^2(\Omega)$, $u\ge 0$.
	The following properties are equivalent:
	\begin{enumerate}
		\item $f\in \sfA(u)$.
		\item $(u,f)$ satisfy the following two conditions
		\begin{align}
	\label{eq:DLSS-syst1}
	\int_\Omega \Lap u\Lap v\,\d x&\ge
			\int_\Omega fv\,\d x\quad\text{for every }
			v\in \PZO, \\
				\label{eq:DLSS-syst2}
			\frb(u,w)&=\int_\Omega f\, uw\,\d x
		\quad\text{for every }w\in
		\ZO .
\end{align}
	\item 
	$\mathsf D(u)<\infty$,
	\begin{equation}
		\label{eq:equivalent-better2}
		\Lap^2 u\ge f+\frac{(\Lap u)^2}u\restr{u>0},\quad 
		\int_\Omega f(x)u(x)\,\d x\ge 0.
	\end{equation}
	\item 
	$\mathsf D(u)<\infty$ and 
	there exists $g\in L^2_+(\Omega)$ such that 
	\begin{equation}
		\label{eq:equivalent-better4}
		f=\Lap^2 u-\frac{(\Lap u)^2}u\restr{u>0}-g,\quad 
		gu=0\text{ a.e.~in $\Omega$.}
	\end{equation}
	\item    
	$\mathsf D(u)<\infty$, $\Lap^2 u\in L^1(\Omega)$,
	\begin{equation}
		\label{eq:split}
		f=\Lap^2 u-\frac{(\Lap u)^2}u\quad \text{on }\{u>0\};
		\qquad
		\Lap^2 u=0,\quad f\le 0 \ \text{ on }\{u=0\}.
	\end{equation}
	Equivalently, the
	associated measure $\mu$ is absolutely continuous
	w.r.t.~$\mathscr L^d$ and 
	\begin{equation}
		\label{eq:mu-structure2}
		\mu=\frac{(\Lap u)^2}{\tilde u}\,
		\Leb d\restr{\{\tilde u>0\}}
		\;+\;(-f)\,\Leb d\restr{\{u=0\}} .
	\end{equation}
\end{enumerate}
\end{theorem}
\begin{proof}
	Condition 2 corresponds to Proposition \ref{le:DLSS-equivalence}.

	The fact that $\sfD(u)<\infty$ is a necessary condition for
	$u\in \dom(\sfA)$ is a consequence of Proposition~\ref{prop:defect-variational}.
	
	Condition 3 corresponds to Corollary \ref{cor:easy2}.

	Condition 5 can be obtained by 
	Theorem \ref{thm:vacuum-trace}. Notice that 
	in the present case the measure $\mu$ is absolutely continuous w.r.t.~$\mathscr L^d$, so that 
	$\Lap^2 u=f+\mu \in L^1(\Omega)$.

	Condition 4 clearly implies Condition 3: since $g\ge0$ the
	identity \eqref{eq:equivalent-better4} yields
	$\Lap^2u\ge f+(\Lap u)^2u^{-1}\restr{u>0}$, while
	$gu=0$, $\int_\Omega \Lap^2u\,u\,\d x=\int_\Omega(\Lap u)^2\,\d x$
	and $\Lap u=0$ $\Leb d$-a.e.~on $\{u=0\}$
	give $\int_\Omega fu\,\d x=0$. Condition 3 therefore holds, and
	Corollary~\ref{cor:easy2} gives $f\in \sfA(u)$ as before.
	Its necessity immediately follows by Condition~5.
\end{proof}

The last characterization of the previous Theorem yields a complete description of the
multivalued structure of $\sfA$: all the elements of $\sfA u$
coincide outside the vacuum region, and on the vacuum they span the
normal cone to the constraint $u\ge0$.
We denote by
\begin{equation}
	\label{eq:normal-cone}
	\partial I_{L^2_+}(u):=\Big\{g\in L^2(\Omega):\ g=0\
	\text{a.e.~on }\{u>0\},\quad g\le0\
	\text{a.e.~on }\{u=0\}\Big\}
\end{equation}
the subdifferential at a nonnegative $u$ of the indicator function
$I_{L^2_+}$ of the convex cone $\big\{v\in L^2(\Omega):v\ge0\big\}$.
\begin{corollary}[Minimal section and normal cone]
	\label{prop:minimal-section}
	A nonnegative $u\in L^2_+(\Omega)$ 
	belongs to $\dom(\sfA)$ iff
	\begin{equation}
		\label{eq:domainA}
		\sfD(u)<\infty,\quad
		\Lap^2 u\in L^1(\Omega),\quad
		\sfA^\circ u:=\Lap^2 u-\frac{(\Lap u)^2}u\in L^2(\Omega).
	\end{equation}
	In this case
	\begin{enumerate}[\rm (a)]
	\item 
		$\sfA^\circ(u)$ is the element of minimal $L^2$-norm 
		in $\sfA(u)$:
		we call it 
		the \emph{minimal ($L^2$) DLSS operator}.
	\item Any element of $\sfA u$ coincides 
		with $\sfA^\circ(u)$ $\Leb
		d$-a.e.~on $\{u>0\}$; 
	\item $\sfA u=\sfA^\circ u+\partial I_{L^2_+}(u)$: equivalently,
		$f\in\sfA u$ if and only if
		\begin{equation}
			\label{eq:sfA-structure}
			f=\sfA^\circ u\ \text{a.e.~on }\{u>0\},
			\qquad
			f\le0\ \text{a.e.~on }\{u=0\} .
		\end{equation}
	\end{enumerate}
	In particular $\sfA u$ is a singleton if and only if
	$\Leb d(\{u=0\})=0$, and the defect $\mu_0$ vanishes exactly at
	the minimal section $f=\sfA^\circ u$.
\end{corollary}

The crucial result for $\sfA$ is the following.
\begin{theorem}
	\label{thm:main-maximal}
	The operator $\sfA$ is maximal {monotone} in $L^2(\Omega)$.
\end{theorem}
{As for Theorem~\ref{thm:frA-maximal}, the proof will be
presented in {Section \ref{sec:maximality}}. By
Proposition~\ref{prop:restriction-maximality}, the maximality of
$\sfA$ in $L^2(\Omega)$ amounts to show that for every $f\in
L^2(\Omega)$ there exists $u\in \dom(\frA)$ such that $u+\frA u\ni
f$, i.e.~$(u,f-u)$ belongs to the graph of $\frA$.}

\medskip\noindent
We now show that $\sfA$ is the $L^2$-realization uniquely determined
by the smooth core $\frA^\infty$, {with no a priori regularity
required of the competing extension}. {The key point is the
following propagation property: monotonicity against the smooth core
already forces monotonicity against the whole minimal section, for an
arbitrary nonnegative datum $u$. Its proof mimics that of
Proposition~\ref{prop:ineq}, with the heat flow moved onto $u$ by
self-adjointness and the perspective term controlled by Jensen's
inequality.

We first state the crucial commutation inequality, whose proof 
follows by the very same argument of Lemma
\ref{le:commutationfrA}.
\begin{lemma}[Commutation inequality for $\sfA$]
	\label{le:commutationsfA}
	If $u\in \dom(\sfA)$ then for every $\tau>0$
	$\sfH_\tau u\in \dom(\sfA)$ and
	\begin{equation}
		\label{eq:commutationsfA}
		\int_\Omega \sfA^\circ(u)\,\sfH_\tau v\,\d x
		\le 
		\int_\Omega \sfA^\circ(\sfH_\tau u)\, v\,\d x
		\quad
		\text{for every }v\in L^2(\Omega),\ v\ge0.
	\end{equation}
	Equivalently,
	\begin{equation}
		\label{eq:nice}
		\sfA^\circ(\sfH_\tau u)\ge 
		\sfH_\tau(\sfA^\circ u)\quad\text{for every }
		u\in \dom(\sfA).
	\end{equation}
\end{lemma}
\begin{lemma}[From the core to the minimal section]
	\label{le:core-to-min}
	Let $u\in L^2_+(\Omega)$ and $f\in L^2(\Omega)$ satisfy
	\begin{equation}
		\label{eq:core-monotone}
		\int_\Omega\big(f-\sfA^\circ v\big)(u-v)\,\d x\ge0
		\qquad\text{for every }v\in
		\rmC^\infty(\Omega)\cap \PZO,\ \Lap v\in \ZO.
	\end{equation}
	Then
	\begin{equation}
		\label{eq:min-monotone}
		\int_\Omega\big(f-\sfA^\circ v\big)(u-v)\,\d x\ge0
		\qquad\text{for every }v\in\dom(\sfA).
	\end{equation}
\end{lemma}
\begin{proof}
	Fix $v\in\dom(\sfA)$ and, for $\tau>0$, set $v_\tau:=\sfH_\tau v$.
	Since $v\ge0$, the Neumann heat kernel on the bounded connected
	$\Omega$ is bounded below by a positive constant, so
	$v_\tau\in\ZO$ is uniformly positive and, together with
	$\Lap v_\tau=\sfH_\tau\Lap v\in\ZO$, this gives
	$(v_\tau,\sfA^\circ v_\tau)\in\frA^\infty$. Using
	$\la\sfA^\circ v_\tau,v_\tau\ra=0$
	(mass identity~\eqref{eq:defect-global}),
	\eqref{eq:core-monotone} yields
	\begin{equation}
		\label{eq:core-step}
		\int_\Omega u\,\Lap^2v_\tau\,\d x-
		\int_\Omega u\,\frac{(\Lap v_\tau)^2}{v_\tau}\,\d x=
		\int_\Omega 
		\sfA^\circ v_\tau \,u\,\d x \le
		\int_\Omega f(u-v_\tau)\, \d x .
	\end{equation}
	By \eqref{eq:commutationsfA} we get
	\begin{equation*}
		\int_\Omega\sfH_\tau u\;\sfA^\circ v\,\d x
		\;\le\;\la f,\,u-\sfH_\tau v\ra .
	\end{equation*}
	As $\tau\downarrow0$, $\sfH_\tau u\to u$ and $\sfH_\tau v\to v$
	strongly in $L^2(\Omega)$, so that, $\sfA^\circ v\in L^2(\Omega)$
	being fixed, $\int_\Omega\sfA^\circ v\,u\,\d x\le\la f,u-v\ra$.
	Recalling $\int_\Omega\sfA^\circ v\,v\,\d x=0$ we obtain
	\eqref{eq:min-monotone}.
\end{proof}}
\begin{corollary}[$\sfA$ is determined by the core]
	\label{cor:test-L2}
	$\sfA$ is the unique maximal monotone operator in $L^2(\Omega)$
	which extends $\frA^\infty$ and {whose domain consists of
	nonnegative functions ($\dom(\sfA)\subseteq L^2_+(\Omega)$)}.
\end{corollary}
\begin{proof}
	Let $\sfB$ be such an operator and $(u,f)\in\graph(\sfB)$;
	then $u\ge0$, and by monotonicity $(u,f)$ satisfies~\eqref{eq:core-monotone}, since
	$\frA^\infty\subseteq\graph(\sfB)$. Lemma~\ref{le:core-to-min}
	then provides \eqref{eq:min-monotone}.
	Since $\sfA^\circ$ is a principal section of $\sfA$ 
	\cite[Chap.~2, Prop. 2.7]{Brezis73} we deduce that  
	$f\in \sfA(u)$, so that $\graph{(\sfB)}\subset 
	\graph{(\sfA)}$ and the maximality of $\sfA$
	forces equality.
\end{proof}

The same technique extends the key estimate E4 of
Section~\ref{subsec:positive}, i.e.~\eqref{eq:crucial-estimate}, to
the whole graph of~$\sfA$: this will be the engine of the entropy
regularization of the DLSS flow
(Proposition~\ref{prop:entropy-regularization}).
\begin{lemma}[Entropy identity on the graph]
	\label{le:entropy-identity}
	Let $(u,f)\in\graph(\sfA)$.
	  Then
	$u\ln u\in L^2(\Omega)$ and
	\begin{equation}
		\label{eq:entropy-identity}
		\int_\Omega f\,u\ln u\,\d x=
		\int_\Omega\Big((\Lap u)^2+
		\Lap u\,\frac{|\nabla u|^2}{u}\Big)\,\d x
		\;\ge\;c_d\int_\Omega|\rmD^2u|^2\,\d x ,
	\end{equation}
	with $c_d>0$ as in \eqref{eq:crucial-estimate} (the integrand
	$\Lap u\,|\nabla u|^2/u$ being defined as $0$ on $\{u=0\}$,
	where $\Lap u=0$ and $\nabla u=0$ a.e.).
\end{lemma}
\begin{proof}
	{Since $u\in\dom(\sfA)\subset\HtN\subset H^1(\Omega)$, the
	same Gagliardo--Nirenberg interpolation used in the proof of
	Proposition~\ref{prop:apriori-estimate3} (inequality
	\eqref{eq:GN} with $\sigma=2+4/d$, admissible in every dimension
	since $2+4/d\le2^*$), together with the elementary bound
	$|a\ln a|\le C(1+a^{\sigma/2})$, gives
	\begin{equation}
		\label{eq:ulnu-GN}
		\|u\ln u\|_{L^2(\Omega)}\le
		C\big(1+\|u\|_{L2}^{2/d}\,\|\nabla u\|_{L2}\big)<+\infty ,
	\end{equation}
	so that $u\ln u\in L^2(\Omega)$ with no boundedness assumption on
	$u$.} For $\tau,\eps\in(0,1)$ the function
	$w_{\tau,\eps}:=\ln(\sfH_\tau u+\eps)$ belongs to $\ZO$
	(as in Step~0 above), so \eqref{eq:DLSS-syst2} gives, with
	$u_\tau:=\sfH_\tau u$,
	\begin{equation*}
		\int_\Omega f\,u\ln(u_\tau+\eps)\,\d x=
		\frb\big(u,\ln(u_\tau+\eps)\big)=
		\int_\Omega\Big(
		\frac{u\,\Lap u\,\Lap u_\tau}{u_\tau+\eps}
		-\frac{u\,\Lap u\,|\nabla u_\tau|^2}{(u_\tau+\eps)^2}
		+\frac{2\Lap u\,\nabla u\cdot\nabla
		u_\tau}{u_\tau+\eps}\Big)\d x .
	\end{equation*}
	Letting $\tau\down0$ at fixed $\eps$ --- by the strong
	convergence $u_\tau\to u$ in $\HtN$ together with the quartic
	bound \eqref{eq:quartic-general} (which keep $\nabla u_\tau$
	bounded in $L^4$ and $|\nabla u_\tau|^2$ bounded in $L^2$), while
	the denominators satisfy $u_\tau+\eps\ge\eps$ --- we obtain
	\begin{equation*}
		\int_\Omega f\,u\ln(u+\eps)\,\d x=
		\int_\Omega\Big(
		\frac{u\,(\Lap u)^2}{u+\eps}
		-\frac{u\,\Lap u\,|\nabla u|^2}{(u+\eps)^2}
		+\frac{2\Lap u\,|\nabla u|^2}{u+\eps}\Big)\d x .
	\end{equation*}
	As $\eps\down0$, all the integrands vanish a.e.~on $\{u=0\}$
	and converge pointwise on $\{u>0\}$ to $(\Lap u)^2$,
	$-\Lap u|\nabla u|^2/u$ and $2\Lap u|\nabla u|^2/u$
	respectively, with the integrable majorants $(\Lap u)^2$ and
	$|\Lap u|\,|\nabla u|^2/u$ (recall
	$\||\nabla u|^2/u\|_{2}\le3\tnrm{\Lap u}$ by
	\eqref{eq:quartic-general}); on the left-hand side,
	{the elementary bound
	$|u\ln(u+\eps)|\le|u\ln u|+(\ln2)\,u$, valid for every
	$\eps\in(0,1]$, provides the $\eps$-uniform majorant
	$|u\ln u|+(\ln2)\,u\in L^2(\Omega)$ (by \eqref{eq:ulnu-GN} and
	$u\in L^2(\Omega)$), which is integrable against $f\in
	L^2(\Omega)$}. This proves the identity in
	\eqref{eq:entropy-identity}. For the lower bound, we apply
	\eqref{eq:crucial-estimate} to $u+\eps\in\PXO$:
	\begin{equation*}
		\int_\Omega\Big((\Lap u)^2+
		\Lap u\,\frac{|\nabla u|^2}{u+\eps}\Big)\,\d x=
		\fra\big(u+\eps,(u+\eps)\ln(u+\eps)\big)\ge
		c_d\int_\Omega|\rmD^2u|^2\,\d x ,
	\end{equation*}
	and let $\eps\down0$ with the same majorants.
\end{proof}

{%
\subsection{A very weak form of the {minimal} DLSS operator}
\label{subsec:very-weak}
Proposition~\ref{prop:defect-variational} shows that
$\dom(\frA)\subset\WpO$: the class $\WpO$ of \eqref{eq:defW} is thus
the natural environment
for the weakest realization of the DLSS operator. Indeed, if
$u\in\WpO$ then Lemma~\ref{le:sqrt-calculus} shows that $u\in
D_{L1}(\Lap)$, with $\Lap u\in L^1(\Omega)$ vanishing
$\Leb d$-a.e.~on $\{u=0\}$, and Proposition~\ref{prop:quartic} gives
$\sfD(u)\le80\int_\Omega(\Lap\sqrt u)^2\,\d x<+\infty$: both terms of
the form \eqref{eq:form} then retain a meaning when the second
argument is taken in the test space $\ZO$, and we can extend $\fra$
to $\WpO\times\ZO$ by setting
\begin{equation}
	\label{eq:form-veryweak}
	\fra(u,\zeta):=\int_\Omega
	\Big(\Lap u\,\Lap\zeta-
	\sfh(u,\Lap u)\,\zeta\Big)\,\d x,
	\qquad u\in\WpO,\ \zeta\in\ZO,
\end{equation}
with $\sfh$ as in \eqref{eq:defh}, so that
$\int_\Omega \sfh(u,\Lap u)\,\d x=\sfD(u)$ and
\begin{equation}
	\label{eq:veryweak-bound}
	|\fra(u,\zeta)|\le
	\|\Lap u\|_{L^1(\Omega)}\,\|\Lap\zeta\|_{L^\infty(\Omega)}+
	\sfD(u)\,\|\zeta\|_{L^\infty(\Omega)} .
\end{equation}
In particular $\zeta\mapsto\fra(u,\zeta)$ defines an element of the
dual of $\ZO$ for every $u\in\WpO$: a \emph{very weak} realization of
the {minimal} DLSS operator, defined on the whole class $\WpO$, with values in
the dual of a space of smooth test functions.
{Notice that $\fra$ extends the minimal section $\sfA^\circ$ of
Proposition~\ref{prop:minimal-section} since 
\begin{equation}
	\label{eq:fra-minimal}
	\fra(u,\zeta)=
	\int_\Omega \sfA^\circ u\,\zeta\,\d x
	\qquad\text{for every }u\in\dom(\sfA),\ \zeta\in\ZO .
\end{equation}
}

At this level no genuine monotonicity survives: differences $u-v$ of
elements of $\WpO$ are not admissible test functions in
\eqref{eq:form-veryweak}. The core of the monotonicity property is
however retained by the following \emph{asymmetric} inequality, which
pairs an arbitrary $u\in\WpO$ with a smooth and uniformly positive
$v\in\PZO$.
\begin{lemma}[Asymmetric inequality]
	\label{le:asymmetric}
	For every $u\in\WpO$ and every $v\in\PZO$, both $\fra(u,v)$ and
	$\fra(v,u)$ are well defined and
	\begin{equation}
		\label{eq:asymmetric}
		\fra(u,v)+\fra(v,u)=
		-\int_{\{u>0\}}
		\Big(\sqrt{\tfrac vu}\,\Lap u-
		\sqrt{\tfrac uv}\,\Lap v\Big)^{\!2}\,\d x\;\le\;0 .
	\end{equation}
\end{lemma}
\begin{proof}
	$\fra(u,v)$ is well defined by \eqref{eq:veryweak-bound}, since
	$\PZO\subset\ZO$; concerning
	$\fra(v,u)=\int_\Omega\big(\Lap v\,\Lap u-(\Lap
	v)^2u/v\big)\,\d x$, it is sufficient to observe that $\Lap v\in
	L^\infty(\Omega)$, $\Lap u\in L^1(\Omega)$, and $(\Lap
	v)^2/v\in L^\infty(\Omega)$, $u\in L^1(\Omega)$.
	Adding the two expressions we get
	\begin{equation*}
		\fra(u,v)+\fra(v,u)=
		\int_\Omega\Big(2\Lap u\,\Lap v-
		\sfh(u,\Lap u)\,v-\frac{(\Lap v)^2}v\,u\Big)\,\d x ,
	\end{equation*}
	where all three terms of the integrand belong to $L^1(\Omega)$.
	$\Leb d$-a.e.~on $\{u=0\}$ the integrand vanishes, since $\Lap
	u=0$ there by Lemma~\ref{le:sqrt-calculus}(a,c); on $\{u>0\}$ it
	coincides with
	$-\big(\sqrt{v/u}\,\Lap u-\sqrt{u/v}\,\Lap v\big)^2$, whence
	\eqref{eq:asymmetric}.
\end{proof}
\begin{remark}[The core of monotonicity]
	\label{rem:asymmetric}
	Since $\fra(v,v)=0$ and, formally, also $\fra(u,u)=0$ --- the
	diagonal identity encoded in {\rm(A.3)} and in \eqref{eq:cond1}
	--- the quantity $-\fra(u,v)-\fra(v,u)$ is the formal expansion
	of the monotonicity gap
	$\la\frA u-\frA v,u-v\ra=\fra(u,u-v)+\fra(v,v-u)$,
	cf.~\eqref{eq:monotone}: Lemma~\ref{le:asymmetric} thus extends
	the monotonicity of the DLSS operator, in an asymmetric form, up
	to the very weak level, where only one of the two arguments is
	smooth and uniformly positive.
	This inequality will inspire a similar property 
	which will be used to characterize the semigroup
	trajectories among weak solutions of the evolution
	equation, through a B\'enilan-type argument (see
	Section~\ref{subsec:very-weak-solutions}).
\end{remark}
}
{\section{The DLSS evolution equation}
\label{sec:evolution}
In this section we collect the results on the evolution equation
driven by the DLSS operator. They all stem from the maximality of
$\sfA$ in $L^2(\Omega)$ (Theorem~\ref{thm:main-maximal}). We first
obtain the generation of a contraction semigroup, together with the
natural notions of strong and weak (integral) solutions
(Section~\ref{subsec:generation}); we then provide sufficient
conditions in distributional form, well adapted to approximation
schemes (Section~\ref{subsec:distributional}); finally we introduce a
weaker formulation, only involving the class $\WpO$ of
\eqref{eq:defW} --- the regularity that the flow itself guarantees at
a.e.~time, for every initial datum and in every dimension --- and we
show that it still characterizes the semigroup trajectories
(Section~\ref{subsec:very-weak-solutions}): the asymmetric inequality
of Lemma~\ref{le:asymmetric} is the key ingredient.

\subsection{Generation of the DLSS semigroup: strong and B\'enilan integral solutions}
\label{subsec:generation}}
\begin{definition}
    \label{def:strong-solution}
    A continuous curve $u:(0,+\infty)\to L^2(\Omega)$
    is a strong solution to the DLSS equation if
    it is locally Lipschitz in $(0,+\infty)$ and
    \begin{equation}
        \label{eq:strong-solution1}
        \frac{\d}{\d t}
        u_t\in -\sfA u_t\quad\text{for a.e.~$t>0.$}
    \end{equation}
\end{definition}
\begin{definition}
    \label{def:benilan-solution}
    A continuous curve $u:(0,+\infty)\to L^2(\Omega)$
    is a {B\'enilan solution} (or integral solution in the sense of
    B\'enilan)
    to the DLSS equation if
    its $L^2$-norm is constant in time and
    \begin{equation}
        \label{eq:benilan-solution}
        \partial_t \int_\Omega u_t v\,\d x
        \ge
        \int_\Omega g\,u_t\,\d x
        \quad
        \text{in $\mathscr D'(0,+\infty)$,\quad  for every
        $(v,g)\in \sfA.$}
    \end{equation}
\end{definition}
{%
\begin{remark}[{B\'enilan solutions} and the B\'enilan inequalities]
	\label{rem:weak-vs-Benilan}
	If the $L^2$-norm of a continuous curve $u$ is constant in time
	then, for every $v\in L^2(\Omega)$,
	\begin{displaymath}
		\frac{\d}{\d t}\,\frac12\tnrm{u_t-v}^2=
		\frac{\d}{\d t}\Big(\frac12\tnrm{u_t}^2+
		\frac12\tnrm v^2-\int_\Omega u_tv\,\d x\Big)=
		-\,\partial_t\int_\Omega u_tv\,\d x
		\quad\text{in }\mathscr D'(0,+\infty);
	\end{displaymath}
	since $\int_\Omega g\,v\,\d x=0$, \eqref{eq:benilan-solution} is precisely the B\'enilan
	system of inequalities \eqref{eq:Benilan} for the operator
	$\sfA$: 
	{B\'enilan solutions} according to
	Definition~\ref{def:benilan-solution} are exactly the integral
	solutions with constant $L^2$-norm.
	Let us stress that the constancy of the norm is not implied by
	\eqref{eq:benilan-solution}: it encodes, at the lowest regularity
	level, the diagonal identity $\la f,u\ra=0$ of
	\eqref{eq:defect-global}, i.e.~the conservation of the total
	mass $\int_\Omega\varrho_t\,\d x=\tnrm{u_t}^2$ of the DLSS equation
	in the original variables $\varrho=u^2$.
\end{remark}
The domain of $\sfA$ is dense in the cone of nonnegative functions:
\begin{lemma}[Closure of the domain]
	\label{le:domain-closure}
	$\overline{\dom(\sfA)}=
	L^2_+(\Omega):=\big\{u\in L^2(\Omega):\ u\ge0\
	\text{a.e.~in }\Omega\big\}$.
\end{lemma}
\begin{proof}
	The inclusion $\subset$ is clear, since
	$\dom(\sfA)\subset L^2_+(\Omega)$ by {\rm(A.1)} and
	$L^2_+(\Omega)$ is closed. Conversely, let
	$u_0\in L^2_+(\Omega)$ and, for $\eps>0$, set
	$v_\eps:=\sfH_\eps u_0+\eps$. Then $v_\eps\in\PZO$ with
	$\Lap^2v_\eps:=\Lap(\Lap v_\eps)\in L^2(\Omega)$ and
	$\sfA^\circ v_\eps=\Lap^2v_\eps-(\Lap v_\eps)^2/v_\eps\in L^2(\Omega)$.
	Hence
	$v_\eps\in\dom(\sfA)$ and $v_\eps\to u_0$ in $L^2(\Omega)$ as
	$\eps\down0$.
\end{proof}
Combining Theorem~\ref{thm:main-maximal} with the general theory
recalled in Section~\ref{subsec:maximal-general} we obtain the
generation result.}
\begin{theorem}
    \label{thm:generation}
    For every nonnegative $u_0\in L^2_+(\Omega)$
    there exists a unique B\'enilan integral solution~$u$ to the DLSS
    equation
    satisfying $\lim_{t\downarrow0}u_t=u_0$
    in $L^2(\Omega)$.
    The map
    $\sfS_t:u_0\mapsto u_t $
    defines a continuous semigroup of contractions in $L^2(\Omega)$.

    If $u_0\in \dom (\sfA)$
    then $u$ is a strong solution,
    $u_t\in \dom (\sfA)$
    for every $t\ge0$,
    $\|\sfA^\circ u_t\|\le \|\sfA^\circ u_0\|$,
    and $u$ is Lipschitz in $[0,+\infty).$
	Moreover, $u$ is right differentiable in $L^2$ everywhere with
	$$\partial_t^+ u(t,\cdot)=\lim_{h\downarrow0}\frac {u(t+h)-u(t)}h
		\text{ strongly in }L^2(\Omega),$$
	and
	\begin{equation}
		\label{eq:right}
		\partial_t^+ u+\Lap^2 u-\frac{(\Lap u)^2}u=0
		\quad \text{for every }t>0.	
	\end{equation}
\end{theorem}
{%
\begin{proof}
	By Theorem~\ref{thm:main-maximal} the operator $\sfA$ is maximal
	monotone in $L^2(\Omega)$ and, by
	Lemma~\ref{le:domain-closure},
	$\overline{\dom(\sfA)}=L^2_+(\Omega)$. The general results
	recalled in Section~\ref{subsec:maximal-general} then provide,
	for every $u_0\in L^2_+(\Omega)$, a unique integral solution
	$u_t:=\sfS_tu_0$ in the sense of \eqref{eq:Benilan} with
	$u_t\to u_0$ as $t\down0$, the semigroup and contraction
	properties of the maps $\sfS_t$ in $L^2_+(\Omega)$, and, when
	$u_0\in\dom(\sfA)$, all the properties of strong solutions
	listed in the statement, cf.\
	\eqref{eq:Cauchy-A}--\eqref{eq:Benilan-integral}.

	By Remark~\ref{rem:weak-vs-Benilan}, in order to identify
	integral and B\'enilan solutions it remains to check that
	$t\mapsto\tnrm{u_t}$ is constant along the semigroup. If
	$u_0\in\dom(\sfA)$ then $u$ is Lipschitz and for a.e.~$t>0$ we
	have $\partial_tu_t=-f_t$ with $f_t\in\sfA u_t$, so that
	\begin{displaymath}
		\frac{\d}{\d t}\,\frac12\tnrm{u_t}^2=
		\int_\Omega \partial_tu_t\,u_t\,\d x=
		-\int_\Omega f_t\,u_t\,\d x=0
	\end{displaymath}
	by \eqref{eq:defect-global} (recall
	\eqref{eq:pairing-bridge}); the general case follows by
	approximating $u_0$ with $u_0^n\in\dom(\sfA)$ and using the
	contraction property
	$\tnrm{\sfS_tu_0^n-\sfS_tu_0}\le\tnrm{u_0^n-u_0}$.
\end{proof}}
The canonicity of Corollary~\ref{cor:test-L2} has a dynamic
counterpart: among all contraction semigroups on $L^2_+(\Omega)$, the
semigroup $\sfS$ is uniquely determined by its behaviour at $t=0^+$ on
the smooth core $\frA^\infty$.
\begin{corollary}[Uniqueness of the contraction semigroup]
	\label{cor:unique-semigroup}
	Let $(T_t)_{t\ge0}$ be a strongly continuous semigroup of
	contractions on $L^2_+(\Omega)$ such that
	\begin{equation}
		\label{eq:tangency}
		\lim_{t\down0}\frac{T_tv-v}t=-\Lap^2 v+\frac{(\Lap v)^2}v
		\quad\text{in }L^2(\Omega)
		\qquad\text{for every }v\in \rmC^\infty(\Omega)\cap \PZO,\ \Lap v\in \ZO.
	\end{equation}
	Then $T_t=\sfS_t$ for every $t\ge0$.
\end{corollary}
\begin{proof}
	By the generation theorem of K\=omura and Crandall--Pazy
	\cite[Thm.~4.1]{Brezis73} there exists a (unique) maximal monotone
	operator $\sfB$ in $L^2(\Omega)$ with
	$\overline{\dom(\sfB)}=L^2_+(\Omega)$ whose semigroup is $T$; in
	particular $\dom(\sfB)\subset L^2_+(\Omega)$. Let
	$(v,g)\in\frA^\infty$. By \eqref{eq:tangency}
	$\limsup_{t\down0}\tnrm{T_tv-v}/t<\infty$, so that
	$v\in\dom(\sfB)$, and the right derivative of $t\mapsto T_tv$ at
	$t=0$, which exists by \eqref{eq:tangency}, coincides with
	$-\sfB^\circ v$ \cite[Chap.~III]{Brezis73}: hence
	$g=\sfB^\circ v\in\sfB v$ and
	$\graph(\frA^\infty)\subset\graph(\sfB)$.
	Corollary~\ref{cor:test-L2} then yields $\sfB=\sfA$, and the two
	semigroups coincide.
\end{proof}
In particular, $\sfS$ is the unique contraction semigroup on
$L^2_+(\Omega)$ extending the classical solutions of the DLSS
equation emanating from smooth, uniformly positive initial data.
Indeed, any semigroup extending such classical local-in-time
evolutions satisfies \eqref{eq:tangency}, since a classical solution
starting from $(v,g)\in\frA^\infty$ is differentiable at $t=0$ with
derivative $-g$; and $\sfS$ itself extends them, since a classical
solution is in particular a strong solution in the sense of
Definition~\ref{def:strong-solution} and therefore coincides with
$\sfS_tv$, by the uniqueness in Theorem~\ref{thm:generation}.
\subsection{The implicit Euler scheme and the vacuum}
\label{subsec:Euler}
The maximality of $\sfA$ makes the implicit Euler scheme for
\eqref{eq:intro-inclusion} well posed: given $u_0\in L^2_+(\Omega)$ and a
step size $\tau>0$, the discrete solutions
$U^n_\tau\in\dom(\sfA)$, $n\ge1$, are defined recursively by
\begin{equation}
	\label{eq:Euler-scheme}
	U^0_\tau:=u_0,\qquad
	\frac{U^n_\tau-U^{n-1}_\tau}\tau+\sfA U^n_\tau\ni0,
	\quad\text{i.e.}\quad
	U^n_\tau:=(\rmI+\tau\sfA)^{-1}U^{n-1}_\tau ,
\end{equation}
each step being uniquely solvable by Theorem~\ref{thm:main-maximal},
whose proof in Section~\ref{sec:maximality} constructs the solution
through the variational inequalities of
Theorem~\ref{thm:existence}. By the
exponential formula \eqref{eq:exponential},
\begin{equation}
	\label{eq:Euler-convergence}
	\lim_{n\to\infty}U^n_{t/n}=\sfS_tu_0
	\quad\text{in }L^2(\Omega),\qquad t>0,
\end{equation}
with the classical error estimate of Crandall--Liggett
\cite[Cor.~4.4]{Brezis73}
\begin{equation}
	\label{eq:Euler-rate}
	\tnrm{U^n_\tau-\sfS_{t}u_0}\le
	2\sqrt{t\tau}\,\tnrm{\sfA^\circ u_0},
	\qquad t=n\tau,
\end{equation}
whenever $u_0\in\dom(\sfA)$. Thus \eqref{eq:Euler-scheme} is a genuine
approximation scheme for the DLSS equation. Notice that it preserves nonnegativity for free just because
$\dom(\sfA)\subset L^2_+(\Omega)$. In fact a single step retains more,
namely the two structural features of the continuous flow: it is an
equation and not an inclusion, and it cannot enlarge the vacuum.
\begin{proposition}[The Euler step is defect-free and does not enlarge the
	vacuum]
	\label{prop:Euler-vacuum}
	Let $u_0\in L^2_+(\Omega)$, let $\tau>0$ and let $(U^n_\tau)_{n\ge0}$ be
	given by \eqref{eq:Euler-scheme}. Then for every $n\ge1$:
	\begin{enumerate}[\rm (i)]
	\item $U^n_\tau\in\dom(\sfA^\circ)$ and the inclusion in
	\eqref{eq:Euler-scheme} is in fact the equation
	\begin{equation}
		\label{eq:Euler-equation}
		\frac{U^n_\tau-U^{n-1}_\tau}\tau+\sfA^\circ U^n_\tau=
		\frac{U^n_\tau-U^{n-1}_\tau}\tau
		+\Lap^2U^n_\tau-\frac{(\Lap U^n_\tau)^2}{U^n_\tau}=0
		\quad\text{in }L^2(\Omega) :
	\end{equation}
	only the minimal operator $\sfA^\circ$ is involved, the defect
	measure $\nu=\mu_0$ vanishes, and no reaction term is created on the
	vacuum.
	\item Up to $\Leb d$-negligible sets,
	\begin{equation}
		\label{eq:Euler-vacuum}
		\{U^n_\tau=0\}\subset\{U^{n-1}_\tau=0\}\subset\cdots\subset\{u_0=0\},
	\end{equation}
	i.e.~one step of the implicit Euler scheme cannot enlarge the vacuum
	region by a set of positive Lebesgue measure:
	$\Leb d(\{U^n_\tau=0\}\setminus\{U^{n-1}_\tau=0\})=0$.
	\item The discrete mass
	$\int_\Omega\varrho^n_\tau\,\d x=\tnrm{U^n_\tau}^2$,
	$\varrho^n_\tau:=(U^n_\tau)^2$, is nonincreasing in $n$ and
	\begin{equation}
		\label{eq:Euler-mass}
		\tnrm{U^n_\tau-U^{n-1}_\tau}^2=
		\tnrm{U^{n-1}_\tau}^2-\tnrm{U^n_\tau}^2.
	\end{equation}
	If moreover $u_0\in\dom(\sfA)$, then
	$\tnrm{U^n_\tau-U^{n-1}_\tau}\le\tau\,\tnrm{\sfA^\circ u_0}$ and
	\begin{equation}
		\label{eq:Euler-mass-loss}
		0\le\tnrm{u_0}^2-\tnrm{U^N_\tau}^2=
		\sum_{n=1}^N\tnrm{U^n_\tau-U^{n-1}_\tau}^2
		\le t\,\tau\,\tnrm{\sfA^\circ u_0}^2 ,
		\qquad t=N\tau :
	\end{equation}
	the mass lost by the scheme is of order $\tau$ on every bounded time
	interval, in accordance with the conservation of the $L^2$-norm
	along the flow.
	\end{enumerate}
\end{proposition}
\begin{proof}
	We argue by induction on $n\ge1$, the datum $U^0_\tau=u_0$ being
	nonnegative. Assume that $g:=U^{n-1}_\tau$ is nonnegative and set
	$u:=U^n_\tau$, $f:=(g-u)/\tau\in\sfA u$, so that $u$ solves the
	resolvent equation $u+\tau\sfA u\ni g$ with a nonnegative datum.
	On $\{u=0\}$ we have $f=g/\tau\ge0$ a.e., and $f\le0$ a.e.~by
	Theorem~\ref{thm:vacuum-trace}(ii): hence $f=g=0$ a.e.~on
	$\{u=0\}$, which is the first inclusion of \eqref{eq:Euler-vacuum}.
	Since the minimal section is the only element of $\sfA u$ vanishing
	on the vacuum, Corollary~\ref{prop:minimal-section} yields
	$f=\sfA^\circ u$ and $\mu_0=0$, i.e.~(i). Finally
	$\dom(\sfA)\subset L^2_+(\Omega)$ gives $U^n_\tau\ge0$ and the
	induction proceeds.

	Concerning (iii), we test \eqref{eq:Euler-equation} by $U^n_\tau$ and
	use the identity $\int_\Omega\sfA^\circ U^n_\tau\,U^n_\tau\,\d x=0$ of
	\eqref{eq:defect-global}, obtaining
	$\tnrm{U^n_\tau}^2=\int_\Omega U^{n-1}_\tau\,U^n_\tau\,\d x$; the
	Cauchy--Schwarz inequality gives
	$\tnrm{U^n_\tau}\le\tnrm{U^{n-1}_\tau}$ and expanding the square
	yields the identity in \eqref{eq:Euler-mass}.
	When $u_0\in\dom(\sfA)$ we use two elementary consequences of the
	monotonicity of $\sfA$: the resolvent $(\rmI+\tau\sfA)^{-1}$ is a
	contraction in $L^2(\Omega)$, and
	$\tnrm{(\rmI+\tau\sfA)^{-1}v-v}\le\tau\tnrm{\sfA^\circ v}$ for every
	$v\in\dom(\sfA)$ \cite[Chap.~II]{Brezis73}. The increments
	$\tnrm{U^n_\tau-U^{n-1}_\tau}$ are therefore nonincreasing in $n$ and
	bounded by $\tnrm{U^1_\tau-u_0}\le\tau\tnrm{\sfA^\circ u_0}$;
	\eqref{eq:Euler-mass-loss} then follows by summing
	\eqref{eq:Euler-mass}.
\end{proof}
\subsection{Space regularity of solutions}
The next result shows that
the class $\WpO$ is the natural one, since 
the square root of any B\'enilan integral solution belongs to 
$L^2(0,+\infty;\HtN)$; in particular 
$\sfS_t u_0\in \WpO$ 
at a.e.~time, for every initial datum and in every dimension.
\begin{proposition}[$\WpO$ regularity]
	\label{prop:flow-regularity}
	For every $u_0\in L^2_+(\Omega)$ the semigroup trajectory
	$u_t:=\sfS_tu_0$ satisfies
	\begin{equation}
		\label{eq:estimate}
		\int_\Omega u_T\,\d x\ge \int_\Omega u_S\, \d x+
		\int_S^T \sfD(u_t)\,\d t\quad\text{for every }0\le S<T,
	\end{equation}
	which holds as an identity if $u_0\in \dom(\sfA)$.
	In particular 
	$t\mapsto\int_\Omega u_t\,\d x$ is nondecreasing,
	$u_t\in\WpO$ for a.e.~$t>0$, and
	\begin{equation}
		\label{eq:flow-Wp}
		\int_0^{+\infty}\sfD(u_t)\,\d t\le
		|\Omega|^{1/2}\,\tnrm{u_0}-\int_\Omega u_0\,\d x.
	\end{equation}
\end{proposition}
\begin{proof}
	Suppose first $u_0\in\dom(\sfA)$, so that $u$ is Lipschitz and,
	for a.e.~$t>0$, $\partial_tu_t=-\sfA^\circ u_t$ by~\eqref{eq:right-derivative}. By
	Proposition~\ref{prop:minimal-section} the measure associated
	with the minimal section $\sfA^\circ u_t$ is the defect-free
	$\mu_t^\circ=(\Lap u_t)^2 u_t^{-1}\Leb
	d\restr{\{\tilde u_t>0\}}$, whose total mass is $\sfD(u_t)$ by
	\eqref{eq:mu-ac-part}; hence \eqref{eq:mass-identity} gives
	\begin{displaymath}
		\frac{\d}{\d t}\int_\Omega u_t\,\d x=
		-\int_\Omega\sfA^\circ u_t\,\d x=
		\sfD(u_t)\ge0
		\quad\text{for a.e.~}t>0 .
	\end{displaymath}
	An integration in time 
	gives
	\eqref{eq:estimate}.

	For a general $u_0\in L^2_+(\Omega)$ choose
	$u_0^n\in\dom(\sfA)$ with $u_0^n\to u_0$ in $L^2(\Omega)$
	(Lemma~\ref{le:domain-closure}): by contraction
	$u_t^n:=\sfS_tu_0^n\to u_t$ in $L^2(\Omega)$ for every
	$t\ge0$. Since $\sfD=\sfD^*$ is sequentially lower
	semicontinuous with respect to the weak convergence in
	$L^2(\Omega)$ (Theorem~\ref{thm:Hessian}(1) and
	Lemma~\ref{le:relaxedD}), Fatou's Lemma yields
	\begin{displaymath}
		\int_S^{T}\sfD(u_t)\,\d t\le
		\liminf_{n\to\infty}\int_S^{T}\sfD(u^n_t)\,\d t\le
		\lim_{n\to\infty}\int_\Omega (u^n_T-u^n_S)\,\d x
		=\int_\Omega (u_T-u_S)\,\d x
	\end{displaymath}
	for every $0\le S<T$.
	
	Passing to the limit as $T\up+\infty$ and
	the upper bound
	$\int_\Omega u_T\,\d x\le|\Omega|^{1/2}\tnrm{u_T}=
	|\Omega|^{1/2}\tnrm{u_0}$ (the norm being constant) 
	yield
	\eqref{eq:flow-Wp}
	and the monotonicity of
	$t\mapsto\int u_t\,\d x$.
	In
	particular $\sfD(u_t)<+\infty$, i.e.~$u_t\in\WpO$
	(Theorem~\ref{thm:Hessian}(2)), for a.e.~$t>0$.
\end{proof}
{%
We show now a regularization property of the flow: a finite initial \emph{entropy} propagates the
$\HtN$-regularity.
For $v\in L^2_+(\Omega)$ we set
\begin{equation}
	\label{eq:defE}
	\sfE(v):=\int_\Omega v^2\ln (v^2)\,\d x\in
	\biggl[-\frac{|\Omega|}e,+\infty\biggr],
\end{equation}
the physical entropy $\int\varrho\ln\varrho$ of the density $\varrho=v^2$:
$\sfE(v)$ is finite precisely when $v^2\ln v\in L^1(\Omega)$, the
negative part of the integrand being bounded.
\begin{proposition}[Entropy regularization]
	\label{prop:entropy-regularization}
	Let $u_0\in L^2_+(\Omega)$,
	$u_t:=\sfS_tu_0$.
	\begin{enumerate}[\rm (a)]
	\item If $\sfE(u_0)<+\infty$ then
		\begin{equation}
			\label{eq:entropy-dissipation}
			\sfE(u_t)+
			4c_d\int_0^t\!\!\int_\Omega
			|\rmD^2u_r|^2\,\d x\,\d r\le\sfE(u_0)
			\qquad\text{for every }t\ge0,
		\end{equation}
		with $c_d>0$ as in \eqref{eq:crucial-estimate}; in
		particular $u\in L^2_{\rm loc}\big([0,+\infty);\HtN\big)$.
	\item {If moreover $d\le3$, then} for every $u_0\in L^2_+(\Omega)$ one has
		$\sfE(u_s)<+\infty$ for a.e.~$s>0$; consequently
		$u\in L^2_{\rm loc}\bigl((0,+\infty);\HtN\bigr)$.
	\end{enumerate}
\end{proposition}
\begin{proof}
	\emph{Step 1: $u_0\in\dom(\sfA)$.}
	By Theorem~\ref{thm:generation}, $u$ is a Lipschitz strong
	solution with $\partial_tu_t=-\sfA^\circ u_t$ for a.e.~$t$,
	and $\sfD(u_t)\le|\Omega|^{1/2}\tnrm{\sfA^\circ u_0}$
	(Proposition~\ref{prop:defect-variational}). {The function
	$\phi(a):=a^2\ln(a^2)$ is $C^1$, with
	$\phi'(a)=2a\ln(a^2)+2a$ satisfying
	$|\phi'(a)|\le C_\delta\,(1+a^{1+\delta})$ for every
	$\delta>0$; taking $\delta=2/d$, the Gagliardo--Nirenberg
	inequality \eqref{eq:GN} gives $u_t\in L^{2+4/d}(\Omega)$ and
	hence $\phi'(u_t)\in L^2(\Omega)$, with
	$\|\phi'(u_t)\|_{L2}\le
	C\big(1+\|u_t\|_{L2}^{2/d}\|\nabla u_t\|_{L2}\big)$ for
	a.e.~$t$ by \eqref{eq:ulnu-GN}; in particular $u_t\ln u_t\in
	L^2(\Omega)$ and $\sfE(u_t)<+\infty$ for every $t\ge0$. Since
	$t\mapsto u_t$ is Lipschitz in $L^2(\Omega)$ with $\partial_t
	u_t=-\sfA^\circ u_t$, the elementary bound
	$|\phi(b)-\phi(a)|\le C\,(1+\max\{a,b\}^{1+2/d})\,|a-b|$
	together with $u_r\in L^{2+4/d}(\Omega)$ (locally uniformly in
	$r$) dominates the difference quotients, so $t\mapsto\sfE(u_t)$
	is absolutely continuous with, at a.e.~$t$,}
	\begin{equation*}
		\frac{\d}{\d t}\sfE(u_t)=
		\int_\Omega\phi'(u_t)\,\partial_tu_t\,\d x=
		-4\int_\Omega \sfA^\circ u_t\,u_t\ln u_t\,\d x-
		2\int_\Omega \sfA^\circ u_t\,u_t\,\d x=
		-4\int_\Omega \sfA^\circ u_t\,u_t\ln u_t\,\d x
	\end{equation*}
	by \eqref{eq:defect-global}. {Since $u_t\in\dom(\sfA)$,
	Lemma~\ref{le:entropy-identity} applies (no boundedness
	assumption needed) and gives}
	$\frac{\d}{\d t}\sfE(u_t)\le-4c_d\int|\rmD^2u_t|^2$;
	an integration yields \eqref{eq:entropy-dissipation}.

	\emph{Step 2: $\sfE(u_0)<+\infty$.}
	Let $u_0^n:=\sfH_{1/n}u_0+\tfrac1n\in\dom(\sfA)$
	(Lemma~\ref{le:domain-closure}) and $u^n_t:=\sfS_tu^n_0$, so
	that $u^n_t\to u_t$ in $L^2(\Omega)$ for every $t\ge0$. We
	claim $\sfE(u^n_0)\to\sfE(u_0)$: by the Jensen inequality
	$(\sfH_\tau u_0)^2\le\sfH_\tau(u_0^2)$ and the monotonicity
	and convexity of $G(a):=a\ln_+a$,
	$G\big((\sfH_\tau u_0)^2\big)\le
	G\big(\sfH_\tau(u_0^2)\big)\le\sfH_\tau\big(G(u_0^2)\big)$,
	and the families
	$\{\sfH_{1/n}\big(G(u_0^2)+u_0^2\big)\}_n$, being convergent
	in $L^1(\Omega)$, are uniformly integrable: a routine
	Vitali argument (using
	$|\phi(a+\eta)-\phi(a)|\le C\eta(1+a^2)$ for the shift
	$\eta=1/n$, and the uniform lower bound $\phi\ge-1/e$) then
	gives the claim. Passing to the limit in
	\eqref{eq:entropy-dissipation} for $u^n$: $\sfE$ is lower
	semicontinuous with respect to $L^2$-convergence (Fatou, as
	$\phi\ge-1/e$), and $u^n\rightharpoonup u$ weakly in
	$L^2((0,t);H^2(\Omega))$ by the uniform bound, so that both
	terms on the left are lower semicontinuous: (a) follows,
	together with $u_r\in\HtN$ for a.e.~$r$ (the weak limit of
	$u^n_r$ remains in the closed subspace $\HtN$).

	\emph{Step 3: general $u_0$.}
	By Proposition~\ref{prop:flow-regularity},
	$\sqrt{u_s}\in\HtN$ for a.e.~$s>0$; for $d\le3$ this gives
	$u_s=(\sqrt{u_s})^2\in L^\infty(\Omega)$, whence
	$\sfE(u_s)\le\tnrm{u_s}^2\ln_+\big(\|u_s\|^2_\infty\big)
	+\tfrac{|\Omega|}e<+\infty$. Applying (a) to the trajectory
	restarted at such an $s$ (semigroup property) we obtain
	$u\in L^2((s,T);\HtN)$ for every $T>s$; since admissible $s$
	can be taken arbitrarily small, (b) follows.
\end{proof}
In particular, for $d\le3$ and any nonnegative initial datum, the
trajectory satisfies at a.e.~$t>0$
$u_t\in\HtN\cap L^\infty(\Omega)$
and $\sqrt{u_t}\in\HtN$ --- the regularity class
$\varrho^{1/2},\varrho^{1/4}\in L^2_{\rm loc}(H^2)$ of Fischer's
uniqueness theorem \cite{Fischer13}, here generated by the semigroup
itself from bare $L^2$ data.}

{\subsection{Distributional formulations of strong and
{B\'enilan} solutions}
\label{subsec:distributional}
The two conditions \eqref{eq:DLSS-syst1}--\eqref{eq:DLSS-syst2}
characterizing the graph of $\sfA$ translate into sufficient
conditions, in distributional form, for a curve to be a strong or a
{B\'enilan} solution: they are the natural target of approximation schemes.
Recall Proposition~\ref{le:DLSS-equivalence}.}
\begin{proposition}
    \label{prop:strong-equivalent}
    Let $u_0\in \dom(\sfA)$ and let 
    $u\in W^{1,2}_{\rm loc}([0,+\infty);L^2(\Omega))
    \cap L^2_{\rm loc}([0,+\infty);\HtN)$ nonnegative
	with $u(0)=u_0$.
    $u$ is a strong solution iff
    $\partial_t u_t+\Lap^2 u_t\ge0$ in $\DHtN$, i.e.
    \begin{equation}
    \label{equivalent1}
        \int_\Omega(\partial_t u_t v+\Lap u_t\Lap v)\,\d x\ge0\quad\text{for a.e.~$t>0$ and every $v\in \PZO$,}
    \end{equation}
    and
    \begin{equation}
        \label{equivalent2}
        \partial_t \int_\Omega u_t^2w\,\d x+
        {2}\,\frb(u_t,w)=0
        \quad\text{for a.e.~$t>0$ and every $w\in \ZO$.}
    \end{equation}
\end{proposition}
{%
\begin{proof}
	Since $u\in W^{1,2}_{\rm loc}((0,+\infty);L^2(\Omega))$, the
	curve $u$ is locally absolutely continuous with values in
	$L^2(\Omega)$ and for a.e.~$t>0$ the strong derivative
	$\partial_tu_t\in L^2(\Omega)$ exists; moreover, for every
	$w\in\ZO$, the elementary estimate
	$\big|\int_\Omega(u_{t+h}^2-u_t^2)w\,\d x\big|\le
	\|w\|_{L^\infty(\Omega)}\tnrm{u_{t+h}-u_t}\tnrm{u_{t+h}+u_t}$
	shows that $t\mapsto\int_\Omega u_t^2w\,\d x$ is locally
	absolutely continuous as well, with
	$\partial_t\int_\Omega u_t^2w\,\d x=
	2\int_\Omega u_t\,\partial_tu_t\,w\,\d x$ for a.e.~$t>0$.

	Set $f_t:=-\partial_tu_t$. For a.e.~$t>0$ we have $u_t\in\HtN$,
	$u_t\ge0$, and \eqref{equivalent1}, \eqref{equivalent2} take
	exactly the form \eqref{eq:DLSS-syst1}, \eqref{eq:DLSS-syst2}:
	hence $(u_t,f_t)\in\graph(\sfA)$, i.e.~$u$ solves the
	differential inclusion \eqref{eq:strong-solution1} for
	a.e.~$t>0$.

	Conversely, let $u$ be a strong solution. By
	Definition~\ref{def:strong-solution} we have
	$f_t:=-\partial_tu_t\in\sfA u_t$ for a.e.~$t>0$, so that
	$(u_t,f_t)\in\graph(\sfA)$ and Theorem~\ref{thm:main-equivalent}
	(condition 2) gives \eqref{eq:DLSS-syst1} and
	\eqref{eq:DLSS-syst2} for such $t$. The first is
	\eqref{equivalent1}; the second becomes \eqref{equivalent2} once
	the identity $\partial_t\int_\Omega u_t^2w\,\d x=
	2\int_\Omega u_t\,\partial_tu_t\,w\,\d x$ obtained above is
	inserted.
\end{proof}}
\begin{proposition}
    \label{prop:weak-equivalent}
    If
    $u\in W^{1,2}_{\rm loc}((0,+\infty);\DHtN)
    \cap L^2_{\rm loc}((0,+\infty);\HtN)$
    satisfy
    $u\ge 0$,
    $\partial_t u_t+\Lap^2 u_t\ge0$ in $\DHtN$, i.e.
    \begin{equation}
    \label{w-equivalent1}
        \partial_t\int_\Omega u_t v\,\d x+
        \int_\Omega \Lap u_t\Lap v\,\d x\ge0\quad\text{for a.e.~$t>0$ and every $v\in \PZO$,}
    \end{equation}
    and
    \begin{equation}
        \label{w-equivalent2}
        \partial_t \int_\Omega u_t^2w\,\d x+
        2\,\frb(u_t,w)=0
        \quad\text{in $\mathscr D'(0,+\infty)$ for every $w\in \ZO$.}
    \end{equation}
    Then $u$ is a B\'enilan solution.
\end{proposition}
{%
\begin{proof}
	Recall (Section~\ref{subsec:monotone-VV}) that
	$\HtN\subset L^2(\Omega)\subset\DHtN$ with
	$\la f,v\ra=\int_\Omega fv\,\d x$ whenever $f\in L^2(\Omega)$
	and $v\in\HtN$, by \eqref{eq:pairing-bridge}. By a standard
	mollification in time, the regularity
	$u\in W^{1,2}_{\rm loc}((0,+\infty);\DHtN)\cap
	L^2_{\rm loc}((0,+\infty);\HtN)$ yields
	$u\in\rmC((0,+\infty);L^2(\Omega))$ and, for every $w\in\ZO$,
	the local absolute continuity of
	$t\mapsto\int_\Omega u_t^2\,w\,\d x$ with
	\begin{equation}
		\label{eq:weighted-chain}
		\partial_t\int_\Omega u_t^2\,w\,\d x=
		2\,\la\partial_tu_t,u_tw\ra
		\quad\text{for a.e.~}t>0,
	\end{equation}
	where $u_tw\in\HtN$ by \eqref{eq:Z-acts} (the multiplication
	$v\mapsto vw$ is a bounded operator in $\HtN$ and a bounded
	symmetric operator in $L^2(\Omega)$; the choice $w\equiv1$
	gives the classical identity
	$\partial_t\tnrm{u_t}^2=2\la\partial_tu_t,u_t\ra$).

	Set $f_t:=-\partial_tu_t\in\DHtN$. For a.e.~$t>0$ the function
	$u_t\in\HtN$ is nonnegative and: \eqref{w-equivalent1} means
	$\la f_t,v\ra\le\int_\Omega\Lap u_t\Lap v\,\d x$ for every
	$v\in\PZO$, i.e.~{\rm(A.2')}; \eqref{w-equivalent2} and
	\eqref{eq:weighted-chain} give
	$\frb(u_t,w)=\la f_t,u_tw\ra$ for every $w\in\ZO$,
	i.e.~{\rm(A.3'')}. By
	Proposition~\ref{le:DLSS-equivalence}, $f_t\in\frA u_t$ for
	a.e.~$t>0$.

	Choosing $w\equiv1$ in \eqref{w-equivalent2} and observing that
	$\frb(u_t,1)=0$, we see that $\tnrm{u_t}$ is constant in time.
	Finally, let $(v,g)\in\graph(\sfA)\subset\graph(\frA)$. For
	a.e.~$t>0$, by \eqref{eq:defect-global} and the monotonicity of
	$\frA$ (Theorem~\ref{thm:DLSS-monotone}),
	\begin{displaymath}
		\partial_t\int_\Omega u_tv\,\d x=
		-\la f_t,v\ra=\la f_t,u_t-v\ra\ge
		\la g,u_t-v\ra=\int_\Omega g\,(u_t-v)\,\d x ,
	\end{displaymath}
	which is \eqref{eq:benilan-solution}: $u$ is a B\'enilan solution.
\end{proof}}

\subsection{Weak solutions and the identification of the
semigroup trajectories}
\label{subsec:very-weak-solutions}

In this subsection we obtain 
various equivalent characterizations of
B\'enilan integral solutions.

We first isolate a useful technical property: a curve
satisfying the B\'enilan-type inequalities against the smooth,
uniformly positive test functions of $\PZO$ automatically satisfies
them against the whole graph of $\sfA$.

Recall that every $u\in \WpO$ satisfies $\Lap u\in L^1(\Omega)$
and  
\begin{equation}
	\label{eq:integration-by-parts}
	\int_\Omega u\,\sfA^\circ v\,\d x=
	\int_\Omega \Lap u\,\Lap v\,\d x-
	\int_\Omega u\frac{(\Lap v)^2}v\,\d x
	\qquad
	\text{for every }v\in \PZO.
\end{equation}
\begin{lemma}[B\'enilan inequalities with test functions are enough]
	\label{le:Benilan-upgrade}
	Let $u:(0,+\infty)\to L^2(\Omega)$ be a continuous curve of
	nonnegative functions such that 
	\begin{equation}
		\label{eq:core-family}
		\partial_t\int_\Omega u_tv\,\d x\ge
		\int_\Omega u_t\,\sfA^\circ v\,\d x
		\quad\text{in }\mathscr D'(0,+\infty),\quad
		\text{for every }v\in\PZO .
	\end{equation}
	Then
	\begin{equation}
		\label{eq:full-family}
		\partial_t\int_\Omega u_tv\,\d x\ge
		\int_\Omega g\,u_t\,\d x
		\quad\text{in }\mathscr D'(0,+\infty),\quad
		\text{for every }(v,g)\in\graph(\sfA).
	\end{equation}
\end{lemma}
Notice that the integral in the right-hand side of \eqref{eq:core-family} can be expanded to \eqref{eq:integration-by-parts} whenever $u_t\in \WpO$.
\begin{proof}
	Let $(v,g)\in\graph(\sfA)$ 
	so that $g=\sfA^\circ v-g_0$, $g_0\in L^2_+(\Omega)$.

	Since $g_0$ is nonnegative, 
	\eqref{eq:full-family} is satisfied if 
	\begin{equation}
		\label{eq:full-family2}
		\int_\Omega u_T\,v\,\d x\ge
		\int_\Omega u_S\,v\,\d x+
		\int_S^T \int_\Omega \sfA^\circ v\,u_t\,\d x\,\d t
		\quad\text{for every }0\le S<T.
	\end{equation}
	For $\tau>0$ set $v_{\tau}:=\sfH_\tau v$: then
	$v_{\tau}\in\PZO$ (Section~\ref{subsec:Heat}) and
	$\Lap v_{\tau}=\sfH_\tau\Lap v\in \ZO$.
	Integrating \eqref{eq:core-family} from $S$ to $T$
	and choosing $v:=v_\tau$ we get
	\begin{equation}
		\label{eq:full-familytau}
		\int_\Omega u_T\,v_\tau\,\d x\ge
		\int_\Omega u_S\,v_\tau\,\d x+
		\int_S^T \int_\Omega \sfA^\circ v_\tau\,u_t\,\d x\,\d t.
	\end{equation}
	By the commutation inequality \eqref{eq:commutationsfA}
	we get
	\begin{equation}
		\label{eq:full-familycommute}
		\int_\Omega u_T\,v_\tau\,\d x\ge
		\int_\Omega u_S\,v_\tau\,\d x+
		\int_S^T \int_\Omega \sfA^\circ v\,\sfH_\tau u_t\,\d x\,\d t.
	\end{equation}
	Passing to the limit as $\tau\downarrow0$ we 
	obtain \eqref{eq:full-family2}.
\end{proof}
Recalling Definition \ref{def:benilan-solution}
we obtain an equivalent formulation which involves no time
regularity of the curve and tests the equation only against smooth,
uniformly positive, functions.
\begin{corollary}
    \label{prop:weak-equivalent2}
    A continuous
    nonnegative curve $u:[0,+\infty)\to L^2(\Omega)$
	is a B\'enilan integral solution if and only if 
    $t\mapsto \|u_t\|_2^2$ is
    constant in time and
	it satisfies \eqref{eq:core-family}.
\end{corollary}
We can now provide an even more expressive 
characterization of B\'enilan integral solution.
First of all, observe that integrating 
the pointwise formulation \eqref{eq:right} 
of a strong solution against a test function $v\in \ZO$ we 
get
\begin{equation}
	\label{eq:strong-dual}
	\partial_t \int_\Omega u_t \,v\,\d x+
	\fra(u_t,v)=0\quad \text{in }(0,+\infty)\quad\text{for every }v\in \ZO.
\end{equation}
If $v\in \ZO$, the form $\fra(u,v)$ can also be defined for arbitrary 
$u\in \WpO$ since $\Lap u$,
$\sfh(u,\Lap u)=(\Lap u)^2/u$ belong to $L^1(\Omega)$ and $w=\sqrt u\in \HtN$, recall \eqref{eq:form-veryweak}:
\begin{equation}
	\label{eq:fra-extension}
	\begin{aligned}
		\fra(u,v)&=\int_\Omega \Big(\Lap u\,\Lap v-\sfh(u,\Lap u)\,v\Big)\,\d x
		=
		\int_\Omega \Big(u\,\Lap^2 v-\sfh(u,\Lap u)\,v\Big)\,\d x
		\\
		&=
		\int_\Omega\big(2w\Lap w+2|\nabla
			w|^2\big)\Lap v\,\d x-
			4\int_{\{w>0\}}\Big(\Lap w+
			\frac{|\nabla w|^2}{w}\Big)^{\!2}v\,\d x
	\end{aligned}
\end{equation}
Since for integral solutions we could lose the control
of the singular measure possibly concentrated on the contact set
$\{u=0\}$, 
we may relax \eqref{eq:strong-dual} to the inequality
$\partial_t\int_\Omega u_t\,v\,\d x+\fra(u_t,v)\ge0$
for every nonnegative test function $v\in \PZO$.

Surprisingly enough, 
such an inequality coupled with norm conservation 
contains sufficient information to characterize integral solutions.
\begin{theorem}[Weak formulation of integral solutions]
	\label{def:weak-solution}
	A continuous curve $u:[0,+\infty)\to L^2_+(\Omega)$ 
	is a B\'enilan integral solution 
	to the DLSS equation if and only if 
	\begin{enumerate}[\rm (W1)]
	\item $t\mapsto\tnrm{u_t}$ is constant;
	\item
		$t\mapsto\sfD(u_t)\in L^1_{\rm loc}[0,+\infty)$;
	\item for every $v\in\PZO$
		\begin{equation}
			\label{eq:weak-solution}
			\partial_t\int_\Omega u_t\,v\,\d x+
			\fra(u_t,v)\ge0
			\quad\text{in }\mathscr D'(0,+\infty),
		\end{equation}
		with $\fra$ as in \eqref{eq:fra-extension}.
	\end{enumerate}
\end{theorem}
By \eqref{eq:fra-minimal}, {\rm(W3)} can be interpreted as a
very weak
form of the differential inclusion $\partial_tu_t+\sfA u_t\ni0$ with
$\sfA=\sfA^\circ+\partial I_{L^2_+}$: the minimal DLSS operator
$\sfA^\circ$ (realized by $\fra$) plus a nonnegative reaction supported
on the vacuum $\{u_t=0\}$, formally expressed by the normal-cone component $\partial
I_{L^2_+}(u_t)$ of Proposition~\ref{prop:minimal-section}. Testing
{\rm(W3)} only against $v\ge0$ retains exactly the one-sided
information that is stable under limits of strong solutions: the
defect measure is allowed to \emph{exceed} the absolutely continuous
square $\sfh(u_t,\Lap u_t)\,\Leb d$ by that reaction, but not to fall
below it.
Condition {\rm(W2)} makes \eqref{eq:weak-solution} meaningful:
setting $w_t:=\sqrt{u_t}$, Lemma~\ref{le:sqrt-calculus} and the
integration by parts
$\int|\nabla w_t|^2\,\d x=-\int w_t\Lap w_t\,\d x$ give
$\|\Lap u_t\|_{L^1(\Omega)}\le4\tnrm{w_t}\tnrm{\Lap w_t}$, with
$\tnrm{w_t}^2\le|\Omega|^{1/2}\tnrm{u_t}$ and
$\tnrm{\Lap w_t}^2\le d\int_\Omega|\rmD^2w_t|^2\,\d x\le
\frac{d^3}4\,\sfD(u_t)$ by \eqref{eq:Hessian}: by
\eqref{eq:veryweak-bound}, $t\mapsto\fra(u_t,v)$ then belongs to
$L^1_{\rm loc}[0,+\infty)$ for every $\zeta\in\PZO$.
{%
\begin{proof}[Proof of Theorem \ref{def:weak-solution}]
	Let us first show that 
	conditions (W1)--(W3) guarantee that 
	$u$ is a B\'enilan integral solution. We apply 
	Corollary \ref{prop:weak-equivalent2} 
	and we want to show \eqref{eq:core-family}.
	
	So we fix $v\in\PZO$ and the
	asymmetric inequality \eqref{eq:asymmetric} gives
	$$-\fra(u_t,v)\ge\fra(v,u_t)=\int_\Omega u_t\,\sfA^\circ v\,\d x.$$
	Therefore
	\begin{displaymath}
		\partial_t\int_\Omega u_tv\,\d x\ge-\fra(u_t,v)\ge
		\int_\Omega  u_t\,\sfA^\circ v\,\d x
		\quad\text{in }\mathscr D'(0,+\infty),
	\end{displaymath}
	which is precisely \eqref{eq:core-family}. 

	Let us now suppose that 
	$u$ is a B\'enilan integral solution.
	Conditions {\rm(W1)--(W2)} hold for every
	$u_0\in L^2_+(\Omega)$ by Proposition~
	\ref{prop:flow-regularity}.
	We have to prove that 
	\begin{equation}
		\label{eq:W3-integrated}
		\int_\Omega (u_T-u_S)\, v\d x+
		\int_S^T\int_\Omega u_t \,\Lap^2 v\,\d x
		\ge 
		\int_S^T\int_\Omega \sfh(u_t,\Lap u_t)\, v
		\,\d x\,\d t
		\quad\text{for every }0\le S<T,
	\end{equation}
	for every choice of $v\in \PZO$.

	Choose $u_0^n\in\dom(\sfA)$ with $u_0^n\to u_0$ in $L^2(\Omega)$
	(Lemma~\ref{le:domain-closure}) and set $u_t^n:=\sfS_tu_0^n$; by
	the contraction property $u^n\to u$ in $\rmC([0,T];L^2(\Omega))$
	for every $T>0$. By
	Theorem~\ref{thm:generation} each $u^n$ satisfies the exact
	equation \eqref{eq:W3-integrated}:
	\begin{equation}
		\label{eq:W3-integratedn}
		\int_\Omega (u^n_T-u^n_S)\, v\d x+
		\int_S^T\int_\Omega u^n_t \,\Lap^2 v\,\d x
		=
		\int_S^T\int_\Omega \sfh(u^n_t,\Lap u^n_t)\, v
		\,\d x\,\d t
		\quad\text{for every }0\le S<T,
	\end{equation}
	The left-hand side of \eqref{eq:W3-integratedn}
	passes to the limit easily.

	In order to deal with the right-hand side, 
	set $w^n:=\sqrt{u^n}$. By \eqref{eq:flow-Wp},
	\eqref{eq:Hessian} and $\tnrm{u_0^n}=\tnrm{u_0}$, the sequence
	$w^n$ is bounded in $L^2(0,T;\HtN)$; moreover
	$\|w^n_t-w_t\|_{L2}\to0$ uniformly in $t\in[0,T]$, so
	$w^n\to w$ strongly in $\rmC([0,T];L^2(\Omega))$ and, by
	interpolation with the uniform $L^2(0,T;\HtN)$-bound,
	$\nabla w^n\to\nabla w$ strongly in $L^2((0,T)\times\Omega)$ while
	$\Lap w^n\rightharpoonup\Lap w$ weakly there. By the chain rule
	\eqref{eq:chain-Lap}, $\Lap u^n=2w^n\Lap w^n+2|\nabla w^n|^2$;
	since a product of an $L^2$-strongly and an $L^2$-weakly convergent
	sequence converges weakly in $L^1$, and
	$|\nabla w^n|^2\to|\nabla w|^2$ strongly in $L^1$, we get
	$\Lap u^n\rightharpoonup\Lap u$ weakly in $L^1((0,T)\times\Omega)$.

	The integrand $\sfh$ in the right-hand side 
	of \eqref{eq:W3-integrated} is a
	nonnegative normal integrand, convex in its second argument, so
	Ioffe's lower semicontinuity theorem
	\cite[Theorem~5.8]{Ambrosio-Fusco-Pallara00}, applied on
	$(S,T)\times\Omega$ with the nonnegative weight $v$
	(using $u^n\to u$ in measure and $\Lap u^n\rightharpoonup\Lap u$
	weakly in $L^1$), gives
	\begin{displaymath}
		\int_S^{T}\!\!\int_\Omega \sfh(u_t,\Lap u_t)\,v\,\d
		x\,\d t\le\liminf_{n\to\infty}\int_S^{T}\!\!
		\int_\Omega
		\sfh(u^n_t,\Lap u^n_t)\,v\,\d x\,\d t .
	\end{displaymath}
	Passing to the limit in \eqref{eq:W3-integratedn}
	we eventually obtain \eqref{eq:W3-integrated}.
\end{proof}
\subsection{Weak solutions in dimension \texorpdfstring{$d\le3$}{d<=3}}
\label{subsec:weak-solutions-d3}
In low dimension the previous picture can be considerably sharpened;
we collect the corresponding results in this subsection.
The $b$-form \eqref{w-equivalent2} yields a sharper,
\emph{defect-free} description of the flow. The form $\frb$ is
bilinear in the pair $(u,\Lap u)$ against smooth tests, so it passes
to the limit along the natural (strong $\times$ weak) convergences of
the flow; this lets us prove that, for $d\le3$, \emph{every}
semigroup trajectory satisfies the \emph{equality} in
\eqref{w-equivalent2}, a statement stronger than the one-sided weak
formulation~\eqref{eq:weak-solution}, whose perspective term
$\int \sfh(u_t,\Lap u_t)\,v\,\d x$ is only lower semicontinuous and
could a priori gain mass on the moving vacuum in the limit.
We obtain in this way
another characterization of the semigroup trajectories among weak  solutions
in the regularity
class of Fischer, where the delicate analysis of the moving vacuum
is replaced by a direct appeal to the uniqueness theorem of
\cite{Fischer13}.
This last step, and only this one, is conditional: \cite{Fischer13}
works with periodic boundary conditions, and the adaptation of its
approximation lemmas to the Neumann setting is not carried out here;
see Remark~\ref{rem:fischer-BC}. Theorem~\ref{thm:bform-flow} below is
instead unconditional.
\begin{theorem}[The semigroup satisfies the $b$-form, $d\le3$]
	\label{thm:bform-flow}
	Let $d\le3$, $u_0\in L^2_+(\Omega)$ and $u_t:=\sfS_tu_0$
	(recall that $u\in L^2_{\rm loc}((0,+\infty);\HtN)$ by
	Proposition~\ref{prop:entropy-regularization}). Then
	\begin{align}
		\label{eq:bform-flow}
		&\partial_t\int_\Omega u_t^2\,w\,\d x+2\,\frb(u_t,w)=0
		&&\text{in }\mathscr D'(0,+\infty),\
		\text{for every }w\in\ZO,\\
		\label{eq:positivity-flow}
		&\partial_t\int_\Omega u_t\,v\,\d x+
		\int_\Omega\Lap u_t\,\Lap v\,\d x\ge0
		&&\text{in }\mathscr D'(0,+\infty),\
		\text{for every }v\in\ZO,\ v\ge0 .
	\end{align}
\end{theorem}
\begin{proof}
	\emph{Step 1 (approximation and uniform bounds).}
	Let $u_0^n:=\sfH_{1/n}u_0+\tfrac1n\in\dom(\sfA)$
	(Lemma~\ref{le:domain-closure}) and $u^n_t:=\sfS_tu^n_0$: by
	contraction,
	$\sup_{t\ge0}\tnrm{u^n_t-u_t}\le\tnrm{u^n_0-u_0}\to0$.
	Fix $0<s<T$. By Proposition~\ref{prop:flow-regularity},
	$\int_{s/2}^{s}\sfD(u^n_r)\,\d r\le|\Omega|^{1/2}\tnrm{u^n_0}
	\le C$, so there are $s_n\in(s/2,s)$ with
	$\sfD(u^n_{s_n})\le2C/s$; by Theorem~\ref{thm:Hessian} and the
	Sobolev embedding ($d\le3$),
	$\|u^n_{s_n}\|_{L^\infty(\Omega)}=
	\|\sqrt{u^n_{s_n}}\|^2_{L^\infty(\Omega)}\le C(s)$, whence
	$\sfE(u^n_{s_n})\le C(s)$. The entropy dissipation
	\eqref{eq:entropy-dissipation}, started at $s_n$, then gives
	\begin{equation*}
		\sup_n\int_s^T\!\!\int_\Omega|\rmD^2u^n_r|^2\,\d x\,\d r
		\le C(s,T):
	\end{equation*}
	hence $u^n\rightharpoonup u$ weakly in $L^2((s,T);H^2(\Omega))$
	and, by the interpolation
	$\|\nabla(u^n_r-u_r)\|^2_{2}\le
	\tnrm{u^n_r-u_r}\,\tnrm{\Lap(u^n_r-u_r)}$,
	$\nabla u^n\to\nabla u$ strongly in
	$L^2((s,T)\times\Omega)$.

	\emph{Step 2 (the equations for $u^n$).}
	$u^n$ is a strong solution: for a.e.~$t$,
	$\partial_tu^n_t=-f^n_t$ with $f^n_t\in\sfA u^n_t$; as in the
	proof of Proposition~\ref{prop:strong-equivalent},
	$\partial_t\int(u^n_t)^2w=2\int u^n_t\partial_tu^n_tw=
	-2\,\frb(u^n_t,w)$ by {\rm(A.3'')}, while {\rm(A.2)} gives
	$\int\partial_tu^n_tv+\int\Lap u^n_t\Lap v=
	\la\mu^n_t,v\ra\ge0$ for $v\in\ZO$, $v\ge0$.

	\emph{Step 3 (passage to the limit).}
	Fix $\varphi\in\rmC^\infty_c(0,+\infty)$ with
	$\supp\varphi\subset(s,T)$. The time-slice terms converge for
	every $t$, uniformly on $[s,T]$
	($\int(u^n_t)^2w\to\int u_t^2w$, $\int u^n_tv\to\int u_tv$, by
	Step 1); for the space-time integrals:
	$\iint u^n\Lap u^n\,\Lap w\,\varphi\to
	\iint u\Lap u\,\Lap w\,\varphi$
	(strong $L^2$ $\times$ weak $L^2$, with $\Lap w\,\varphi\in
	L^\infty$); writing
	$\Lap u^n\,\nabla u^n=\Lap u^n(\nabla u^n-\nabla u)+
	\Lap u^n\,\nabla u$, we get
	$\iint\Lap u^n\nabla u^n\cdot\nabla w\,\varphi\to
	\iint\Lap u\nabla u\cdot\nabla w\,\varphi$ (bounded $\times$
	strong, then weak $\times$ fixed); and
	$\iint\Lap u^n\Lap v\,\varphi\to\iint\Lap u\Lap v\,\varphi$.
	This proves \eqref{eq:bform-flow} and
	\eqref{eq:positivity-flow} in $\mathscr D'(s,T)$, and $0<s<T$
	are arbitrary.
\end{proof}

Remarkably, this characterization involves no assumption on the time
derivative of the curve: in the regularity class of Fischer \cite{Fischer13} ---
which, by Propositions~\ref{prop:flow-regularity} and
\ref{prop:entropy-regularization}, the flow itself generates from
bare $L^2$ data --- the $b$-form equation alone determines the
solution.
\begin{theorem}[$b$-form characterization in the Fischer class,
	$d\le3$]
	\label{thm:bform-fischer}
	Assume that the uniqueness theorem of Fischer
	{\rm\cite[Theorem 5]{Fischer13}}, stated there for periodic boundary
	conditions, holds in the homogeneous Neumann setting considered here
	{\rm(}see Remark~\ref{rem:fischer-BC}{\rm)}.
	Let $d\le3$, $u_0\in L^2_+(\Omega)$, and let
	$u:[0,+\infty)\to L^2_+(\Omega)$ be a continuous curve with
	$u(0)=u_0$, in the regularity class of
	Fischer, that is
	\begin{equation}
		\label{eq:fischer-class}
		u,\ \sqrt u\ \in\ L^2_{\rm loc}\big((0,+\infty);\HtN\big)
		\qquad(\varrho^{1/2},\varrho^{1/4}\in L^2_{\rm loc}(H^2)),
	\end{equation}
	satisfying the $b$-form equation \eqref{eq:bform-flow}. Then
	$u_t=\sfS_tu_0$ for every $t\ge0$. Since every semigroup
	trajectory has all the above properties
	(Theorem~\ref{thm:bform-flow},
	Propositions~\ref{prop:flow-regularity} and
	\ref{prop:entropy-regularization}), for $d\le3$ the semigroup
	trajectories are exactly the curves 
	in
	the class \eqref{eq:fischer-class} satisfying
	\eqref{eq:bform-flow}.
\end{theorem}
\begin{proof}
	Choosing $w\equiv1$ in \eqref{eq:bform-flow} and recalling that
	$\frb(u_t,1)=0$ (cf.~the proof of
	Proposition~\ref{prop:weak-equivalent}) we first observe that
	$\tnrm{u_t}$ is constant in $(0,+\infty)$, hence in
	$[0,+\infty)$ by continuity: the total mass of
	$\varrho:=u^2$ is conserved.

	For a.e.~$t>0$ the curve satisfies
	$u_t\in\HtN\cap L^\infty(\Omega)$ (by
	\eqref{eq:fischer-class} and the Sobolev embedding, $d\le3$)
	and $\sqrt{u_t}\in\HtN$. In this class the $b$-form is Fischer's
	renormalized operator: for every $u$ with $u,\sqrt u\in\HtN$ an
	elementary integration by parts (with $\varrho:=u^2$) gives
	\begin{equation}
		\label{eq:bform-fischer-ibp}
		2\,\frb(u,w)=\int_\Omega\varrho\,\Lap^2 w\,\d x
		-\int_\Omega\sum_{i,j=1}^d
		\frac{\partial_i\varrho\,\partial_j\varrho}{\varrho}\,
		\partial^2_{ij}w\,\d x
		\qquad(w\in\ZO),
	\end{equation}
	where $\partial_i\varrho\,\partial_j\varrho/\varrho=
	4\,\partial_iu\,\partial_ju$ requires no positivity and is
	meaningful on the vacuum $\{u=0\}$; hence \eqref{eq:bform-flow} is
	exactly the renormalized weak formulation of the DLSS equation
	considered in~\cite{Fischer13}. Note also
	that $\varrho$ belongs to the class
	$W^{1,1}_{\rm loc}((0,+\infty);\DHtN)\cap
	L^1_{\rm loc}((0,+\infty);L^\infty(\Omega))$ underlying that
	formulation \cite[Definition 1]{Fischer13}: indeed
	$\|\varrho_t\|_{L^\infty(\Omega)}\le C\|u_t\|_{\HtN}^2\in
	L^1_{\rm loc}(0,+\infty)$, while \eqref{eq:bform-flow}
	identifies $\partial_t\varrho_t$ with $-2\frb(u_t,\cdot)$, and
	by \eqref{eq:frb} and the embedding
	$\HtN\subset W^{1,4}(\Omega)$ ($d\le3$) the form
	$\frb(u_t,\cdot)$ extends to $\HtN$ with
	$\|\frb(u_t,\cdot)\|_{\DHtN}\le C\|u_t\|_{\HtN}^2\in
	L^1_{\rm loc}(0,+\infty)$. In this setting the weak
	formulation admits at most one solution with a given initial
	datum, by (the argument of) the uniqueness theorem of Fischer
	\cite[Theorem 5]{Fischer13}; this is the step for which the
	standing assumption of the theorem is invoked, see
	Remark~\ref{rem:fischer-BC} for a discussion of the boundary
	conditions.

	Fix $s>0$ with $\sfE(u_s)<+\infty$ (a property of a.e.~$s>0$,
	cf.~Step~3 of the proof of
	Proposition~\ref{prop:entropy-regularization}) and let
	$\hat u_t:=\sfS_{t-s}u_s$, $t\ge s$. By
	Theorem~\ref{thm:bform-flow},
	Proposition~\ref{prop:flow-regularity} and
	Proposition~\ref{prop:entropy-regularization}{\rm(a)}
	(applied to the trajectory restarted at $s$), $\hat u$ solves
	\eqref{eq:bform-flow} in the class \eqref{eq:fischer-class} on
	$(s,+\infty)$ and attains the same value $u_s$ at $t=s$:
	Fischer's uniqueness theorem yields $u_t=\sfS_{t-s}u_s$ for
	every $t\ge s$. Letting $s\down0$ along admissible values,
	$u_s\to u_0$ in $L^2(\Omega)$ and the continuity of the
	semigroup give $u_t=\sfS_tu_0$ for every $t\ge0$.
\end{proof}
\begin{remark}[On the boundary conditions in \cite{Fischer13}]
	\label{rem:fischer-BC}
	Fischer's uniqueness theorem \cite[Theorem 5]{Fischer13} is
	stated for periodic boundary conditions
	($\Omega=[{\mathbb S}^1]^d$), and \cite[Definition
	3]{Fischer13} extends the weak formulation to combined
	Dirichlet--Neumann conditions on $\rmC^{1,1}$ domains, where
	the trace of $\sqrt\varrho$ is prescribed and the test
	functions are taken in $H^2_0(\Omega)$ (a class for which, as
	observed there, no existence result is available): the
	homogeneous Neumann realization considered in this paper is not
	explicitly covered. The structure of the proof, however, is the
	same: it rests on the monotonicity inequality
	\cite[eq.~(5)]{Fischer13} --- the pointwise inequality behind
	the asymmetric Lemma~\ref{le:asymmetric} --- and on
	renormalized test functions regularized by a spatial
	convolution, an operation which is global on the torus and
	whose natural Neumann-compatible substitute in $\Omega$ is the
	Heat semigroup $\sfH_\delta$, the device systematically
	employed in this paper (cf.~Corollary~\ref{cor:main-ineq} and
	Lemma~\ref{le:convex-estimate}). The detailed adaptation to the
	Neumann setting of the approximation Lemmas~11--15 of
	\cite{Fischer13} is however not carried out here.
\end{remark}

\subsection{Two possible extensions}
\label{subsec:extensions}
\newcommand{\density}{r}
We discuss here the weighted setting and the full quantum drift-diffusion
equation~\eqref{eq:isothermal-intro} with isothermal pressure and external
potential. In both cases the additional terms
are monotone and of lower order, so that the core difficulty remains the DLSS
operator itself: we therefore limit ourselves to the algebraic structure that
makes the construction work, without carrying out the details.
\subsubsection*{The weighted setting}
	Nothing in the construction of this paper is tied to the Lebesgue
	measure. Let
	$V\in\rmC^\infty(\overline\Omega)$ and replace $L^2(\Omega)$ by
	$L^2(\Omega,\gamma)$, $\d\gamma:=e^{-V}\,\d x$, together with the Dirichlet form
	$\frd_\gamma(u,v):=\int_\Omega\nabla u\cdot\nabla v\,\d\gamma$ on
	$H^1(\Omega)$. The associated self-adjoint operator with homogeneous
	Neumann conditions is the weighted Laplacian
	\begin{equation}
		\label{eq:weighted-Lap}
		\Lap_\gamma u:=\nabla_\gamma\!\cdot\!\nabla u
		=\Lap u-\nabla V\cdot\nabla u,
		\qquad
		\nabla_\gamma\!\cdot\!F:=e^{V}\nabla\!\cdot\!\big(e^{-V}F\big)
		=\nabla\!\cdot\!F-\nabla V\cdot F ,
	\end{equation}
	with $\nabla_\gamma\cdot$ being the adjoint of $-\nabla$ in $L^2(\Omega,\gamma)$, 
	see also \cite[§1.3]{Gianazza-Savare-Toscani09} and Remark~\ref{rem:secondorder-refs}.
	The construction of this paper should generalize to this case, with
	$\Lap$ replaced by $\Lap_\gamma$ and $\d x$ by $\d\gamma$, although we
	do not carry out the details.
	In particular the DLSS
	operator~\eqref{eq:intro-Amin} becomes
	\begin{equation}
		\label{eq:weighted-operator}
		\sfA_\gamma^\circ u=\Lap_\gamma^2u-\frac{(\Lap_\gamma u)^2}u ,
		\qquad
		\partial_tu+\Lap_\gamma^2u-\frac{(\Lap_\gamma u)^2}u=0 ,
	\end{equation}
	and, in the variable $\density:=u^2$, the density \emph{with respect to
	$\gamma$}, the divergence form \eqref{eq:intro-rho-div} reads
	\begin{equation}
		\label{eq:weighted-div}
		\partial_t\density=
		-2\,\nabla_\gamma\!\cdot\!\biggl(\density\,
		\nabla\frac{\Lap_\gamma\sqrt \density}{\sqrt \density}\biggr).
	\end{equation}
	Let $\varrho:=\density\,e^{-V}$ be the density with respect to
	$\Leb d$ of the measure $u^2\gamma$,
	which is the quantity preserved in $L^1$:
	multiplying \eqref{eq:weighted-div} by $e^{-V}$ and using
	the ground state transform 
	$\sqrt \density=e^{V/2}\sqrt\varrho$ yields the 
	classical divergence form with an explicit potential.
	Indeed, using $\Lap_\gamma(e^{V/2}w)=e^{V/2}\big(\Lap w - w\,W\big)$ with
	\begin{equation}
		\label{eq:potential:W}
		W:=\tfrac14|\nabla V|^2 - \tfrac12\Lap V ,
	\end{equation}
	one gets
	$\frac{\Lap_\gamma\sqrt \density}{\sqrt \density}
	=\frac{\Lap\sqrt\varrho}{\sqrt\varrho}-W$ 
	and therefore
	\begin{equation}
		\label{eq:weighted-div-rho}
		\partial_t\varrho+
		2\,\nabla\!\cdot\!\Big(\varrho\,\nabla\Big[
		\frac{\Lap\sqrt\varrho}{\sqrt\varrho}-W\Big]\Big)=0 ,
		\qquad
		\int_\Omega\varrho\,\d x=\int_\Omega \density\,\d\gamma=\|u\|_{L^2(\Omega,\gamma)}^2
		\ \text{ constant.}
	\end{equation}
	Thus, in the original variables, the potential enters \emph{only} as an
	additional drift: rewriting \eqref{eq:weighted-div-rho} as
	$\partial_t\varrho=\nabla\!\cdot\!\big(\varrho\,\nabla\big[
	-2\Lap\sqrt\varrho/\sqrt\varrho+2W\big]\big)$,
	it is the quantum drift-diffusion equation
	\eqref{eq:isothermal-intro} with $\theta=0$, $\eps^2=2$ and external
	potential $V=2W$, $W$ as in~\eqref{eq:potential:W}.
	Two caveats are in order: first, not every
	external field arises in this way.
	Indeed, setting $\psi:=e^{-V/2}>0$ one
	computes $\Lap\psi=W\psi$, so that the obtainable fields are exactly
	those of the form $W=\Lap\psi/\psi$ with $\psi>0$, i.e.~those for
	which $\Lap-W$ possesses a positive zero mode; in semiconductor
	modelling the external potential is instead a general given, or
	self-consistently coupled, electrostatic potential. 
	Second, the ground
	state transform does not preserve the boundary condition: from
	$\sqrt\varrho=e^{-V/2}u$ and $\partial_\nn u=0$ one gets
	$\partial_\nn\sqrt\varrho+\tfrac12(\partial_\nn V)\sqrt\varrho=0$, so
	that \eqref{eq:weighted-div-rho} is a Neumann problem only when
	$\partial_\nn V=0$ on $\partial\Omega$, and a Robin problem otherwise.
	The constant function is still annihilated,
	$\Lap_\gamma1=0$, so that the steady state is $\density\equiv1$,
	i.e.~$\varrho=e^{-V}$, the density of $\gamma$, up to normalization.

	The transfer to the weighted setting is possible because $\Lap_\gamma$ is
	again a self-adjoint diffusion operator on the finite measure $\gamma$,
	generating a Markov semigroup. Consequently the Leibniz rule, the mass
	and dissipation identities, the completing-the-square identity behind
	Theorem~\ref{thm:DLSS-monotone} and the normal-cone structure on the
	vacuum should all persist, since the maximality argument uses only the
	algebra of the form and the Markov property.

	The single point where the hypotheses genuinely change is the convexity
	bound \eqref{eq:H2bound}, which is replaced by the Bochner, or
	Bakry--\'Emery, identity: for $u\in H^2(\Omega)$ with
	$\partial_\nn u=0$,
	\begin{equation}
		\label{eq:weighted-Bochner}
		\int_\Omega(\Lap_\gamma u)^2\,\d\gamma=
		\int_\Omega\Big(|\rmD^2u|^2+
		\la\rmD^2V\,\nabla u,\nabla u\ra\Big)\d\gamma
		+\int_{\partial\Omega}\mathrm{I\!I}(\nabla u,\nabla u)\,\d\sigma_\gamma ,
	\end{equation}
	where $\mathrm{I\!I}$ denotes the second fundamental form of
	$\partial\Omega$. Hence \eqref{eq:H2bound} survives, in the form
	$\int_\Omega(\Lap_\gamma u)^2\,\d\gamma\ge\int_\Omega|\rmD^2u|^2\,\d\gamma$, exactly
	under the curvature-dimension condition $\mathrm{CD}(0,\infty)$: $\Omega$
	convex \emph{and} $V$ convex. 
	If only $\rmD^2V\ge-\kappa$, the bound costs 
	$\kappa\int_\Omega|\nabla u|^2\,\d\gamma$, 
	which affects the constants in Theorem~\ref{thm:Hessian} but, being of
	lower order, should not affect the validity of the monotonicity and
	maximality results.

\subsubsection*{The isothermal pressure term and the electric field}
	The quantum drift-diffusion equation~\eqref{eq:isothermal-intro},
	\begin{equation}
		\label{eq:theta-rho}
		\partial_t\varrho=\nabla\!\cdot\!\Big(\varrho\,\nabla\Big(
		\theta\ln\varrho-\eps^2\,\frac{\Lap\sqrt\varrho}{\sqrt\varrho}+V\Big)\Big)
		=\theta\Lap\varrho-\eps^2\,\nabla\!\cdot\!\Big(\varrho\,\nabla
		\frac{\Lap\sqrt\varrho}{\sqrt\varrho}\Big)
		+\nabla\!\cdot\!(\varrho\,\nabla V),
		\qquad \theta\ge0,\ \eps^2>0,
	\end{equation}
	is compatible with the present framework, at least as far as
	monotonicity is concerned. 
	Since 
	$$\partial_t\varrho=2u\,\partial_tu,\quad 
	\Lap(u^2)=2\big(u\Lap u+|\nabla u|^2\big),
	\quad 
	\nabla\!\cdot\!(u^2\nabla V)=
	2u\,\nabla u\cdot\nabla V+u^2\Lap V,$$ in the variable
	$u=\sqrt\varrho$ equation~\eqref{eq:theta-rho} takes the form
	$$\partial_tu+\sfM^\circ u=0,\qquad
	\sfM^\circ:=\sfM_0^\circ+\sfV^\circ,\qquad
	\sfM_0^\circ u:=\frac{\eps^2}2\,\sfA^\circ u+\theta\,\sfL^\circ u,$$ 
	where
	\begin{equation}
		\label{eq:theta-operator}
		\sfL^\circ u:=-\frac{\Lap(u^2)}{2u}=
		-\Big(\Lap u+\frac{|\nabla u|^2}u\Big) ,
		\qquad
		\sfV^\circ u:=-\frac{\nabla\!\cdot\!(u^2\nabla V)}{2u}=
		-\Big(\nabla V\cdot\nabla u+\frac12\,u\,\Lap V\Big),
	\end{equation}
	the normalization $\eps^2=2$ being the one adopted in
	\eqref{eq:intro-u} and throughout the paper.

	Let us consider first the field-free case
	$V\equiv0$. 
	The new term $\sfL^\circ$ reproduces the structure of $\sfA^\circ$ one
	order lower. Integrating by parts and using the Neumann condition
	$\partial_\nn u=0$, its associated form is
	\begin{equation}
		\label{eq:theta-form}
		\int_\Omega \sfL^\circ u\,v\,\d x=
		\int_\Omega\Big(\nabla u\cdot\nabla v-
		\frac{|\nabla u|^2}u\,v\Big)\,\d x ,
	\end{equation}
	i.e.~the form $\fra$ of \eqref{eq:form} with $\Lap$ replaced by
	$\nabla$, and the same completing-the-square computation as in
	\eqref{eq:monotone} yields the exact analogues of \eqref{eq:est0},
	\eqref{eq:est1} and \eqref{eq:monotone}:
	\begin{gather}
		\label{eq:theta-est}
		\int_\Omega \sfL^\circ u\,u\,\d x=0,
		\qquad
		\int_\Omega \sfL^\circ u\,(u-1)\,\d x=
		\int_\Omega\frac{|\nabla u|^2}u\,\d x ,
		\\
		\label{eq:theta-monotone}
		\int_\Omega\big(\sfL^\circ u-\sfL^\circ v\big)(u-v)\,\d x=
		\int_\Omega\Big|\nabla u\sqrt{\frac vu}-
		\nabla v\sqrt{\frac uv}\Big|^2\,\d x\ \ge0 .
	\end{gather}
	Hence $\sfL^\circ$ is monotone and, as in \eqref{eq:est0}, conserves the
	$L^2(\Omega)$-norm of $u$, i.e.~the mass of $\varrho$; being a sum of
	monotone operators, $\sfM_0^\circ$ is monotone for every $\theta\ge0$. Its
	dissipation is
	\begin{displaymath}
		\int_\Omega \sfM_0^\circ u\,(u-1)\,\d x=
		\frac{\eps^2}2\,\sfD(u)+\theta\int_\Omega\frac{|\nabla u|^2}u\,\d x ,
	\end{displaymath}
	so that for $\theta>0$ the extra term only \emph{increases} the
	coercivity of the operator. The same happens in the entropy
	identity of Lemma~\ref{le:entropy-identity}: since
	$\int_\Omega\sfL^\circ u\,u\ln u\,\d x=
	-\frac12\int_\Omega\Lap(u^2)\ln u\,\d x=\tnrm{\nabla u}^2$,
	\begin{displaymath}
		\int_\Omega \sfM_0^\circ u\;u\ln u\,\d x\ \ge\
		\frac{\eps^2}2\,c_d\int_\Omega|\rmD^2u|^2\,\d x+
		\theta\,\tnrm{\nabla u}^2 .
	\end{displaymath}
	The new term does not affect the vacuum either: on
	$\dom(\frA)\subset\HtN$ we have $\nabla u=0$ and $\Lap u=0$
	$\Leb d$-a.e.~on $\{u=0\}$, so that $\sfL^\circ u=0$ there and no
	further defect measure arises. It is in fact dominated by $\sfD$: setting
	$w:=\sqrt u$ and recalling that $\nabla w=0$ $\Leb d$-a.e.~on $\{w=0\}$,
	the Cauchy--Schwarz inequality and the second bound of
	\eqref{eq:Hessian} give
	\begin{equation}
		\label{eq:theta-subordinate}
		\int_\Omega\frac{|\nabla u|^2}u\,\d x=
		4\int_{\{w>0\}}\frac{|\nabla w|^2}w\,w\,\d x\le
		4\bigg(\int_{\{w>0\}}\frac{|\nabla w|^4}{w^2}\,\d x\bigg)^{1/2}\tnrm w
		\le C_d\,|\Omega|^{1/4}\,\tnrm u^{1/2}\,\sfD(u)^{1/2} ,
	\end{equation}
	where we used $\tnrm w^2=\|u\|_{L^1(\Omega)}\le|\Omega|^{1/2}\tnrm u$ and
	$\sfD^*=\sfD$. In particular
	$\dom(\sfM^\circ)=\dom(\sfM_0^\circ)=\dom(\sfA^\circ)$. It is therefore
	very likely that $\sfM_0^\circ$ still admits a unique maximal monotone extension, with
	the same normal-cone structure on the vacuum: what remains to be done is
	to repeat the approximation scheme of Section~\ref{sec:maximality} with
	$\fra$ replaced by $\frac{\eps^2}2\,\fra$ plus $\theta$ times the form
	\eqref{eq:theta-form}, the additional term only improving the
	coercivity. We do not pursue the point here.

	The further contribution $\sfV^\circ$ of the drift term,
	unlike $\sfA^\circ$ and $\sfL^\circ$, is \emph{linear} in $u$.
	Assuming $V$ smooth and compatible with the no-flux condition,
	$\partial_\nn V=0$ on $\partial\Omega$, it is in fact antisymmetric in
	$L^2(\Omega)$:
	\begin{equation}
		\label{eq:drift-antisymmetric}
		\int_\Omega\big(\sfV^\circ u\,v+\sfV^\circ v\,u\big)\,\d x=
		-\int_\Omega\Big(\nabla V\cdot\nabla(uv)+uv\,\Lap V\Big)\,\d x=
		-\int_{\partial\Omega}uv\,\partial_\nn V\,\d\sigma=0 ,
	\end{equation}
	so that $\int_\Omega\sfV^\circ u\,u\,\d x=0$: the drift conserves the
	mass as well, and both $\sfV^\circ$ and $-\sfV^\circ$ are monotone,
	contributing nothing to the monotonicity gap
	\eqref{eq:theta-monotone}. 
	
	The full operator $\sfM^\circ$ is
	therefore monotone, for every $\theta\ge0$ and every such $V$.
	Being of first order, $\sfV^\circ$ is moreover an infinitesimally small
	perturbation of $\sfM_0^\circ$: the entropy estimate above and
	\eqref{eq:ulnu-GN} give
	$$\tnrm{\sfV^\circ u}\le a\,\tnrm{\sfM_0^\circ u}+\omega_a(\tnrm u)$$
	for every $a>0$, while \eqref{eq:est0} and \eqref{eq:drift-antisymmetric}
	bound $\tnrm{u}$ by the datum along the resolvent equation.

	\smallskip\noindent
	Alternatively, the drift may be absorbed into the weighted formalism of
	the previous paragraph, by running the construction in
	$L^2(\Omega,\gamma)$ with $\gamma=\psi^2\Leb d$ and $\psi>0$ as above.
	Here a \emph{single} $\psi$ has to serve both
	the second and the fourth order term, the former contributing the drift
	potential $-2\theta\ln\psi$ and the latter the quantum field
	$\eps^2W=\eps^2\Lap\psi/\psi$. In the variables $u$ and $\Leb d$ this
	reproduces \eqref{eq:theta-rho} precisely when
	\begin{equation}
		\label{eq:theta-gamma}
		\eps^2\Lap\psi=V\psi+2\theta\,\psi\ln\psi
		\ \ \text{in }\Omega,
		\qquad
		\partial_\nn\psi=0\ \ \text{on }\partial\Omega,
		\qquad
		\psi>0 .
	\end{equation}
	This 
	semilinear Neumann problem
	can be solved for every smooth $V$ and $\theta>0$ by the sub- and
	supersolution method \cite{Amann71},
	observing that the constants $\rme^{-\max V/(2\theta)}$ and
	$\rme^{-\min V/(2\theta)}$ form an ordered pair of sub- and
	supersolutions of \eqref{eq:theta-gamma}.

\section{Maximality of the DLSS operator}
\label{sec:maximality}
The maximality of $\sfA$ is obtained through an approximation scheme.
For $\eps\in(0,1]$ we regularize the resolvent problem
$u+\sfA u\ni f$ into a variational inequality on the space
$\VO=W^{1,p}(\Omega)\cap\HtN$ ($p>d$), penalized by a convex term
$\eps\Phi$ whose finiteness forces the strict positivity
$\min_{\overline\Omega}u_\eps>0$ and thereby tames the singular
perspective nonlinearity $(\Lap u)^2/u$; existence for fixed $\eps$
follows from the general variational-inequality result of
Baiocchi--Capelo (Theorem~\ref{thm:BC-VI}). Uniform entropy and~$W^{1,p}$~estimates then permit the passage to the limit
$\eps\downarrow0$, in which the defect measure is recovered by lower
semicontinuity.

The penalization is kept \emph{smooth} on purpose. A hard two-sided
constraint $\eps\le u\le 1/\eps$ would also give coercivity --- the
upper bound forces $1/u\ge\eps$, hence
$\int_\Omega(\Lap u)^2/u\ge\eps\,\tnrm{\Lap u}^2$ --- but its
minimizers would touch the obstacle and satisfy only a one-sided
variational inequality, whose admissible variations are the increments
$u-v$ with $v$ in the convex set. The entropy and Sobolev estimates,
which require testing against nonlinear functions of $u$, would then be
out of reach. The barrier $\eps\Phi$ instead keeps the approximants
strictly positive and in the interior, where they solve an Euler
\emph{equation} and all the a priori estimates remain available.

The same scheme, with the duality map $\rmJ$ in place
of the identity, yields the maximality of $\frA$ in the $H^2$--$H^{-2}$
duality (Section~\ref{subsec:frA-maximality}).

Let us first point out the properties that we are going to prove, tailored to the characterization of $\sfA$ 
given in Proposition \ref{le:DLSS-equivalence}.
\begin{lemma}
		\label{cor:equivalent-maximal}
	$\sfA$ is maximal monotone in $L^2(\Omega)$ if
	for every $f\in L^2(\Omega)$ there exists
	a nonnegative $u\in \HtN$
	\begin{alignat}{2}
		\label{eq:maximal1}
		\int_\Omega \Lap u\Lap v\,\d x&\ge
		\int_\Omega (f-u)v\,\d x \quad&&\text{for every }
		v\in \PZO \\
		\label{eq:maximal2}
		\frb(u,w)&=
		\int_\Omega (f-u)\,uw\,\d x
		\qquad \qquad&& \text{for every }
		w\in \ZO.
	\end{alignat}
	In this case $u$ is the unique
	solution of the resolvent equation
	$u+\sfA u\ni f$.
\end{lemma}
\subsection{Estimate for the \texorpdfstring{$W^{1,p}(\Omega)$}{W1p}-seminorm.}

\begin{lemma}
	$\sfH_t$ is a contraction semigroup
	in $W^{1,p}(\Omega)$ and
	for every $u\in W^{1,p}(\Omega)$
	\begin{equation}
		|\nabla \sfH_t u|^p\le 
		\sfH_t 	|\nabla u|^p,\quad
		\int_\Omega |\nabla \sfH_t u|^p\,\d x
		\le 
		\int_\Omega |\nabla u|^p\,\d x.
	\end{equation}
	Moreover, for a.e.~$s>0$ we have
	$\int_\Omega |\nabla u_s|^{p-2}|\rmD u_s|^2\,\d x<+\infty$ and
	\begin{equation}
		\label{eq:W1pderivative}
		\frac 1p\int_\Omega |\nabla  u|^p\,\d x-
		\frac 1p\int_\Omega |\nabla \sfH_t u|^p\,\d x
		\ge \int_0^t
		\Big(\int_\Omega |\nabla u_s|^{p-2}|\rmD^2 u_s|^2\,\d x\Big)\,\d s
	\end{equation}
\end{lemma}
\begin{proof}
	Let $H:[0,+\infty)\to [0,+\infty)$ be
	an increasing and smooth convex function
	with bounded derivative.
	Recall that for $s>\eps>0$ the function
	$u_s=\sfH_s u$ is Lipschitz and in $H^2(\Omega)$ with
	$\Lap u_s\in H^1(\Omega)$ so that 
	$\frac \d{\d s}|\nabla u_s|^2=
	2\nabla u_s\cdot \nabla \Lap u_s $.
	
	We can easily compute
	\begin{align*}
		\frac\d{\d s}\int_\Omega 
		H\big(|\nabla u_s|^2\big)\,\d x
		&=
		2\int_\Omega 
		H'\big(|\nabla u_s|^2\big)
		\nabla u_s\cdot \nabla \Lap u_s\,\d x
		\\&=
		-
		2\int_\Omega 
		H'\big(|\nabla u_s|^2\big)
		|\rmD^2 u_s|^2\,\d x
		+
		\int_\Omega 
		H'\big(|\nabla u_s|^2\big)
		\Lap \big(|\nabla u_s|^2\big)\,\d x
		\\&\le 
		-
		2\int_\Omega 
		H'\big(|\nabla u_s|^2\big)
		|\rmD^2 u_s|^2\,\d x
		-2
		\int_\Omega 
		H''\big(|\nabla u_s|^2\big)
		\Big|\nabla \big|\nabla u_s\big|^2\Big|^2\,\d x
		\\&\le 
		-
		2\int_\Omega 
		H'\big(|\nabla u_s|^2\big)
		|\rmD^2 u_s|^2\,\d x.
	\end{align*}
	It follows that 
	\begin{displaymath}
		\int_\Omega 
		H\big(|\nabla u|^2\big)\,\d x-
		\int_\Omega 
		H\big(|\nabla u_t|^2\big)\,\d x
		\ge 
		2\int_0^t 
		\bigg(
		\int_\Omega 
		H'\big(|\nabla u_s|^2\big)
		|\rmD^2 u_s|^2\,\d x\bigg)\,\d s
	\end{displaymath}
	Choosing 
	a sequence of smooth convex functions approximating
	$H_p(\theta):=\frac 2p\theta^{p/2}$
	we conclude.
\end{proof}
\begin{lemma}
\label{le:p-vs-S}
	Suppose that for $u\in 
	\HtN\cap W^{1,p}(\Omega)$
	\begin{equation}
		\label{eq:ass}
		\rmE_p^p(u):=
		\liminf_{t\down 0}
		t^{-1}\bigg(
		\frac 1p\int_\Omega |\nabla  u|^p\,\d x-
		\frac 1p\int_\Omega |\nabla \sfH_t u|^p\,\d x\bigg)<+\infty
	\end{equation}
	Then 
	\begin{equation}
	\label{eq:from-Fatou}
		\int_\Omega |\nabla u|^{p-2}|\rmD^2 u|^2\,\d x\le \rmE_p^p(u)<+\infty
	\end{equation}
	In particular the map 
	$x\mapsto |\nabla u|^p$ belongs to $W^{1,1}(\Omega)$ 
	with
	\begin{equation}
		\label{eq:BV-estimate}
		\int_\Omega \Big|\nabla \big(|\nabla u|^p)\big)\Big|\,\d x
		\le 
		\sqrt {2p \rmE_p^p(u)} 
		\Big(\int_\Omega |\nabla u|^p
		\, \d x\Big)^{1/2}
	\end{equation}
	and there exists a constant $C$
	only depending on $\Omega$ and $p$ such that 
	setting $\tilde p:=pd/(d-1)$ we have
	\begin{equation}
		\label{eq:extra-integrability}
		\|\nabla u\|_{L^{\tilde p}(\Omega)}
		\le C\bigg[\rmE_p(u)
		+\|\nabla u\|_{L^p(\Omega)}\bigg].
	\end{equation}
\end{lemma}
\begin{proof}
	We just observe that 
	$\sfH_s u\to u$ strongly in $W^{1,p}(\Omega)$ 
	and in $\HtN$.
	By \eqref{eq:ass}
	we know that there exists a vanishing sequence
	$t_n\down0$ such that 
	\begin{displaymath}
		\lim_{n\to\infty}\frac1{t_n}
		\int_0^{t_n}
		\Big(\int_\Omega |\nabla u_s|^{p-2}|\rmD^2 u_s|^2\,\d x\Big)\,\d s=
		\rmE_p(u)<+\infty
	\end{displaymath}
	and therefore a vanishing sequence
	$s_n\in (0,t_n)$ 
	such that 
	\begin{displaymath}
		\lim_{n\to\infty}
		\Big(\int_\Omega |\nabla u_{s_n}|^{p-2}|\rmD^2 u_{s_n}|^2\,\d x\Big)\,\d s\le \rmE_p(u)<+\infty
	\end{displaymath}
	Possibly extracting a further subsequence,
	we get \eqref{eq:from-Fatou}
	applying Fatou's Lemma.
	
	Now we 
	select 
	a convex, Lipschitz, increasing and smooth function 
	$H:[0,+\infty)\to [0,+\infty)$ satisfying
	\begin{displaymath}
		H'(r)\le r^{p/2-1},\quad 
		rH'(r)\le cH(r)
	\end{displaymath}
	and we estimate 
	the gradient of $H(|\nabla u|^2)$
	\begin{align*}
		\big|\nabla H(|\nabla u|^2)\big|\le 
		2H'(|\nabla u|^2)\, |\nabla u|\, |\rmD^2 u|
	\end{align*}
	Since  $r H'(r) \le c H(r)$ 
	we get
	\begin{align*}
		\int_\Omega \big|\nabla H(|\nabla u|^2)\big|\,\d x
		&\le 
		2
		\Big(\int_\Omega H'(|\nabla u|^2)
		\, |\nabla u|^2\,\d x\Big)^{1/2}
		\Big(\int_\Omega H'(|\nabla u|^2)
		\, |\rmD^2 u|^2\,\d x\Big)^{1/2}
		\\&
		\le 2\sqrt {\rmE^p_p(u)} 
		\Big(c\int_\Omega H(|\nabla u|^2)
		\, \d x\Big)^{1/2}
	\end{align*}
	Selecting now an approximation of 
	$H_p(r):=r^{p/2}$ 
	with $c:=p/2$ 
	we get
	\begin{align*}
	\int_\Omega \Big|\nabla \big(|\nabla u|^p)\big)\Big|\,\d x
		&\le 
		\sqrt {2p \rmE^p_p(u)} 
		\Big(\int_\Omega |\nabla u|^p
		\, \d x\Big)^{1/2}
	\end{align*}
	\eqref{eq:extra-integrability}
	then follows by Sobolev inequalities.
\end{proof}

\subsection{Approximation}
\label{sec:approximation}

Let us fix $p>d$ so that $W^{1,p}(\Omega)\subset \rmC^\alpha(\overline\Omega)$ 
with $\alpha=1-d/p>0$ and 
let us denote by $\VO$ the 
reflexive Banach space 
$$\VO=W^{1,p}(\Omega)\cap 
\HtN=\Big\{u\in W^{1,p}(\Omega)\cap H^2(\Omega):\partial_\nn u=0\text{ on }\partial\Omega\Big\},$$ 
endowed with the norm
\begin{equation}
	\label{eq:Vnorm}
	\|u\|_\V:=
\Htnrm{u}
+
\pnrm {\nabla u}.
\end{equation}
We also set
\begin{equation}
	\label{eq:positive-cone}
	\PVO:=\Big\{u\in \VO:\min _{\overline\Omega}u>0\Big\}
\end{equation}
and we observe that 
$\VO\subset \XO$, $\PVO\subset \PXO$
since $p>d$ and every element of $\VO$ 
has a continuous 
representative in $\overline\Omega$.
In particular, the form
$\fra$ introduced by \eqref{eq:form}
is well defined in $\PVO\times \VO$.

We choose a parameter $\theta\ge d/\alpha$ and we 
	set for every 
	nonnegative $u\in \rmC(\overline\Omega)$ 
\begin{equation}
	\Psi(u):=
	 \int_{\Omega\cap \{u>0\}} \frac 1{u^\theta}\,\d x.
\end{equation}
We eventually
define the function
$\Phi:\VO\to[0,+\infty]$
\begin{equation}
	\label{eq:Phi}
	\Phi(u):=
	\begin{cases}
	\displaystyle \frac 1p
	\pnrm{\nabla u}^p+
	\Psi(u)&\text{if }u\in \PVO,\\
	+\infty&\text{otherwise}.
	\end{cases}
	\qquad
\end{equation}
\begin{lemma}
	The function $\Phi:\VO\to [0,+\infty]$ is proper, convex,
	and lower semicontinuous, it satisfies
	\begin{equation}
		\label{eq:F1}
		\Phi(u)\ge \frac 1p \pnrm{\nabla u}^p
		\quad\text{for every }u\in \VO,
	\end{equation}
and the property
\begin{equation}
	\label{eq:F2}
	\Phi(u)<+\infty
	\quad \Rightarrow\quad
	u\in \PVO,
	\quad 
	\min_{\overline\Omega}u>0.
\end{equation}
\end{lemma}
\begin{proof}
	The convexity of $\Phi$ and
	\eqref{eq:F1}, \eqref{eq:F2} 
	are clear; we have only to check
	the lower semicontinuity.
	
	Let 
	$u_n$ be a sequence in 
	$\dom(\Phi)$ 
	converging to 
	$u$ in $\VO$ with
	$\Phi(u_n)\le S<+\infty$.
	Since 
	$u\in W^{1,p}(\Omega)$ as well,
	$u$ is nonnegative and 
	belongs to $\rmC^\alpha(\overline\Omega)$,
	$u_n$ is converging to 
	$u$ uniformly, and by Fatou's Lemma
	\begin{displaymath}
		\Psi(u) \le S<\infty.
	\end{displaymath}
	We can prove that 
	$\min_{\overline\Omega}u>0$ by contradiction. In fact, 
	since $u$ is $\alpha$-H\"older continuous in $\overline\Omega$,
	if $u(x_0)=0$ for some 
	$x_0\in \overline\Omega$, 
then there would exist $L>0$ such that
$u(x)\le L|x-x_0|^\alpha$ for every $x\in \Omega$ and therefore 
\begin{displaymath}
	\Psi(u) 
\ge 
	L^{-\theta}
	\int_\Omega \frac {1}{|x-x_0|^{\alpha\theta}}\,\d x=+\infty.
\end{displaymath}
	We deduce that 
	$u\in \PVO$ 
	and $\Phi(u)\le S$ 
	by the continuity of
	the $W^{1,p}(\Omega)$ seminorm.
\end{proof}
\begin{theorem}
\label{thm:existence}
	Let us fix 
	$f\in L^2(\Omega)$.
	For every $\eps\in (0,1]$ 
	there exists a unique element
	$u_\eps \in \dom(\Phi)
	\subset \PVO$ 
	which solves the family of variational inequalities
	\begin{equation}
	\label{eq:VI}
		\eps \Phi(u_\eps)+
		\scalpt{u_\eps-f}{u_\eps -v}
		+\fra(u_\eps ,u_\eps-v)\le 
		\eps \Phi(v)
		\quad\text{for every }v\in \dom(\Phi).
	\end{equation}
	$u_\eps$ satisfies the uniform estimate
	\begin{equation}
		\label{eq:basic-estimate}
		\eps \Phi(u_\eps)+
		\frac 12
		\|u_\eps\|_{L2}^2
		+\int_\Omega \frac{(\Lap u_\eps)^2}{u_\eps}\,\d x
		\le 
		a,
		\qquad
		a:= |\Omega|+
		\int_\Omega f_-\,\d x+
	\frac 12 \int_\Omega (f+1)_+^2\,\d x.
	\end{equation}
\end{theorem}
\begin{proof}
	We can apply 
	the general 
	existence result 
	for variational inequalities 
	stated in
	Theorem 10.1 of
	\cite{Baiocchi-Capelo84},
	that we 
	stated for reflexive Banach spaces
	in the Appendix
	(see Theorem \ref{thm:BC-VI}).
	
	To match the notation of 
	Theorem \ref{thm:BC-VI}
	we define $\V:=\VO$ and 	
	\begin{displaymath}
		\mathsf f(u,v):=
			\scalpt{u-f}{u-v}+
			\fra(u,u-v)
			\quad \text{for every }u,v\in \dom(\Phi).
	\end{displaymath}
	Using
	\eqref{eq:monotone}
	it is easy to check that 
	$\mathsf f$ satisfies the monotonicity condition
	$$\mathsf f(u,v)+\mathsf f(v,u)\ge \|u-v\|_{L2}^2\ge 0$$
	so that (VI.1) holds true.
	
	Condition (VI.2) is obvious;
	it is also immediate to check 
	(VI.3), i.e.~that 
	for every 
	$u_0,u_1,v\in \dom(\Phi)$ 
	the map 
	$t\mapsto \mathsf f((1-t)u_0+t u_1,v)$ is continuous:
	the $L^2$-scalar product is
	continuous 
	and for $\fra$ 
	we use the fact that $\inf u_i>0$.
	
	Concerning (VI.4), 
	choosing $v_0:=1$ 
	we have 
	\begin{align*}
	\scalpt{u-f}{u-1}
	&=
	\|u\|_{L2}^2-
	\scalpt{f+1}{u}+\int_\Omega f\,\d x
	\\&\ge
	\frac 12 
	\int_\Omega u^2\,\d x-
	\int_\Omega f_-\,\d x-
	\frac 12 \int_\Omega (f+1)_+^2\,\d x
\end{align*}
	so that 
	\begin{align*}
		\eps \Phi(u)&+
		\scalpt{u-f}{u-1}+
		\fra(u,u-1)
		-
		\eps \Phi(1)
	\\&	\ge \eps \Phi(u)+
		\frac 12
		\|u\|_{L2}^2
		+\int_\Omega \frac{(\Lap u)^2}u\,\d x
		-a,
	\end{align*}
	where we used $\Phi(1)=\Psi(1)=|\Omega|$ and $\eps\le1$, so that
	$\eps\Phi(1)\le|\Omega|$ is absorbed in $a$. 
	If 
	we define the set 
	$$B:=
	\Big\{u\in \dom(\Phi):
	\eps\Phi(u)+
		\scalpt{u-f}{u-1}+
		\fra(u,u-1)\le \eps \Phi(1)\Big\}$$ 
	then  
	we see that 
	$B$ is bounded in $W^{1,p}(\Omega)$
	(thanks to \eqref{eq:F1});
	in particular $\sup u$ is uniformly bounded in $B$
	and therefore also $\tnrm{\Lap u}$ is bounded. 
	It follows that 
	$B$ is a bounded
	subset of $\VO$.	
\end{proof}
\subsection{Estimates for the solution
of the regularized variational inequality}
We now derive the Euler equation associated with 
\eqref{eq:VI};
for a vector $\xxi\in L^p(\Omega;\R^d)$ we set
\begin{equation}
	\label{eq:p-duality}
	J_p(\xxi):=|\xxi|^{p-2}\xxi.
\end{equation}
\begin{proposition}
\label{prop:Euler}
Under the same assumptions of 
Theorem 
\ref{thm:existence}
the unique solution
	$u_{\eps}\in \dom(\Phi)
	\subset \PVO$ 
	of 
	\eqref{eq:VI} 
	\begin{align}
	\label{eq:VE}
	\eps
	\int_\Omega J_p(\nabla u_\eps )\cdot \nabla z\,\d x 
	-\eps\theta  \int_\Omega 
		\frac 1{u_\eps^{\theta+1}}z\,\d x
		+
		\scalpt{u_\eps-f}{z}
		+\fra (u_\eps,z)
		=0
		\quad 
		\text{for every $z\in \VO$.}
	\end{align}
\end{proposition}
\begin{proof}
	In order to get \eqref{eq:VE} we
	start from \eqref{eq:VI}
	for $v_t:=u_\eps+tz$ observing that 
	for $|t|$ sufficiently small $v_t\in \dom(\Phi)$ and 
	\begin{displaymath}
		\Phi(u_\eps+tz)-\Phi(u_\eps)\le 
		t
		\int_\Omega J_p(\nabla v_t)\cdot 
		\nabla z\,\d x
		-\theta  t \int_\Omega 
		\frac 1{v_t^{\theta+1}}z\,\d x
	\end{displaymath}
	and
	\begin{displaymath}
		\scalpt{u_\eps-f}{u_\eps-v_t}
		+
		\fra (u_\eps,u_\eps-v_t)
		=
		-t\bigl[
		\scalpt{u_\eps-f}{z}+
		\fra (u_\eps,z)
		\bigr].
	\end{displaymath}
	Dividing by $t$, passing to the limit
	as $t\downarrow0$ or $t\uparrow 0$ we obtain \eqref{eq:VE}. 
\end{proof}
\begin{lemma}
\label{le:further-estimates}
Under the same assumptions
of Theorem \ref{thm:existence},
the solution $u_\eps\in \PVO$ satisfies
	\begin{equation}
		\label{eq:ap1}
		\int_\Omega 
		\bigg[
		\eps\theta \frac 1{u_\eps^{\theta+1}}+
		\frac{(\Lap u_\eps)^2}{u_\eps}\bigg]\,\d x+
		\int_\Omega (f-u_\eps)\,\d x=0
	\end{equation}
	\begin{equation}
		\label{eq:ap2}
	\int_\Omega\bigg[\eps |\nabla u_\eps|^p
		-\eps 
		\frac 1{u_\eps^{\theta}}\bigg]
		\,\d x
		+\int_\Omega (u_\eps-f)u_\eps\,\d x=0
	\end{equation}
	\begin{equation}
		\label{eq:ap3}
		\int_\Omega 
		\bigg[
		\eps J_p(\nabla u_\eps)
		\cdot \nabla (u_\eps\, w)-
		\eps \theta 
		\frac 1{u_\eps^{\theta}}w
		\bigg]\,\d x+
		\frb(u_\eps,w)+
		\int_\Omega (u_\eps-f)u_\eps\,w\,\d x=0
	\end{equation}
	for every $w\in \ZO$.
\end{lemma}
\begin{proof}
	\eqref{eq:ap1}
	follows from \eqref{eq:VE} by taking $w\equiv -1$ and using \eqref{eq:est1}.
	
	\eqref{eq:ap2} 
		follows from 
		the same argument we
		used to derive \eqref{eq:VE} by taking 
		$v_t=(1+t)u$.
		
	\eqref{eq:ap3} 
		follows from \eqref{eq:VE} by 
		replacing $w$ by $uw$
		and recalling the definition of 
		$\frb$.
\end{proof}
Now we obtain two 
more refined and crucial estimates.
The first one involves 
a perturbation argument 
by means of the Heat flow
and the quantity $\rmE_p$
we introduced in Lemma 
\ref{le:p-vs-S}.
\begin{proposition}
\label{prop:apriori-estimate2}
	The solution
	$u_\eps$ of 
	\eqref{eq:VI} satisfies
		\begin{equation}
		\label{eq:crucial-estimate-eps}
	\eps
	\rmE_p^p(u_\eps)
	+
	\int_\Omega |\nabla u_\eps|^2\,\d x
	\le 
	-\int_\Omega f\,\Lap u_\eps\,\d x.
	\end{equation}
	In particular, if $f\in H^1(\Omega)$ we get
	\begin{equation}
		\label{eq:H1-estimate}
		\int_\Omega |\nabla u_\eps|^2\,\d x
	\le 
	\int_\Omega |\nabla f|^2\,\d x.
	\end{equation}
\end{proposition}
\begin{proof}
	We use a perturbation argument
	selecting $v:=u_{\eps,t}=\sfH_t u_\eps$ in 
	\eqref{eq:VI}.
	
	By \eqref{eq:fra-vs-S} 
	we have $\fra(u_\eps,u_\eps-\sfH_t u_\eps)\ge 0$ so that 
	\begin{equation}
		\label{eq:step1}
		\eps \Phi(u_\eps)-
		\eps \Phi(u_{\eps,t})+
		\scalpt{u_\eps}{u_\eps-u_{\eps,t}}
		\le 
		\scalpt{f}{u_\eps-u_{\eps,t}}
	\end{equation}
	Dividing by $t>0$ and using the fact that 
	\begin{align*}
		\lim_{t\down0}
		\frac 1t\scalpt{f}{u_\eps-u_{\eps,t}}
		&=
		\scalpt{f}{-\Lap u_\eps}\\
		\lim_{t\down0}
		\frac 1t\scalpt{u_\eps}{u_\eps-u_{\eps,t}}
		&=
		\scalpt{u_\eps}{-\Lap u_\eps}=
		\tnrm{\nabla u}^2
	\end{align*}
	and 
	\begin{displaymath}
	\Phi(u_\eps)-
		\Phi(u_{\eps,t})
		\ge 
		\frac{1}p
		\Big(
		\pnrm{\nabla u_\eps}^2-
		\pnrm{\nabla u_{\eps,t}}^p
		\Big)
		+
		\Big(\Psi(u_\eps)-\Psi(u_{\eps,t})\Big)
		\ge \frac{1}p
		\Big(
		\pnrm{\nabla u_\eps}^2-
		\pnrm{\nabla u_{\eps,t}}^p
		\Big)
	\end{displaymath}
	we get
	\begin{align*}
		\eps \limsup_{t\down0}
		\frac 1t \bigg[\frac{1}p
		\Big(
		\pnrm{\nabla u_\eps}^2-
		\pnrm{\nabla u_{\eps,t}}^p
		\Big)
		\bigg]+
		\tnrm{\nabla u_\eps}^2\le 
		\scalpt{f}{-\Lap u_\eps}
	\end{align*}
	Applying Lemma \ref{le:p-vs-S}
	we eventually obtain
	\eqref{eq:crucial-estimate-eps}.
\end{proof}
\begin{proposition}
	\label{prop:apriori-estimate3}
	There exists a constant $C$ 
	independent of $\eps\in (0,1]$
	such that 
	\begin{equation}
		\label{eq:great}
		\eps\rmE_p^p(u_\eps)+
		\frac 12\int_\Omega |\nabla u_\eps|^2\,\d x+
		\frac {c_d}2\int_\Omega |\rmD^2 u_\eps|^2
		\,\d x
		\le C(1+\|f\|_{L2}^{2+4/d}). 
	\end{equation}
\end{proposition}
\begin{proof}
	Recall that by \eqref{eq:basic-estimate}
	\begin{equation}
	\label{eq:simple-but-useful}
		\frac\eps p\pnrm{\nabla u_\eps}^p+\frac 12 
		\tnrm{u_\eps}^2\le a.
	\end{equation}
	Choosing $z:=u_\eps\ln u_\eps$
	in \eqref{eq:VE}
	and recalling \eqref{eq:crucial-estimate-eps}
	we get
	\begin{equation}
		\label{eq:crucial2}
		\begin{aligned}
		\eps \pnrm{\nabla u_\eps}^p
		&+\eps\theta\int_\Omega  
		\frac{\ln_- u_\eps}{u_\eps^\theta}\,\d x
		+c_d\tnrm{\rmD^2 u_\eps}^2
		+\int_\Omega u^2_\eps \ln_+ u_\eps\,\d x
		\\&\le 
		\int_\Omega f u_\eps \ln u_\eps\,\d x
		+
		\eps \int_\Omega |\nabla u_\eps|^p
		\ln_- u_\eps\,\d x
		+\eps |\Omega|
		\end{aligned}
	\end{equation}
	where we used the fact that 
	$$\theta \frac{\ln_+ u_\eps}{u_\eps^\theta}
	\le \frac 1{ \rme},\quad 
	u_\eps^2\ln_- u_\eps \le 
	\frac 1{2\rme},\quad 
	\frac 1{\rme}+\frac 1{2\rme}\le 1$$ 
	We also have, for every 
	$\beta\in (p,pd/(d-1)]$ and $\beta'=\beta/(\beta-1)$
	\begin{align*}
		\int_\Omega |\nabla u_\eps|^p
		\ln_- u_\eps\,\d x&\le 
		\|\nabla u_\eps\|_{L^\beta}^{p/\beta}
		\Big(\int_\Omega \big(\ln_- u_\eps\big)^{\beta/(\beta-p)}
		\,\d x\Big)^{1-p/\beta}
		\\&\le 
		\frac \delta \beta 
		\|\nabla u_\eps\|_{L^\beta}^{p}
		+
		\frac 1{\delta {\beta'}}
		\Big(\int_\Omega \big(\ln_- u_\eps\big)^{\beta/(\beta-p)}
		\,\d x\Big)^{(\beta-p)/(\beta-1)}
		\\&\le 
		\frac \delta \beta 
		\|\nabla u_\eps\|_{L^\beta}^{p}
		+\frac 1{\delta \beta'}
		\Big(1+
		\int_\Omega \big(\ln_- u_\eps\big)^{\beta/(\beta-p)}
		\,\d x\Big)
		\\&\le 
		\|\nabla u_\eps\|_{L^p}^{p}+
		\rmE_p(u_\eps)^p+ 
		\frac 1{\delta \beta'}+
		\frac 1{\delta \beta'}
		\int_\Omega \big(\ln_- u_\eps\big)^{\beta/(\beta-p)}
		\,\d x,
	\end{align*}
	where we selected $\delta>0$ sufficiently small in order to have
	\begin{displaymath}
		\delta\|\nabla u_\eps\|_{L^\beta}^{p}
		\le \|\nabla u_\eps\|_{L^p}^{p}+
		\rmE_p(u_\eps)^p
	\end{displaymath}
	as in \eqref{eq:extra-integrability}.
	
	Using the previous inequality to
	give a bound to the right-hand side of \eqref{eq:crucial2}
	and summing up
	\eqref{eq:crucial-estimate-eps}
	we get
	\begin{align}
	\notag
		\int_\Omega &|\nabla u_\eps|^2\,\d x+
		c_d\int_\Omega |\rmD^2 u_\eps|^2\,\d x
		+
		\int_\Omega u_\eps^2\ln_+ u_\eps\,\d x
		\\&	\notag
		\le 
		\int_\Omega
		fu_\eps
		\ln u_\eps\,\d x-
		\int_\Omega f\Lap u\,\d x
		+\eps|\Omega|+
		\eps \int_\Omega \bigg[
		\frac 1{\delta \beta'}\Big(1+
		\big(\ln_- u_\eps\big)^{\beta/(\beta-p)}\Big)-
		\frac{\ln_- u_\eps}{u_\eps^\theta}\bigg]
		\,\d x
		\\&\le 
		\int_\Omega
		fu_\eps
		\ln u_\eps\,\d x-
		\int_\Omega f\Lap u\,\d x+C_2\eps 
		\label{eq:step0}
	\end{align}
	where
	\begin{displaymath}
		C_2:=|\Omega|+\max_{0<r\le 1}
		\bigg[\frac 1{\delta\beta'}
		\Big(1+(\ln_- r)^{\beta/(\beta-p)}\Big)-
		\frac{\ln_- r}{r^\theta}\bigg].
	\end{displaymath}
	We can eventually estimate
	\begin{equation}
	\label{eq:fDelta-bound}
		-\int_\Omega f\Lap u\,\d x
		\le 
		\frac{c_d}2\int_\Omega |\rmD^2 u|^2\,\d x+
		\frac{d}{2c_d}\int_\Omega f^2\,\d x
			\end{equation}
	choosing $\sigma=2+4/d\in (2,2^*)$
	we can use Gagliardo-Nirenberg interpolation inequalities
	\begin{equation}
		\label{eq:GN}
		\|u_\eps\|_{L^\sigma}
		\le C_3 \|u_\eps\|_{L2}^{1-\theta}
		\|\nabla u_\eps\|_{L2}^\theta,\quad	
		\theta=d(1/2-1/\sigma),\quad 
		\sigma\theta=2,\ 
		\sigma(1-\theta)=4/d
	\end{equation}
	to get 
	\begin{displaymath}
		\|u_\eps\ln u_\eps\|_{L2}
		\le C_4 \Big(1+\|u_\eps\|^{\sigma/2}_{L^\sigma}\Big)
		\le C_5
		\bigg[1+
		 \|u_\eps\|_{L2}^{2/d}
		\|\nabla u_\eps\|_{L2}
		\bigg]
	\end{displaymath}
	Recalling that 
	$\|u_\eps\|_{L2}
	\le C(1+\|f\|_{L2}),$ 
	we thus get 
	\begin{displaymath}
		\int_\Omega f u_\eps\ln u_\eps\,\d x 
		\le C_7 (1+\|f\|_{L2}^{2+4/d})
		+\frac 12 \|\nabla u_\eps\|_{L2}^2
	\end{displaymath}
	Combining this inequality with
	\eqref{eq:fDelta-bound}
	and \eqref{eq:step0} we eventually get
	\eqref{eq:great}.
\end{proof}
\subsection{Maximality of \texorpdfstring{$\sfA$}{A}}
\begin{theorem}
	\label{thm:eta}
	For every $f\in L^2(\Omega)$ 
	the family $u_{\eps}$ 
	weakly converges in $H^2(\Omega)$
	to a (nonnegative) limit $u\in \HtN$
	as $\eps\down0$
	satisfying the 
	conditions 
	\eqref{eq:maximal1} and
	\eqref{eq:maximal2}
	of Corollary \ref{cor:equivalent-maximal}.
	In particular $\sfA$ is maximal monotone.
\end{theorem}
\begin{proof}
	By the 
	apriori estimates
	\eqref{eq:great} 
	and 
	\eqref{eq:basic-estimate}
	we can extract a subsequence 
	$n\mapsto \eps_n\down0$
	and limits $u\in \HtN$ such that 
	\begin{displaymath}
		u_{\eps_n}\weakto u\quad
		\text{in }\HtN.
	\end{displaymath}
	We also observe that 
	for every $z\in \ZO$ 
	\begin{equation}
	\label{eq:disappearing}	
		\eps
		\Big|\int_\Omega J_p(\nabla u_\eps)\nabla z\,\d x
		\Big|
		\le \eps
		\|\nabla u_\eps\|_p^{p-1}
		\|\nabla z\|_{p}
		\le \eps^{1/p} 
		(pa)^{(p-1)/p}
		\|\nabla z\|_{p}=o(1)\quad\text{as }\eps\down0
	\end{equation}
	thanks to \eqref{eq:simple-but-useful}.
	
	\eqref{eq:VE}
	yields for every $z\in \PVO$
	\begin{align}
\notag		\eps
	\int_\Omega J_p(\nabla u_\eps )\cdot \nabla z\,\d x 
		+\int_\Omega \Lap u_\eps\,\Lap z\,\d x
		&=
		\scalpt{f-u_\eps}{z}
		+\eps\theta  \int_\Omega 
		\frac 1{u_\eps^{\theta+1}}z\,\d x
	+\int_\Omega 
	\frac{(\Lap u_\eps)^2}{u_\eps}
	\,z\,\d x
	\\&\label{eq:trivial}\ge 
	\scalpt{f-u_\eps}{z}
	\end{align}
	Passing to the limit 
	in \eqref{eq:trivial}
	along the sequence $\eps_n$
	and using \eqref{eq:disappearing},
	we 
	get \eqref{eq:DLSS-syst1}.
	
	Now we want to pass to the limit
	in \eqref{eq:ap3}.
	We first observe that
	\eqref{eq:ap1} 
	and \eqref{eq:basic-estimate}
	yield
	\begin{displaymath}
		\eps \int_\Omega \frac1{u_\eps^{\theta+1}}
		\,\d x
		\le C_1
	\end{displaymath}
	for some constant $C_1$ independent of $\eps$.
	We obtain 
	\begin{displaymath}
		\eps \int_\Omega \frac1{u_\eps^\theta}
		\,\d x
		\le C_2\eps 
		\Big(\int_\Omega \frac1{u_\eps^{\theta+1}}
		\,\d x
		\Big)^{\theta/(\theta+1)}
		= C_2
		\eps^{1/(\theta+1)}	\Big(\,\eps \int_\Omega \frac1{u_\eps^{\theta+1}}
		\,\d x
		\Big)^{\theta/(\theta+1)}
		\le C_2C_1^{\theta/(\theta+1)}\eps^{1/(\theta+1)}
		=o(1)
	\end{displaymath}
	as $\eps\downarrow0$.
	
	Similarly,
	\begin{align*}
		\eps \Big|\int_\Omega 
		J_p(\nabla u_\eps)\cdot \nabla(u_\eps
		w)\,\d x\Big|
		&\le 
		\eps \|w\|_{L^\infty}\int_\Omega 
		|\nabla u_\eps|^p\,\d x+
		\eps \lip(w)
		\int_\Omega 
		|\nabla u_\eps|^{p-1}u_\eps\,\d x
		\\&\le 
		C(w)\eps\bigg(
		\pnrm{\nabla u_\eps}^p
		+
		\pnrm{u_\eps}^p\bigg)
	\end{align*}
	We use the apriori bounds
	on $\|\nabla u_\eps\|^p_{L^p(\Omega)}+
	\rmE_p(u_\eps)^p\le C/\eps$
	of \eqref{eq:great}
	and \eqref{eq:simple-but-useful}
	to 
	derive
	$\|\nabla u_\eps\|^p_{L^{\tilde p}(\Omega)}
	\le C_{\tilde p}/\eps$
	for $\tilde p=pd/(d-1)$,
	thanks to 
	\eqref{eq:extra-integrability}.
	
	On the other hand
	$\|\nabla u_\eps\|_{L2}\le C_2$ 
	so that 
	choosing $\theta\in (0,1)$ 
	so that $\frac 1p=(1-\theta)\frac 12
	+\theta \frac 1{\tilde p}$ we get
	\begin{displaymath}
		\|\nabla u_\eps\|_{L^p}\le 
		\|\nabla u_\eps\|_{L2}^{1-\theta}
		\|\nabla u_\eps\|_{L^{\tilde p}}^{\theta}\le
		 C_2C_{\tilde p}\eps^{-\theta/p}
	\end{displaymath}
	so that 
	\begin{displaymath}
		\eps 
		\|\nabla u_\eps\|_{L^p}^p
		\le C\eps^{1-\theta}\to0
		\ \text{as }\eps\down0.
	\end{displaymath}
	The uniform bound of $u_\eps$ in $L^2(\Omega)$ and
	the previous estimate also shows that 
	$\eps\pnrm{u_\eps}^p\to 0$ as 
	$\eps\down0$.
	We can now pass to the limit in 
	\eqref{eq:ap3} obtaining
	\eqref{eq:DLSS-syst2}.
	
	By Corollary 
	\ref{cor:equivalent-maximal}
	we deduce that 
	$(u,u-f)$ 
	belongs to the graph of $\sfA$,
	so that 
	uniqueness of $u$ 
	and the limit of the whole family $u_\eps$ follow by the monotonicity of $\sfA$.
\end{proof}
\subsection{Maximality of \texorpdfstring{$\frA$}{A} in the \texorpdfstring{$H^2$--$H^{-2}$}{H2-H-2} duality}
\label{subsec:frA-maximality}
\begin{proof}[Proof of Theorem~\ref{thm:frA-maximal}]
	By Theorem~\ref{thm:Minty} it is enough to solve
	$\lambda\rmJ v+\frA v\ni f$ for every $\lambda>0$ and every
	$f\in\DHtN$. We repeat the construction of
	Theorems~\ref{thm:existence} and~\ref{thm:eta} with the following
	changes.

	The datum enters through the continuous functional
	$v\mapsto\la f,v\ra$ on $\VO\subset\HtN$, and the $L^2$ coercive
	term $\scalpt{u-f}{u-v}$ is replaced by
	$\lambda\la\rmJ u-f,u-v\ra$; accordingly one sets
	$\mathsf f(u,v):=\lambda\la\rmJ u-f,u-v\ra+\fra(u,u-v)$. The
	regularizer $\eps\Phi$ of \eqref{eq:Phi} is unchanged and still
	forces $\min_{\overline\Omega}u_\eps>0$ (property \eqref{eq:F2}).

	Hypotheses {\rm(VI.1)--(VI.3)} hold as before, with
	$\la\rmJ(u-v),u-v\ra=\|u-v\|_{\HtN}^2\ge0$ in place of
	$\scalpt{u-v}{u-v}$ in {\rm(VI.1)}. For the coercivity {\rm(VI.4)}
	we test with $v_0\equiv1\in\dom(\Phi)$: since
	$\int_\Omega\Lap u=0$ and
	$\fra(u,u-1)=\int_{\{u>0\}}(\Lap u)^2/u\,\d x\ge0$,
	\begin{displaymath}
		\mathsf f(u,1)\ge\lambda\la\rmJ u,u-1\ra
		=\lambda\Big(\|u-\Lap u\|_{L2}^2-\int_\Omega u\Big)
		\ge\lambda\big(c\,\|u\|_{\HtN}^2-|\Omega|^{1/2}\tnrm u\big),
	\end{displaymath}
	so the sublevel set is bounded in $\HtN$ uniformly in $\eps$, and,
	for each fixed $\eps$, in $\VO$ together with the $\eps\Phi$ term.
	The coercivity is thus supplied by $\rmJ$ alone: the entropy
	estimate of Theorem~\ref{thm:existence} plays no role, and only the
	tests against $1$ and against $u$ enter, exactly as in
	\eqref{eq:maximal1}--\eqref{eq:maximal2}.

	Theorem~\ref{thm:BC-VI} then yields $u_\eps\in\dom(\Phi)$ solving
	the corresponding variational inequality; the limit $\eps\down0$ is
	taken as in Theorem~\ref{thm:eta} (the term
	$\eps\,p^{-1}\pnrm{\nabla u_\eps}^p$ vanishing as in
	\eqref{eq:disappearing}, the defect being lower semicontinuous by
	Ioffe's theorem) and gives $u\in\HtN$ with
	$\lambda\rmJ u+\frA u\ni f$. Uniqueness and the resolvent estimate
	\eqref{eq:resolvent-estimate-W} follow from
	Theorem~\ref{thm:Minty}.
\end{proof}
The two maximality properties are logically independent (cf.~the
discussion after Proposition~\ref{prop:restriction-maximality}): here
the duality map $\rmJ$ supplies the $\HtN$-coercivity that, for the
$L^2$ realization, had to be extracted from the entropy and
Gagliardo--Nirenberg estimates.

\appendix

\section{The Neumann Laplacian in \texorpdfstring{$L^1(\Omega)$}{L1}}
\label{app:sec2-technical}

\begin{proof}[Proof of Lemma~\ref{le:useful-approximation}]
  The density of $\ZO$ in $\HtN$ follows by the estimates~\eqref{eq:HeatEstimates}, since $\sfH_t u\in \ZO$ for every $u\in \HtN$ and
  $t>0$,
  with $\sfH_tu\to u$ in $\HtN$ as $t\downarrow0$.

  $\ZO$ is also an algebra:
  if $u,w\in \ZO$ the Leibniz rule yields
  \begin{align}
  	\label{eq:Leibniz1}
    \nabla (uw)&=u\nabla w+w\nabla u\in L^\infty(\Omega),\\
    \label{eq:Leibniz2}
    \Lap(uw)&=u\Lap w+2\nabla u\cdot\nabla w+w\Lap u\in L^\infty(\Omega)
  \end{align}
  so that $uw\in \ZO$ as well.
  \eqref{eq:Leibniz2}
  also shows that
  \eqref{eq:Z-acts}
  holds.

  In order to show that $\ZO$ separates the points of
  $\overline\Omega$ it is sufficient to observe that
  for every $f\in \Lip(\overline\Omega)$
  $\sfH_t f\in \ZO$ and $\sfH_t f\to f$ uniformly in
  $\overline\Omega$.
  For every $x_0,x_1\in \overline\Omega$ we can find
  $f\in \Lip(\overline\Omega)$ such that $f(x_0)\neq f(x_1)$
  so that for sufficiently small $t>0$ we get
  $\sfH_t f(x_0)\neq \sfH_tf(x_1)$.

  Since $\Omega$ is bounded, we also have that $1\in \ZO$ and hence by Stone-Weierstraß also the density of $\ZO$ in $\rmC(\overline\Omega)$.
\end{proof}

\begin{proof}[Proof of Proposition~\ref{prop:L1-closure}]
  If $u_n\in\ZO$, $u_n\to u$ and $\Lap u_n\to g$ in $L^1(\Omega)$, then
  for every $\zeta\in\ZO$ Green's formula gives
  $\int_\Omega u_n\,\Lap\zeta\,\d x=\int_\Omega \Lap u_n\,\zeta\,\d x$;
  passing to the limit, $u\in D_{L1}(\Lap)$ with $\Lap u=g$. This proves
  that the graph-closure is contained in $D_{L1}(\Lap)$.

  Conversely, let $u\in D_{L1}(\Lap)$ with $\Lap u=g$ and set
  $u_\tau:=\sfH_\tau u\in\ZO$, $\tau>0$. By the self-adjointness of
  $\sfH_\tau$ in $L^2(\Omega)$ (extended by density to the pairing
  between $L^1(\Omega)$ and $L^\infty(\Omega)$) and the commutation
  $\sfH_\tau(\Lap\zeta)=\Lap(\sfH_\tau\zeta)$ for $\zeta\in\ZO$, we get
  \begin{equation*}
    \int_\Omega u_\tau\,\Lap\zeta\,\d x
    =\int_\Omega u\,\sfH_\tau(\Lap\zeta)\,\d x
    =\int_\Omega u\,\Lap(\sfH_\tau\zeta)\,\d x
    =\int_\Omega g\,\sfH_\tau\zeta\,\d x
    =\int_\Omega (\sfH_\tau g)\,\zeta\,\d x
  \end{equation*}
  for every $\zeta\in\ZO$. Since $u_\tau\in\ZO$, Green's formula applied
  to $u_\tau$ then forces $\Lap u_\tau=\sfH_\tau g$. As $\tau\downarrow0$,
  $u_\tau\to u$ and $\Lap u_\tau=\sfH_\tau g\to g$ in $L^1(\Omega)$ by the
  strong continuity of $(\sfH_\tau)_{\tau\ge0}$ in $L^1(\Omega)$, so $u$
  belongs to the graph-closure of $\ZO$.
\end{proof}

 \begin{proof}[Proof of the integration-by-parts identity \eqref{eq:L1-IBP}]
  Let $\zeta\in\ZO$. Then $\nabla\zeta\in H^1(\Omega;\R^d)\subset
  W^{1,1}(\Omega;\R^d)$ with $\mathrm{div}(\nabla\zeta)=\Lap\zeta\in
  L^\infty(\Omega)$ and $\partial_\nn\zeta=0$ on $\partial\Omega$. Since
  $u\in W^{1,1}(\Omega)$, the classical (first-order) Green's formula
  gives
  \begin{equation*}
    \int_\Omega u\,\Lap\zeta\,\d x=-\int_\Omega \nabla u\cdot\nabla\zeta\,\d x,
  \end{equation*}
  the boundary term $\int_{\partial\Omega}u\,\partial_\nn\zeta\,\d S$
  vanishing identically because $\partial_\nn\zeta\equiv0$, irrespective
  of any trace of $u$. Combined with the defining identity
  $\int_\Omega u\,\Lap\zeta\,\d x=\int_\Omega g\,\zeta\,\d x$
  (Definition~\ref{def:L1-Laplacian}), this proves
  \eqref{eq:L1-IBP} for $\zeta\in\ZO$.

  Let now $\zeta\in\Lip(\overline\Omega)$ and set
  $\zeta_t:=\sfH_t\zeta\in\ZO$, $t>0$. By the Lipschitz contraction
  property of $\sfH_t$, $\|\nabla\zeta_t\|_{L^\infty(\Omega)}\le
  \lip(\zeta)$ for every $t>0$; by strong continuity of $(\sfH_t)_{t\ge0}$
  on the form domain $H^1(\Omega)$, $\nabla\zeta_t\to\nabla\zeta$ in
  $L^2(\Omega;\R^d)$ as $t\downarrow0$, hence, along a subsequence,
  $\nabla\zeta_t\to\nabla\zeta$ a.e.\ in $\Omega$. Since \eqref{eq:L1-IBP}
  holds for each $\zeta_t\in\ZO$, dominated convergence (using
  $g\in L^1(\Omega)$ on the left and $\nabla u\in L^1(\Omega;\R^d)$,
  dominated by $\lip(\zeta)|\nabla u|$, on the right) passes to the limit
  along that subsequence; since the limit $-\int_\Omega \nabla
  u\cdot\nabla\zeta\,\d x$ does not depend on the subsequence, the full
  family converges and \eqref{eq:L1-IBP} holds for~$\zeta$.
\end{proof}

\section{Auxiliary regularization estimates: proofs}
\label{app:sparse-proofs}
We collect here the proofs of the regularization estimates for the
Heat flow stated in Section~\ref{subsec:Heat}.
\begin{proof}[Proof of Lemma~\ref{le:convex-estimate}]
	\eqref{eq:pointwise-convex-estimate}
	is equivalent to 
	\begin{equation}
		\label{eq:pointwise-convex-estimate2}
		\int_\Omega \Theta\big(\uu\big)\,\sfH_t w
		\,\d x 
		=\int_\Omega \sfH_t(\Theta\circ\uu)w
		\,\d x 
		\ge 
		\int_\Omega \Theta\big(\sfH_t \uu\big) w\,\d x
		\quad \text{for every }w\in L^\infty_+(\Omega).
	\end{equation}
	By approximation, it is not restrictive to assume that $\Theta$ is of class $\rmC^2$ with
	bounded first and second derivatives.
	We then introduce the functions
	$\uu_t=\sfH_t \uu,\ w_t:=\sfH_t w$ and
	\begin{displaymath}
		\zeta(s):=
		\int_\Omega \Theta\big(\sfH_{t-s}\uu\big)\,\sfH_s w
		\,\d x,\quad s\in [0,t]
	\end{displaymath}
	and we observe that \eqref{eq:pointwise-convex-estimate2} corresponds to 
	\begin{displaymath}
		\zeta(t)\ge \zeta(0).
	\end{displaymath}
	Since $\zeta\in \rmC([0,t])\cap \rmC^2((0,t))$ 
	it is sufficient to prove that $\frac\d{\d s}\zeta(s)\ge 0$ in $(0,t)$.
	Since 
	\begin{displaymath}
		\frac{\d}{\d s}w_s=\Lap w_s,\quad
		\frac{\d}{\d s}\Theta(\uu_{t-s})=
		\sum_i\partial_i \Theta(\uu_{t-s})\partial_s u^i_{t-s}=
		-\sum_i\partial_i \Theta(\uu_{t-s})\Lap u^i_{t-s},
	\end{displaymath}
	\begin{align*}
		\Lap \Theta(\uu_{t-s})
		&=
		\sum_k \partial^2_{x_k}\Theta(\uu_{t-s})
		=
		\sum_k \partial_{x_k}
		\Big(\sum_i \partial_i \Theta(\uu_{t-s})
		\partial_{x_k}u^i_{t-s}\Big)
		\\&=
		\sum_k 
		\Big(\sum_i \partial_i \Theta(\uu_{t-s})
		\partial^2_{x_k}u^i_{t-s}
		+
		\sum_{i,j} \partial^2_{ij}\Theta(\uu_{t-s}) \partial_{x_k}u^i_{t-s}
		\partial_{x_k}u^j_{t-s}\Big)
		\\&=
		\sum_i \partial_i \Theta(\uu_{t-s})
		\Lap u^i_{t-s}
		+
		\sum_k \Big(\sum_{i,j} \partial^2_{ij}\Theta(\uu_{t-s}) \partial_{x_k}u^i_{t-s}
		\partial_{x_k}u^j_{t-s}\Big)
		\ge \sum_i \partial_i \Theta(\uu_{t-s})
		\Lap u^i_{t-s}
	\end{align*}
	where we use 
	$\sum_{i,j} \partial^2_{ij}\Theta(\uu_{t-s}) \xi^i\xi^j\ge 0$ 
	for every choice of vectors $\xxi=(\xi^i)$
	thank to the convexity of $\Theta$.
	We can eventually compute
	\begin{align*}
		\zeta'(s)&=
		\int_\Omega \Theta(\uu_{t-s})\Lap w_s\,\d x
		-
		\sum_i
		\int_\Omega \partial_i \Theta(\uu_{t-s})\Lap u^i_{t-s}\,w_s \,\d x
		\\&=
		\int_\Omega 
		\Big(\Lap \Theta(\uu_{t-s}) 
		-
		\sum_i
		\partial_i \Theta(\uu_{t-s})\Lap u^i_{t-s}\Big)\, w_s \,\d x\ge 0.
	\end{align*}
	\end{proof}
\begin{proof}[Proof of Corollary~\ref{cor:main-ineq}]
	It is sufficient to apply the previous Lemma
	choosing $\Theta(a,b):=\sfh(a,b)$
	and $\uu=(u,\Lap u)$.
\end{proof}
\begin{proof}[Proof of Corollary~\ref{cor:main-ineq2}]
	We have to prove the inequality 
	\begin{equation}
	\label{eq:equivalence-main}
		\int_\Omega \Lap u\,\Lap \sfH_t u\,\d x\le 
		\int_{u>0}\frac{(\Lap u)^2}u
		\sfH_t u\,\d x.
	\end{equation}
	Writing $t=2\tau$ 
	and using the commutation between $\Lap$ and $\sfH_\tau$ we have
	\begin{align*}
		\int_{u>0}
		\frac{(\Lap u)^2}u
		\sfH_t u\,\d x
		&=
		\int_{u>0}
		\frac{(\Lap u)^2}u
		\sfH_\tau(\sfH_\tau u)\,\d x
		\ge 
		\int_\Omega 
		\frac{(\Lap \sfH_\tau u)^2}{\sfH_\tau u}
		(\sfH_\tau u)\,\d x
		=
		\int_\Omega 
		{(\Lap \sfH_\tau u)^2}\,\d x
	\end{align*}
	and 
	\begin{align*}
		\int_\Omega 
		\Lap u\Lap \sfH_t u\,\d x
		&=
		\int_\Omega 
		\Lap u\Lap 
		\big[\sfH_\tau(\sfH_\tau  u)\big]\,\d x
		=
		\int_\Omega 
		\Lap (\sfH_\tau u)\, \Lap (\sfH_\tau  u)\,\d x=
		\int_\Omega 
		{(\Lap \sfH_\tau u)^2}\,\d x.
	\end{align*}
	We thus get \eqref{eq:equivalence-main}.
\end{proof}

\section{Second order calculus for nonnegative functions: proofs}
\label{app:second-order}

\newcommand{\myH}{{\boldsymbol \sfH}}
\newcommand{\myT}{{\boldsymbol \sfT}}
\newcommand{\myG}{{\boldsymbol \sfG}}
\newcommand{\myTG}{{\boldsymbol \sfP}}
\newcommand{\myHG}{{\boldsymbol \sfQ}}

In this appendix we prove the results stated in
Section~\ref{subsec:second-order}. Throughout, $w$ denotes a
nonnegative function in $\HtN$,
$H(x):=\rmD^2 w(x)$, $\Lap w(x)=
\operatorname{tr} H(x)$,
 $\gg(x):=\nabla w(x)/\sqrt{w(x)}$ ($\gg(x)=0$ where $w(x)=0$), and we use the shorthand
 \begin{equation}
	\label{eq:PHRCQ}
	\begin{aligned}
	\myH:={}&\int_\Omega|\rmD^2 w|^2\,\d x=
	\int_\Omega|H(x)|^2\,\d x
	,\qquad
	\myT:=\int_\Omega(\Lap w)^2\,\d x=
	\int_\Omega|\operatorname{tr} H(x)|^2\,\d x,\\
	\myG:={}&\int_{\{w>0\}}\frac{|\nabla w|^4}{w^2}\,\d x=
	\int_\Omega|\gg(x)|^4\,\d x,\\
	\myTG:={}&\int_{\{w>0\}}\Lap w\,\frac{|\nabla w|^2}{w}\,\d x=
	\int_{\Omega}\operatorname{tr} H(x)\,|\gg(x)|^2\,\d x
	,\\
	\myHG:={}&\int_{\{w>0\}}\frac{\la\rmD^2 w\,\nabla w,\nabla w\ra}{w}\,\d x
	=\int_\Omega\la H(x)\,\gg(x),\gg(x)\ra\,\d x.
	\end{aligned}
\end{equation}
whenever the corresponding integrals are meaningful.
Notice that $\myT\le d\,\myH$, since
$(\operatorname{tr}H)^2\le d\,|H|^2$ for every symmetric matrix
$H\in\R^{d\times d}$.
We will also use the positive cone of the test class $\ZO$ of~\eqref{eq:defZ},
\begin{equation}
	\label{eq:defTest}
	\PZO:=\bigl\{v\in\ZO:\ \essinf_\Omega v>0\bigr\};
\end{equation}
if $w\in\PZO$ then all five quantities in \eqref{eq:PHRCQ} are finite,
since $\rmD^2 w\in L^2(\Omega)$, $\nabla w\in L^\infty(\Omega)$ and $w$
is bounded away from $0$.

\begin{lemma}[Square roots in the positive cone]
	\label{le:stability-app}
	If $u\in\PZO$ then $w:=\sqrt u\in\PZO$, with
	\begin{equation}
		\label{eq:chain-positive}
		\nabla w=\frac{\nabla u}{2\sqrt u},\qquad
		\rmD^2 w=\frac{\rmD^2 u}{2\sqrt u}
		-\frac{\nabla u\otimes\nabla u}{4\,u^{3/2}},\qquad
		\Lap w=\frac{\Lap u}{2\sqrt u}
		-\frac{|\nabla u|^2}{4\,u^{3/2}} .
	\end{equation}
\end{lemma}
\begin{proof}
	Since $m:=\essinf_\Omega u>0$ and $u\in \Lip(\overline\Omega)$,
	the function $\varphi(s):=\sqrt s$ is smooth with bounded
	derivatives of every order on $[m,\sup u]$, so the chain rule for
	Sobolev functions gives \eqref{eq:chain-positive}; in particular
	$\rmD^2w\in L^2(\Omega;\R^{d\times d})$,
	$\nabla w\in L^\infty(\Omega;\R^d)$, $\Lap w\in
	L^\infty(\Omega)$ and $\essinf w=\sqrt m>0$.
	Concerning the boundary condition, $u$ has a Lipschitz
	representative with $u\ge m>0$ on $\overline\Omega$ and the trace
	of $\nabla w=\nabla u/(2\sqrt u)$ on $\partial\Omega$ is
	$\mathrm{tr}(\nabla u)/(2\sqrt{u\restr{\partial\Omega}})$, so that
	$\partial_\nn w=\partial_\nn u/(2\sqrt
	{u\restr{\partial\Omega}})=0$.
\end{proof}

\begin{lemma}[First integration by parts rule]
	\label{le:rule1-app}
	For every $w\in\PZO$ we have
	\begin{equation}
		\label{eq:rule1}
		\myG=\myTG+2\myHG  .
	\end{equation}
\end{lemma}
\begin{proof}
	Consider the vector field
	$V:=w^{-1}|\nabla w|^2\,\nabla w$. Each factor $w^{-1}$,
	$|\nabla w|^2$, $\partial_i w$ belongs to $H^1(\Omega)\cap
	L^\infty(\Omega)$
    hence $V\in H^1(\Omega;\R^d)\cap
	L^\infty(\Omega;\R^d)$, with
	\begin{equation}
		\label{eq:divV}
		\mathrm{div}\,V=
		2\,\frac{\la\rmD^2 w\,\nabla w,\nabla w\ra}{w}
		+\Lap w\,\frac{|\nabla w|^2}{w}
		-\frac{|\nabla w|^4}{w^2}\in L^1(\Omega)
	\end{equation}
	by the Leibniz rule.
	The Gauss--Green formula on the Lipschitz domain $\Omega$ gives
	$\int_\Omega \mathrm{div}\,V\,\d x=\int_{\partial\Omega}
	\mathrm{tr}(V)\cdot\nn\,\d\sigma$; since traces of $H^1\cap
	L^\infty$ functions multiply,
	$\mathrm{tr}(V)\cdot\nn=\mathrm{tr}\big(w^{-1}|\nabla
	w|^2\big)\,\partial_\nn w=0$.
	Integrating \eqref{eq:divV} we obtain \eqref{eq:rule1}.
\end{proof}

\begin{lemma}[Pointwise sum of squares inequality]
	\label{le:SOS}
	Let $d\ge2$ and set
	\begin{equation}
		\label{eq:SOS-choice}
		\lambda:=1-\frac1d=\frac{d-1}d,\qquad
		\mu:=2-\frac1d=\frac{2d-1}d,\qquad
		\kappa:=\frac1{d^2}.
	\end{equation}
	Then for every symmetric $H\in\R^{d\times d}$ and every $\gg \in\R^d$
	\begin{equation}
		\label{eq:SOS}
		F(H,\gg ):=(1-\lambda)|H|^2+\lambda(\operatorname{tr}H)^2
		+(2\mu-4)\la H\gg ,\gg \ra
		+\mu\,(\operatorname{tr}H)\,|\gg |^2+(3-\mu)\,|\gg |^4
		\;\ge\;\kappa\,|H|^2 .
	\end{equation}
\end{lemma}
\begin{proof}
	If $\gg =0$ then
	$F(H,0)\ge(1-\lambda)|H|^2=\frac1d|H|^2\ge\kappa|H|^2$.
	Let $\gg \neq0$; since both sides of \eqref{eq:SOS} are invariant
	under the substitution $(H,\gg )\mapsto(OHO^{\top},O\gg )$ for an
	orthogonal matrix $O$, we may assume $\gg =\sqrt r\,e_1$ with
	$r=|\gg |^2>0$, so that $\la H\gg ,\gg \ra=H_{11}r$.
	The off-diagonal entries of $H$ appear in $F(H,\gg )-\kappa|H|^2$
	only through the term $(1-\lambda-\kappa)\sum_{i\neq
	j}H_{ij}^2$, whose coefficient
	$1-\lambda-\kappa=\frac1d-\frac1{d^2}=\frac{d-1}{d^2}$ is
	nonnegative: $F(H,\gg)$ is therefore smallest for diagonal matrices
	$H=\mathrm{diag}(a_1,\dots,a_d)$.
	Setting $z:=a_2+\dots+a_d$ and using
	$\sum_{i\ge2}a_i^2\ge z^2/(d-1)$, a direct computation shows that
	it is sufficient to prove that, for all real $a_1,r,z$,
	\begin{equation}
		\label{eq:G-quadratic}
		G(a_1,r,z):=(1-\kappa)a_1^2+(3-\mu)r^2+\delta z^2
		+(3\mu-4)\,a_1r+2\lambda\,a_1z+\mu\,zr\;\ge\;0,
	\end{equation}
	where $\delta:=\lambda+\frac{1-\lambda-\kappa}{d-1}
	=\frac{d^2-d+1}{d^2}$.
	$G$ is the quadratic form of the symmetric matrix
	\begin{equation*}
		M=\begin{pmatrix}
			1-\kappa & \tfrac{3\mu-4}2 & \lambda\\[2pt]
			\tfrac{3\mu-4}2 & 3-\mu & \tfrac\mu2\\[2pt]
			\lambda & \tfrac\mu2 & \delta
		\end{pmatrix}
		\qquad\text{in the variables }(a_1,r,z).
	\end{equation*}
	With the choice \eqref{eq:SOS-choice} the principal minors of $M$
	are
	\begin{equation}
		\label{eq:minors}
		m_1=\frac{d^2-1}{d^2},\qquad
		m_2=\frac{16d^2-13d-4}{4d^3},\qquad
		m_3=\det M=\frac{11d^2-8d-4}{4d^5}:
	\end{equation}
	indeed, $3-\mu=\frac{d+1}d$ and $3\mu-4=\frac{2d-3}d$ give
	\begin{displaymath}
		m_2=\frac{(d-1)(d+1)^2}{d^3}-\frac{(2d-3)^2}{4d^2}
		=\frac{4(d^3+d^2-d-1)-d(4d^2-12d+9)}{4d^3}
		=\frac{16d^2-13d-4}{4d^3},
	\end{displaymath}
	while in the expansion of $4d^5\det M$ over the common
	denominator all the terms of order $d^5,d^4,d^3$ cancel
	identically, leaving $11d^2-8d-4$.
	Both numerators in \eqref{eq:minors} are strictly positive for
	every $d\ge2$, so $M$ is positive definite by Sylvester's
	criterion and \eqref{eq:G-quadratic} holds.
\end{proof}

\begin{proposition}[Estimates on the positive cone]
	\label{prop:positive-case}
	For $w\in\PZO$ set
	\begin{equation}
		\label{eq:defcalJ}
		\mathcal J(w):=\int_\Omega
		\Big(\Lap w+\frac{|\nabla w|^2}{w}\Big)^{\!2}\,\d x
		=\myT +2\myTG +\myG  .
	\end{equation}
	Then
	\begin{align}
		\label{eq:sharp-positive}
		\mathcal J(w)&\ge\frac1{d^2}\,\myH ,\\
		\label{eq:R-rough}
		\sqrt \myG &\le\sqrt \myT +2\sqrt \myH ,\\
		\label{eq:R-refined}
		\myG &\le\frac{16}9\,\myH +\frac23\,\mathcal J(w).
	\end{align}
	Moreover, if $u\in\PZO$ and hence $w=\sqrt u\in \PZO$ by
	Lemma~\ref{le:stability-app}, it follows that
	$\sfD(u)=4\,\mathcal J(\sqrt u\,)$.
\end{proposition}
\begin{proof}
	The identity in \eqref{eq:defcalJ} follows by expanding the
	square; all three terms are finite on $\PZO$.
	The last statement follows from
	$\Lap u=2w\Lap w+2|\nabla w|^2$
	(Lemma~\ref{le:stability-app} applied to $w^2=u$), which gives
	$(\Lap u)^2/u=4\big(\Lap w+|\nabla w|^2/w\big)^2$ pointwise.

	\emph{Proof of \eqref{eq:sharp-positive}, case $d\ge2$.}
	Adding $\mu-2$ times equation \eqref{eq:rule1}, i.e.~$(\mu-2)(\myTG +2\myHG -\myG )=0$,
	to \eqref{eq:defcalJ} we obtain
	\begin{equation}
		\label{eq:mu-form}
		\mathcal J(w)=\myT +\mu\,\myTG +(2\mu-4)\,\myHG +(3-\mu)\,\myG  .
	\end{equation}
	We split $\myT =\lambda \myT +(1-\lambda)\myT $ and recall 
	inequality \eqref{eq:H2bound}, $\myT \ge \myH $. 
    Since $1-\lambda=\frac1d$ is non-negative:
	\begin{equation*}
		\mathcal J(w)\ \ge\
		\int_\Omega\Big[\lambda(\Lap w)^2+(1-\lambda)|\rmD^2w|^2
		+\mu\,\Lap w\,\frac{|\nabla w|^2}{w}
		+(2\mu-4)\,\frac{\la\rmD^2w\,\nabla w,\nabla w\ra}{w}
		+(3-\mu)\,\frac{|\nabla w|^4}{w^2}\Big]\,\d x .
	\end{equation*}
	The integrand is $F\big(H(x),\gg(x)\big)$
	with $F$ as in \eqref{eq:SOS}.
	Lemma~\ref{le:SOS} then yields $\mathcal J(w)\ge\kappa \myH $ with
	$\kappa=1/d^2$ for $d\geq 2$.

	\emph{Case $d=1$.} Integrating by parts
	$\int_\Omega w''\frac{(w')^2}w\,\d x=
	\int_\Omega \frac{\big(\frac13(w')^3\big)'}{w}\,\d x
	=\frac13\int_\Omega\frac{(w')^4}{w^2}\,\d x$
	(the boundary term $\frac{(w')^3}{3w}$ vanishes since $w'=0$ on
	$\partial\Omega$), i.e.~$\myTG =\frac13\myG $, so that by
	\eqref{eq:rule1} also $\myHG =\frac13\myG $ and
	\begin{equation}
		\label{eq:d1-identity}
		\mathcal J(w)=\myT +\tfrac53\myG \ \ge\ \myT =\myH  ,
	\end{equation}
	which is \eqref{eq:sharp-positive} for $d=1$ and, after
	multiplication by $4$, the identity of Remark~\ref{rem:d1}.

	\emph{Proof of \eqref{eq:R-rough}.}
	By the pointwise Cauchy--Schwarz inequality
	$|\la\rmD^2w\,\nabla w,\nabla w\ra|\le|\rmD^2w||\nabla w|^2$ and
	the Cauchy--Schwarz inequality in $L^2$,
	\begin{equation}
		\label{eq:CQ-bounds}
		|\myTG |\le\sqrt{\myT \,\myG },\qquad
		|\myHG |\le\sqrt{\myH \,\myG } .
	\end{equation}
	By \eqref{eq:rule1}, $\myG =\myTG +2\myHG \le \sqrt \myG \big(\sqrt \myT +2\sqrt
	\myH \big)$; if $\myG >0$ we divide by $\sqrt \myG $, and \eqref{eq:R-rough}
	is trivial if $\myG =0$.

	\emph{Proof of \eqref{eq:R-refined}.}
	Subtracting twice equality \eqref{eq:rule1} from \eqref{eq:defcalJ} yields
	$\mathcal J(w)=\myT -4\myHG +3\myG $, hence by \eqref{eq:CQ-bounds} and
	$\myT \ge0$,
	\begin{equation*}
		3\myG =\mathcal J(w)-\myT +4\myHG 
        \le \mathcal J(w)+4\sqrt \myH \sqrt \myG 
        \le \mathcal J(w) + \frac83\myH + \frac32 \myG.
	\end{equation*}
        Subtract $\frac32\myG$ and multiply by $\frac23$ to obtain \eqref{eq:R-rough}.
\end{proof}

\begin{proof}[Proof of Lemma~\ref{le:sqrt-calculus}]
	\emph{(a)} By the standard level set property of Sobolev
	functions, $\nabla w=0$ $\Leb d$-a.e.~on $\{w=0\}$; applying the
	same property to each $\partial_i w\in H^1(\Omega)$ on the
	smaller set $\{w=0\}\subset\{\partial_iw=0\}$ (up to a $\Leb
	d$-negligible set) we get $\rmD^2w=0$ $\Leb d$-a.e.~on
	$\{w=0\}$.

	\emph{(b)} Since $w\in L^2(\Omega)$ and $\nabla w\in
	L^2(\Omega;\R^d)$ we have $2w\nabla w\in L^1(\Omega;\R^d)$; the
	chain rule $\nabla(w^2)=2w\nabla w$ holds for the truncations
	$w_M:=w\wedge M\in H^1\cap L^\infty$ and passes to the limit
	$M\uparrow+\infty$ by dominated convergence.

	\emph{(c)} Let $\zeta\in\ZO$. Since $u\in W^{1,1}(\Omega)$ and
	$\partial_\nn\zeta=0$, the first order Green formula (as in the
	proof of \eqref{eq:L1-IBP} in
	Appendix~\ref{app:sec2-technical}) gives
	\begin{equation}
		\label{eq:step-c1}
		\int_\Omega u\,\Lap\zeta\,\d x=
		-\int_\Omega \nabla u\cdot\nabla\zeta\,\d x=
		-2\int_\Omega w\nabla w\cdot\nabla\zeta\,\d x .
	\end{equation}
	Consider the truncated vector fields $V_M:=(w\wedge M)\nabla w$,
	$M>0$. Since $w\wedge M\in H^1(\Omega)\cap L^\infty(\Omega)$ and
	$\nabla w\in H^1(\Omega;\R^d)$, we have $V_M\in
	W^{1,1}(\Omega;\R^d)$ with
	$\mathrm{div}\,V_M=\mathbf 1_{\{w<M\}}|\nabla w|^2+(w\wedge M)\Lap w
	\in L^1(\Omega)$,
	and the trace of $V_M$ on $\partial\Omega$ is the product of the
	traces of its factors, so that
	$\mathrm{tr}(V_M)\cdot\nn=(w\wedge
	M)\restr{\partial\Omega}\,\partial_\nn w=0$.
	The Gauss--Green formula then gives
	$\int_\Omega V_M\cdot\nabla\zeta\,\d x=
	-\int_\Omega \mathrm{div}(V_M)\,\zeta\,\d x$, and we can pass to
	the limit $M\uparrow+\infty$ by dominated convergence, since
	$|V_M|\le w|\nabla w|\in L^1$ and
	$|\mathrm{div}\,V_M|\le|\nabla w|^2+w|\Lap w|\in L^1$:
	\begin{equation}
		\label{eq:step-c2}
		\int_\Omega w\nabla w\cdot\nabla\zeta\,\d x=
		-\int_\Omega\big(|\nabla w|^2+w\Lap w\big)\,\zeta\,\d x .
	\end{equation}
	Combining \eqref{eq:step-c1} and \eqref{eq:step-c2} we obtain
	$\int_\Omega u\Lap\zeta\,\d x=\int_\Omega\big(2w\Lap w+2|\nabla
	w|^2\big)\zeta\,\d x$ for every $\zeta\in\ZO$, which is
	\eqref{eq:chain-Lap} by Definition~\ref{def:L1-Laplacian}. The
	vanishing of $\Lap u$ on $\{u=0\}=\{w=0\}$ follows from (a).

	\emph{(d)} follows from \eqref{eq:chain-Lap} and $u=w^2$ by
	direct computation on $\{u>0\}$.
\end{proof}

\begin{proof}[Proof of Proposition~\ref{prop:quartic}]	
	With the notation~\eqref{eq:PHRCQ}, the statements to be proven~\eqref{eq:quartic} and~\eqref{eq:D-upper} become
	\begin{equation}\label{eq:quartic:PHRCQ}
		\myG \leq \bigl(\sqrt{\myT} + 2 \sqrt{\myH}\bigr)^2 \leq 9 \myT 
		\qquad\text{and}\qquad
		\sfD(u) \leq 8 \myT + 8 \myG \leq 80 \myT \,.
	\end{equation}
	We may assume $w\not\equiv0$, otherwise everything vanishes.
	Set $v_\tau:=\sfH_\tau w$ for $\tau>0$, we write $\myT(v_\tau)$, $\myH(v_\tau)$, etc.\@ to denote the quantities $\myT$, $\myH$, etc.\@ in~\eqref{eq:PHRCQ} with $w$ replaced by $v_\tau$. By the estimates~\eqref{eq:HeatEstimates} of
	Section~\ref{subsec:Heat}, $v_\tau\in\ZO$; moreover
	$v_\tau\ge0$ has a continuous representative with
	$\min_{\overline\Omega}v_\tau>0$ by the strict positivity of the
	Neumann heat kernel \cite{Davies89}, so $v_\tau\in\PZO$.
	Since $w\in\HtN$ we have, as $\tau\downarrow0$,
	$\Lap v_\tau=\sfH_\tau(\Lap w)\to\Lap w$ in $L^2(\Omega)$, and by
	\eqref{eq:H2bound} applied to $v_\tau-w\in\HtN$ also
	$\rmD^2v_\tau\to\rmD^2w$ in $L^2$; in particular
	$\myT(v_\tau)\to \myT$, $\myH(v_\tau)\to \myH$, and $v_\tau\to w$,
	$\nabla v_\tau\to\nabla w$ strongly in~$L^2$.

	By \eqref{eq:R-rough} applied to $v_\tau\in\PZO$,
	\begin{equation}
		\label{eq:R-tau}
		\myG(v_\tau)\le
		\bigl(\sqrt{\myT(v_\tau)}+2\sqrt{\myH(v_\tau)}\bigr)^{\!2}.
	\end{equation}
	Consider the integrand
	\begin{equation}
		\label{eq:Ioffe-integrand}
		g(s,\xi):=\begin{cases}
			|\xi|^4/s^2,&s>0,\\
			0,&s=0,\ \xi=0,\\
			+\infty,&s=0,\ \xi\neq0,
		\end{cases}
	\end{equation}
	which is nonnegative, lower semicontinuous on
	$[0,+\infty)\times\R^d$ and convex in $\xi$. Since $v_\tau\to w$
	in $L^2(\Omega)$ (hence in measure) and $\nabla
	v_\tau\to\nabla w$ (weakly) in $L^1(\Omega;\R^d)$, the lower
	semicontinuity theorem of Ioffe
	\cite[Theorem 5.8]{Ambrosio-Fusco-Pallara00} gives
	\begin{equation*}
		\int_\Omega g(w,\nabla w)\,\d x\le
		\liminf_{\tau\down0}\int_\Omega g(v_\tau,\nabla
		v_\tau)\,\d x=
		\liminf_{\tau\down0}\myG(v_\tau) ,
	\end{equation*}
	and the left-hand side coincides with $\myG=\myG(w)$ thanks to
	Lemma~\ref{le:sqrt-calculus}(a). Combining with \eqref{eq:R-tau}
	and the convergence of $\myT(v_\tau),\myH(v_\tau)$ we obtain the left first
	inequality in~\eqref{eq:quartic:PHRCQ}; the left second follows from
	the convexity inequality \eqref{eq:H2bound}, which gives
	$\tnrm{\rmD^2w}\le\tnrm{\Lap w}$, i.e.~$\myH\le \myT$, so that
	$\big(\sqrt{\myT}+2\sqrt{\myH}\big)^2\le 9\myT$.

	Finally, by \eqref{eq:pointwise-sqrt} and the elementary
	inequality $(a+b)^2\le2a^2+2b^2$,
	\begin{equation*}
		\sfD(u)=4\int_{\{w>0\}}\Big(\Lap w+\frac{|\nabla
		w|^2}{w}\Big)^{\!2}\,\d x\le
		8\myT+8\myG ,
	\end{equation*}
	and the right inequality in~\eqref{eq:quartic:PHRCQ} follows from the left, since 
	$\sfD(u)\le 8\myT+8\myG\le 8\myT+72\myT=80\myT$.
\end{proof}

\begin{proof}[Proof of Lemma~\ref{le:relaxedD}]
	Convexity of each $\sfG_\tau$ follows from the joint convexity of
	the integrand $\sfh$ of~\eqref{eq:defh} and the linearity of
	$u\mapsto(\sfH_\tau u,\Lap\sfH_\tau u)$; a supremum of convex
	functionals is convex.

	Concerning lower semicontinuity, fix $\tau>0$ and let
	$u_n\weakto u$ weakly in $L^2(\Omega)$, $u_n\ge0$. Since $\Omega$
	is bounded, the ultracontractive semigroup $\sfH_{\tau/2}$ admits
	a bounded kernel, hence it is a Hilbert--Schmidt (thus compact)
	operator on $L^2(\Omega)$. Writing
	$\sfH_\tau u_n=\sfH_{\tau/2}\big(\sfH_{\tau/2}u_n\big)$ and
	$\Lap\sfH_\tau u_n=\sfH_{\tau/2}\big(\Lap\sfH_{\tau/2}u_n\big)$,
	where $\big(\sfH_{\tau/2}u_n\big)_n$ and
	$\big(\Lap\sfH_{\tau/2}u_n\big)_n$ are bounded in $L^2$ and
	converge weakly to $\sfH_{\tau/2}u$, $\Lap\sfH_{\tau/2}u$
	respectively, we deduce that
	\begin{equation*}
		\sfH_\tau u_n\to\sfH_\tau u,\qquad
		\Lap \sfH_\tau u_n\to\Lap\sfH_\tau u\qquad
		\text{strongly in }L^2(\Omega).
	\end{equation*}
	Along a subsequence realizing the $\liminf$ of
	$\sfG_\tau(u_n)$ we can also assume a.e.~convergence; since $\sfh$
	is nonnegative and lower semicontinuous on $[0,+\infty)\times\R$,
	Fatou's Lemma gives
	$\liminf_n\sfG_\tau(u_n)\ge\sfG_\tau(u)$.
	A supremum of lower semicontinuous functionals is lower
	semicontinuous.
\end{proof}

\begin{proof}[Proof of Theorem~\ref{thm:Hessian}]
	We may assume $u\not\equiv0$ and we split the proof into three
	claims, which we then combine.

	\smallskip
	\emph{Claim 1: $\sfG_\tau(u)\le\sfD(u)$ for every $\tau>0$, so
	that $\sfD^*(u)\le\sfD(u)$.}
	We may assume $\sfD(u)<+\infty$, so that $u\in D_{L1}(\Lap)$
	and, by \eqref{eq:defD},
	$\sfD(u)=\int_\Omega \sfh(u,\Lap u)\,\d x$ with $\sfh$ as in
	\eqref{eq:defh}.
	We apply the Jensen inequality of Lemma~\ref{le:convex-estimate} to the
	vector $\uu:=(u,\Lap u)\in L^1(\Omega;\R^2)$ and to the convex
	l.s.c.\ function $\sfh$: recalling the commutation
	$\Lap\sfH_\tau u=\sfH_\tau(\Lap u)$ of \eqref{eq:L1-closure}, we
	get, $\Leb d$-a.e.~in $\Omega$,
	\begin{equation*}
		\frac{(\Lap\sfH_\tau u)^2}{\sfH_\tau u}
		=\sfh\big(\sfH_\tau u,\Lap \sfH_\tau u\big)
		=\sfh\big(\sfH_\tau\uu\big)
		\le\sfH_\tau\big(\sfh\circ\uu\big) .
	\end{equation*}
	Integrating over $\Omega$ and using that $\sfH_\tau$ is symmetric
	and Markov ($\sfH_\tau1=1$), so that $\int_\Omega \sfH_\tau
	f\,\d x=\int_\Omega f\,\d x$ for every nonnegative measurable
	$f$, we obtain $\sfG_\tau(u)\le\sfD(u)$.

	\smallskip
	\emph{Claim 2: if $u\in D_{L1}(\Lap)$ then
	$\sfD^*(u)\ge\sfD(u)$.}
	As $\tau\down0$ we have $\sfH_\tau u\to u$ and
	$\Lap\sfH_\tau u=\sfH_\tau(\Lap u)\to\Lap u$ in $L^1(\Omega)$ by
	\eqref{eq:L1-closure}; along a suitable sequence $\tau_n\down0$
	both convergences hold $\Leb d$-a.e., and Fatou's Lemma with the
	nonnegative l.s.c.~integrand $\sfh$ yields
	\begin{equation*}
		\liminf_{n\to\infty}\sfG_{\tau_n}(u)\ge
		\int_\Omega \sfh(u,\Lap u)\,\d x=\sfD(u) .
	\end{equation*}

	\smallskip
	\emph{Claim 3: if $\sfD^*(u)<+\infty$ then $w:=\sqrt u\in\HtN$
	and the two bounds \eqref{eq:Hessian} hold.}
	For $\tau>0$ set
	$u_\tau:=\sfH_\tau u\in\PZO$ (as in the proof of
	Proposition~\ref{prop:quartic}) and
	$w_\tau:=\sqrt{u_\tau}\in\PZO$
	(Lemma~\ref{le:stability-app}). Since $u_\tau>0$ in
	$\overline\Omega$,
	\begin{equation*}
		4\,\mathcal J(w_\tau)=\sfD(u_\tau)=
		\int_\Omega\frac{(\Lap \sfH_\tau u)^2}{\sfH_\tau u}\,\d x
		=\sfG_\tau(u)\le\sfD^*(u),
	\end{equation*}
	so that Proposition~\ref{prop:positive-case} gives
	\begin{equation}
		\label{eq:tau-bounds}
		\int_\Omega |\rmD^2 w_\tau|^2 \, \d x
		\le d^2\,\mathcal J(w_\tau)\le
		\frac{d^2}4\,\sfD^*(u),
		\qquad
		\int_\Omega \frac{|\nabla w_\tau|^4}{w_\tau^2} \, \d x
		\le\Big(\frac{16}9\,d^2+\frac23\Big)
		\mathcal J(w_\tau)
		\le\Big(\frac{4d^2}9+\frac16\Big)\sfD^*(u).
	\end{equation}
	Since $(\sqrt a-\sqrt b)^2\le|a-b|$ for $a,b\ge0$,
	\begin{equation*}
		\|w_\tau-w\|_{L^2(\Omega)}^2\le
		\|u_\tau-u\|_{L^1(\Omega)}\le
		|\Omega|^{1/2}\|u_\tau-u\|_{L^2(\Omega)}\to0
		\quad\text{as }\tau\down0 .
	\end{equation*}
	By \eqref{eq:tau-bounds} and the interpolation inequality
	$\tnrm{\nabla v}^2\le\eps\tnrm{\rmD^2v}^2+C_\eps\tnrm v^2$
	(valid on bounded Lipschitz domains), the family $(w_\tau)$ is
	bounded in $H^2(\Omega)$, so that $w_\tau\weakto w$ weakly in
	$H^2(\Omega)$ as $\tau\down0$.
	Since each $w_\tau$ belongs to $\HtN$ and $\HtN$ is a closed
	subspace of $H^2(\Omega)$ (the normal trace $v\mapsto
	\partial_\nn v$ is continuous from $H^2(\Omega)$ to
	$L^2(\partial\Omega)$), it is also weakly closed and $w\in\HtN$,
	i.e.~$u\in\WpO$.
	The first bound in \eqref{eq:Hessian} follows from
	\eqref{eq:tau-bounds} by the weak lower semicontinuity of
	$v\mapsto\tnrm{\rmD^2v}$; the second follows by applying, as in
	the proof of Proposition~\ref{prop:quartic}, Ioffe's Theorem to
	the integrand \eqref{eq:Ioffe-integrand} along the convergences
	$w_\tau\to w$ in $L^2(\Omega)$ and $\nabla w_\tau\weakto\nabla
	w$ in $L^2(\Omega;\R^d)$.

	\smallskip
	We can now conclude the proof of the three statements.

	\emph{(1)} If $u\in D_{L1}(\Lap)$, Claims 1 and 2 give
	$\sfD^*(u)=\sfD(u)$. If $u\notin D_{L1}(\Lap)$ then
	$\sfD(u)=+\infty$ by definition, and also $\sfD^*(u)=+\infty$:
	otherwise Claim 3 would give $\sqrt u\in\HtN$, and
	Lemma~\ref{le:sqrt-calculus}(c) would yield
	$u\in D_{L1}(\Lap)$, a contradiction.

	\emph{(2)} If $\sfD(u)<+\infty$ then $\sfD^*(u)<+\infty$ by
	Claim 1, hence $w=\sqrt u\in\HtN$ and \eqref{eq:Hessian} hold by
	Claim 3. Conversely, if $\sqrt u\in\HtN$, i.e.~$u\in\WpO$, then
	Lemma~\ref{le:sqrt-calculus}(a,c) shows that $u$ complies with
	the vacuum constraint in \eqref{eq:defD-explicit} and
	Proposition~\ref{prop:quartic} gives
	$\sfD(u)\le80\int_\Omega(\Lap\sqrt u)^2\,\d x<+\infty$.

	\emph{(3)} If $\sqrt u\in\HtN$, the first inequality in
	\eqref{eq:twosided} follows from the first bound in
	\eqref{eq:Hessian} combined with (1), and the second is
	\eqref{eq:D-upper}; if $\sqrt u\notin\HtN$ both sides are
	$+\infty$ by (2).
\end{proof}

\section{A general existence result for variational inequalities}
The following existence result for abstract variational inequalities
is tailored to the needs of the present paper.
Results of this type
are classical in the theory of variational inequalities for monotone
operators, we refer
to the monograph~\cite{Baiocchi-Capelo84} and
to~\cite{Baiocchi04}, whose approach is close in spirit to
the one adopted here.
\begin{theorem}
	\label{thm:BC-VI}
	Let $\V$
	be a reflexive Banach space,
	let 
	$\Phi:\V\to [0,+\infty]$ 
	be a convex lower semicontinuous function
	and let 
	$\mathsf f:\dom(\Phi)
	\times \dom(\Phi)\to \R$ 
	satisfying the following properties:
	\begin{enumerate}[\rm ({VI}.1)]
		\item 
	for every $u,v\in \dom(\Phi)$ 
	\begin{equation}
	\mathsf f(u,v)+\mathsf f(v,u)\ge 0.
	\end{equation}
	\item For every 
	$u\in \dom(\Phi)$ 
	the
	map $v\mapsto \mathsf f(u,v)$ 
	is concave, upper semicontinuous and 
	$\mathsf f(u,u)\le 0$. 
	\item For every $v\in \dom(\Phi)$ 
	the map 
	$u\mapsto \mathsf f(u,v)$ 
	is lower semicontinuous along
	segments in $\dom(\Phi)$.
	\item 	
		There exists $v_0\in \dom(\Phi)$ 
		such that the 
		set 
		\begin{equation}
			\bigl\{u\in \dom(\Phi):
			\Phi(u)+
			\mathsf f(u,v_0)\le \Phi(v_0)\bigr\}
		\end{equation}
		is bounded in $\V$.
	\end{enumerate}
	Then there exists an element
	$u\in \dom(\Phi)$ satisfying
	\begin{equation}
		\label{eq:abstract-VI}
		\Phi(u)+\mathsf f(u,v)\le \Phi(v)
		\quad\text{for every }
		\quad v\in \dom(\Phi).
	\end{equation}
\end{theorem}

\section*{Acknowledgments}
G.S.\ has been partially supported by funding from the European
Research Council (ERC) under the European Union's Horizon Europe
research and innovation programme (grant agreement No.\ 101200514,
project acronym OPTiMiSE). Views and opinions expressed are however
those of the author(s) only and do not necessarily reflect those of
the European Union or the European Research Council Executive Agency.
Neither the European Union nor the granting authority can be held
responsible for them.\newline
The research of A.S.\ is partially based upon work from COST Action 24122 mSPACE, supported by COST (European Cooperation in Science and Technology), www.cost.eu.

\printbibliography
\end{document}